\documentclass[11pt]{amsart}

\usepackage{hyperref}
\usepackage{amsmath,amsthm,amssymb,color}
\usepackage{pdfsync}
\usepackage[latin1]{inputenc}
\usepackage{graphicx}
\usepackage{pstricks,epic,eepic}
\usepackage{esint} 
\usepackage{xcolor} 
\usepackage[misc]{ifsym}

\hypersetup{ 
	colorlinks,
	linkcolor={red!50!black},
	citecolor={blue!50!black},
	urlcolor={blue!80!black}
}

\usepackage{pgfplots}
\usetikzlibrary{arrows}

\newtheorem{theorem}{Theorem}[section]
\newtheorem*{theorem*}{Theorem}
\newtheorem*{lemma*}{Lemma}
\newtheorem*{theorema*}{Theorem A}
\newtheorem*{theoremb*}{Theorem B}
\newtheorem*{theoremc*}{Theorem C}
\newtheorem*{mainlemma*}{Main Lemma}

\newtheorem{lemma}[theorem]{Lemma}

\newtheorem{coro}[theorem]{Corollary}
\newtheorem{propo}[theorem]{Proposition}
\newtheorem*{claim}{Claim}
\newtheorem*{conjecture*}{Conjecture}
\theoremstyle{definition}
\newtheorem{definition}[theorem]{\bf Definition}

\theoremstyle{remark}
\newtheorem{remark}[theorem]{\bf Remark}

\numberwithin{equation}{section}

\begin{document}

\title
[Zero set of solutions near the boundary]
{Size of the zero set of solutions of elliptic PDEs near the boundary of Lipschitz domains with small Lipschitz constant}

\author[Josep M. Gallegos]{Josep M. Gallegos$^1$}
\thanks{Supported by the European Research Council (ERC) under the European Union's Horizon 2020 research and innovation programme (grant agreement 101018680), by MICINN (Spain) under the grant PID2020-114167GB-I00, and by the Deutsche Forschungsgemeinschaft (DFG, German Research Foundation) under Germany's Excellence Strategy - EXC-2047/1 - 390685813. \\ \\
$^1$Departament de Matem\`atiques, Universitat Aut\`onoma de Barcelona, Edifici C Facultat de Ci\`encies, 08193 Bellaterra, Spain
\\ \Letter~ jgallegos@mat.uab.cat 
}

\begin{abstract}
	Let $\Omega \subset \mathbb R^d$ be a $C^1$ domain or, more generally, a Lipschitz domain with small Lipschitz constant and $A(x)$ be a $d \times d$ uniformly elliptic, symmetric matrix with Lipschitz coefficients.
	Assume $u$ is harmonic in $\Omega$, or with greater generality $u$ solves $\operatorname{div}(A(x)\nabla u)=0$ in $\Omega$, and $u$ vanishes on $\Sigma = \partial\Omega \cap B$ for some ball $B$.
	We study the \textit{dimension of the singular set} of $u$ in $\Sigma$, in particular we show that there is a countable family of open balls $(B_i)_i$ such that $u|_{B_i \cap \Omega}$ does not change sign and $K \backslash \bigcup_i B_i$ has Minkowski dimension smaller than $d-1-\epsilon$ for any compact $K \subset \Sigma$. We also find upper bounds for the $(d-1)$-dimensional Hausdorff measure of the zero set of $u$ in balls intersecting $\Sigma$ in terms of the frequency.
	
	As a consequence, we prove a new \textit{unique continuation principle at the boundary} for this class of functions 
	and show that the \textit{order of vanishing} at all points of $\Sigma$ is bounded except for a set of Hausdorff dimension at most $d-1-\epsilon$.
\end{abstract}

\maketitle

\section{Introduction}
	In this paper, we study the size of the zero set of solutions $u$ of a certain class of elliptic PDEs (see Section \ref{section: conditions on A(x)}) near the boundary of a Lipschitz domain.
	Assume $\Omega\subset \mathbb R^d$ is a Lipschitz domain with small Lipschitz constant and $\Sigma$ is an open set of the boundary $\partial\Omega$ where $u$ vanishes. 
	We investigate the dimension of the set $\mathcal S_{\Sigma}'(u) = \{x \in \Sigma ~|~ u^{-1}(\{0\}) \cap B(x,r)\cap\Omega \neq \varnothing, ~\forall r>0\}$, the \textit{set where $u$ changes sign in every neighborhood}. 
		
	In a more regular setting, for example in the case $\Omega$ is a $C^{1,\operatorname{Dini}}$ domain (see \cite{AE, DEK, KN, KZ1, KZ2} for the definition) and $u$ is harmonic, $\mathcal S_{\Sigma}'(u)$ coincides with the usual \textit{singular set} at the boundary of $u$: $\mathcal S_{\Sigma}(u) = \{x \in \Sigma ~|~ |\nabla u(x)| = 0\}$ (see Proposition \ref{prop: equality singular set}). Note that all $C^{1,\alpha}$ domains (for any $\alpha>0$) are $C^{1,\operatorname{Dini}}$ and all $C^{1,\operatorname{Dini}}$ domains are $C^1$, but the converse is not true.
	Nonetheless, in the case where $\Omega$ is a Lipschitz domain, $\nabla u(x)$ (or $\partial_\nu u$) only exists in $\Sigma$ in a weaker sense (see Appendix \ref{appendix: nontangential limits}) as far as we know, which anticipates that we will not be able to find fine estimates of the size and dimension of $\mathcal S_{\Sigma}(u)$ (see Section \ref{section: counterexample to hausdorff estimates}). The situation is different for $\mathcal S_{\Sigma}'(u)$, for which we prove the following Minkowski dimension estimate:
	\begin{theorem}
		\label{thm: u is positive besides a set of positive codimension}
		Let $\Omega \subset \mathbb R^d$ be a Lipschitz domain, and let $A(x)$ be a uniformly elliptic symmetric matrix with Lipschitz coefficients defined on $\overline\Omega$. Let $B$ be a ball centered in $\partial\Omega$ and suppose that $\Sigma = B \cap \partial\Omega$ is a Lipschitz graph with slope $\tau<\tau_0$, where $\tau_0$ is some positive constant depending only on $d$ and the ellipticity of $A(x)$.
		Let $u\not\equiv0$ be a solution of $\operatorname{div}(A(x)\nabla u(x)) = 0$ in $\Omega$, continuous in $\overline \Omega$ that vanishes in $\Sigma$. Then there exists some small constant $\epsilon_1(d) > 0$ and a family of open balls $(B_i)_i, ~i \in \mathbb N$ centered on $\Sigma$ such that 
		\begin{enumerate}
			\item $u|_{B_i \cap \Omega}$ is either strictly positive or negative, for all $i \in \mathbb N$,
			\item $K\backslash \bigcup_i B_i$ has Minkowski dimension at most $d-1-\epsilon_1$ for any compact $K \subset \Sigma$.
		\end{enumerate}
		Moreover, in the planar setting ($d=2$), the set $K\backslash\bigcup_i B_i$ is finite for any compact $K\subset \Sigma$.
	\end{theorem}
	Recall that the upper Minkowski dimension of a set $E \subset \mathbb R^{d-1}$ can be defined as 
	\begin{equation}
	\label{eq: Minkowski dimension}
	\operatorname{dim}_{\overline{\mathcal M}} E = \limsup_{j\to\infty} \frac{\log (\#\{\mbox{dyadic cubes $Q$ of side length $2^{-j}$ that satisfy $Q \cap E \neq \varnothing$}\})}{j\log 2}.
	\end{equation}
	{The previous result gives the following corollary:}
	\begin{coro}
		\label{coro:minkowski_dim}
		{Assume $\Omega\subset\mathbb R^d$, $\Sigma$, $A(x)$, $\tau<\tau_0$, and $u\not\equiv0$ as in the statement of Theorem \ref{thm: u is positive besides a set of positive codimension}. Then, there exists a constant $\epsilon_1(d)>0$ such that}
		\[
			\color{black}\operatorname{dim}_{\overline{\mathcal M}} \left ( \mathcal S_\Sigma'(u) \cap K \right) \leq d-1-\epsilon_1
		\] 
		{for any compact set $K\subset \Sigma$.}
		
	\end{coro}
	\begin{remark}
		We remark that Theorem \ref{thm: u is positive besides a set of positive codimension} {and Corollary \ref{coro:minkowski_dim}}
		\begin{itemize}
			\item are new even in the harmonic case {as the set $\mathcal S_\Sigma'(u)$ has not been well studied before (as far as I know)},
			\item are valid for harmonic functions in Riemannian manifolds with Lipschitz boundary (with small Lipschitz constant depending on the metric),
			\item include the case when $\Omega$ is a $C^1$ domain, {situation where not too much is known either},
			\item give Hausdorff dimension estimates for the set $\mathcal S_{\Sigma}'(u)$ by taking an exhaustion of $\Sigma$ by compact sets. {Note that Hausdorff dimension estimates are weaker than Minkowski dimension estimates but these were not known either.}
		\end{itemize}
	\end{remark}
	\vspace{1mm}
	\begin{remark}
		{Some of the results of the present paper suggest that the set $\mathcal S_\Sigma'(u)$ might be a natural substitute of the usual singular set $\mathcal S_\Sigma(u)$ in the case of Lipschitz domains (and rougher). These results are the fact that $\mathcal S_\Sigma(u) = \mathcal S_\Sigma'(u)$ in the case $\Omega$ is a $C^{1,\operatorname{Dini}}$ domain (Proposition \ref{prop: equality singular set}), the existence of an example of a Lipschitz domain where $\operatorname{dim}_{\mathcal H} \mathcal S_\Sigma(u)=d-1$ and no better (see Section \ref{section: counterexample to hausdorff estimates}), and Corollary \ref{coro:minkowski_dim} showing that better dimension estimates are true for $\mathcal S_\Sigma'(u)$. }
	\end{remark}
	Moreover, we show an upper bound estimate on the size of the zero set of $u$ on balls centered at $\Sigma$ in terms of the \textit{frequency function} $N(x,r)$ (see Definition \ref{def: frequency function}):
		\begin{theorem}
		\label{thm: measure of nodal set near the boundary}
		Assume $\Omega$, $\Sigma$, $A(x)$, $\tau<\tau_0$, and $u\not\equiv0$ as in the statement of Theorem \ref{thm: u is positive besides a set of positive codimension}.
		Let $x\in\Sigma$ and $0<r<r_0$ with $ r_0$ depending on $\operatorname{dist}(x, \partial\Omega \backslash \Sigma)$, the Lipschitz constant $L_A$ and the ellipticity $\Lambda_A$ of $A(x)$, and $d$.
		There exists $\tilde x \in \Omega$ such that $\vert x-\tilde x\vert \approx_{\Lambda_A} \operatorname{dist}(\tilde x, \partial\Omega) \approx_{\Lambda_A} r$ and
		\[
		\mathcal H^{d-1}(\{u = 0\} \cap B(x,r) \cap \Omega) \leq C r^{d-1} (N(\tilde x, S r)+1)^\alpha 
		\]
		for some large $S$ depending on $L_A, \Lambda_A$ and $d$, and some $\alpha\geq 1$ depending on $d$. 
	\end{theorem}
	The precise $\tilde x$ appearing in the statement of Theorem \ref{thm: measure of nodal set near the boundary} is the center of a certain dyadic cube related to a Whitney cube decomposition of $\Omega$ but we have freedom in choosing it. For more details see Section \ref{section: whitney structure} and Remark \ref{rmk: x in the thm of measure of nodal set}. This result is analogous to Theorem $2$ in \cite{LMNN} for a more general class of functions but with a worse exponent. 
	Further, we briefly discuss the application of this theorem to the study of the zero set of Dirichlet eigenfunctions of the operator $\operatorname{div}(A\nabla \cdot)$ in $\Omega$ in Section \ref{section: Yau}. See \cite{LMNN} for more background on this result and its applications in the harmonic case.
	
	Let us give some historical background for the results of this paper.
	L. Bers asked the following question. Consider a harmonic function $u$ in the upper half-plane $\mathbb R^d_+$, $C^1$ up to the boundary such that there exists $E \in \partial \mathbb R^d_+ = \mathbb R^{d-1}$ where $u = \vert \nabla u \vert = 0$ on $E$. Does $\operatorname{measure}_{d-1}(E)>0$ imply $u \equiv 0$? This question has positive answer in the plane, thanks to the subharmonicity of $\log |\nabla u|$. But in $\mathbb R^d_+,~ d \geq 3$, there are examples constructed by J. Bourgain and T. Wolff \cite{BW} which give a negative response in general.
	
	A related conjecture by F.-H. Lin \cite{Lin} which is still open is the following:
	\begin{conjecture*}
		Let $\Omega \subset \mathbb R^d$ be a Lipschitz domain and $\Sigma = B\cap \partial\Omega$ for some ball $B$ centered in $\partial\Omega$. Let $u$ be a harmonic function in $\Omega$ and continuous up to the boundary that vanishes in $\Sigma$. If the set where $\partial_\nu u = 0$ in $\Sigma$ has positive surface measure, then $u$ must be identically zero. 
	\end{conjecture*}
	This conjecture was proved in the $C^{1,1}$ case in \cite{Lin}, where it was also shown that $\mathcal S_{\Sigma}(u)$ is a $(d-2)$-dimensional set (see also \cite{BZ} for more quantitative estimates).
	V. Adolfsson, L. Escauriaza and C. Kenig also gave a positive answer to the conjecture in \cite{AEK} in the case $\Omega$ is a convex Lipschitz domain. Their work was followed by I. Kukavica and K. Nystr\"om \cite{KN}, and V. Adolfsson and L. Escauriaza \cite{AE} where the conjecture is solved in the case $\Omega$ is $C^{1, \operatorname{Dini}}$. Moreover, in \cite{AE} it is also proved that $\mathcal S_{\Sigma}(u)$ has Hausdorff dimension at most $d-2$.
	
	In the interior of the domain, the singular and critical sets of $u$ have been extensively studied too. In particular, the recent results from A. Naber and D. Valtorta \cite{NV} are remarkable: they find $(d-2)$-dimensional Minkowski content bounds and prove they are $(d-2)$-rectifiable with the introduction of a new quantitative stratification. Some of their techniques have been adapted to study the Minkowski content of the singular set at the boundary in the convex Lipschitz case by S. McCurdy \cite{Mc} and in the $C^{1, \operatorname{Dini}}$ by C. Kenig and Z. Zhao \cite{KZ1}. {Unfortunately, the methods relying on pointwise monotonicity that have been commonly used to study this problem are no longer useful in general Lipschitz domains, and new ideas are needed in this case.}
	The conjecture of F.-H. Lin saw no further progress until X. Tolsa proved the result for harmonic functions in Lipschitz domains with small Lipschitz constants in \cite{To}. His proof uses the new powerful methods developed by A. Logunov and E. Malinnikova (see \cite{Lo1,Lo2,LM2}) to study zero sets of Laplace eigenfunctions in compact Riemannian manifolds. These techniques are used on the boundary of the domain $\Omega$ with the trade-off of restricting its Lipschitz constant. This idea from \cite{To} has also been successfully used by A. Logunov, E. Malinnikova, N. Nadirashvili and F. Nazarov in \cite{LMNN} to study the zero set of Dirichlet Laplace eigenfunctions in Lipschitz domains with small Lipschitz constant. In \cite{LMNN}, the authors also develop novel methods to control the zero set near the boundary that will be relevant in the present paper.
	
	The main tool used in this paper is \textit{Almgren's frequency function} (see Definition \ref{def: frequency function}), a quantity that controls the doubling properties of the $L^2$ averages of $u$ in spheres. {This function was also used in most of the works mentioned in this introduction}.
	Our proof of Theorem \ref{thm: u is positive besides a set of positive codimension} requires two technical lemmas. One is Lemma \ref{lemma: Key lemma}, an adaptation of Lemma 3.1 - \textit{Key lemma} from \cite{To} and it controls the behavior of the frequency function at points in $\Omega$ near the boundary (see also Lemma 7 from \cite{LMNN}). The other is Lemma \ref{lemma: 2nd hyperplane lemma}, inspired by Lemma 8 - \textit{Second hyperplane lemma: cubes without zeros} from \cite{LMNN} and it controls the size of balls centered at the boundary that contain no zeros of $u$. Both results were originally proved only for harmonic functions, hence we first extend them for solutions of second order linear elliptic PDEs in divergence form.
	Afterwards, we combine both lemmas in a {new combinatorial} argument that controls the size of the zero set of $u$ near the boundary. {We remark that the extension of Lemma 8 from \cite{LMNN} to the elliptic case presents many difficulties and it might be interesting on its own.}

	Theorem \ref{thm: u is positive besides a set of positive codimension} allows us to prove an analogous \textit{unique continuation at the boundary} result to the one in \cite{To} for more general elliptic PDEs in divergence form:
	\begin{coro}
		\label{cor: unique continuation}
		Assume $\Omega$, $\Sigma$, $A(x)$, $\tau<\tau_0$, and $u\not\equiv0$ as in the statement of Theorem \ref{thm: u is positive besides a set of positive codimension}. The set $\{x \in \Sigma ~|~ \partial_\nu u(x) = 0\}$ has $(d-1)$-Hausdorff measure $0$.
	\end{coro}
	Observe that the assumption that $u$ vanishes continuously in $\Sigma$ implies that $\nabla u$ exists $\sigma$-a.e. as a non-tangential limit in $\Sigma$, and $\nabla u = (\partial_\nu u) \nu \in L^2_{\text{loc}}(\sigma)$. Here $\sigma$ stands for the surface measure restricted to $\Sigma$ and $\nu$ is the outer unit normal (see Remark \ref{rmk: existence of normal derivative} and Appendix \ref{appendix: nontangential limits}). 
	
	The proof of Corollary \ref{cor: unique continuation} uses that the elliptic measure $\omega_A$ associated to the elliptic operator $\operatorname{div}(A\nabla \cdot )$ is an $\mathcal A_\infty$ Muckenhoupt weight with respect to $\sigma$ (see Definiton \ref{def: Ainfty weight}) and that, since $\vert u\vert$ is locally comparable to a Green function near most points of $\Sigma$ (thanks to not changing sign in a neighborhood), $\partial_\nu u$ is comparable to $d\omega_A/d\sigma$.
	
	Theorem \ref{thm: u is positive besides a set of positive codimension} also has a second corollary that controls the \textit{vanishing order of the zeros} in the set where $u$ {does not change sign nearby}. 	
	\begin{definition}
		\label{def: vanishing order of zeros}
		The \textit{vanishing order of the zero at a point $x \in \Sigma$} is defined as the supremum of the $\alpha>0$ such that there exist $C_\alpha>0$ finite and $r_0>0$ satisfying
		\[
		\fint_{B(x,r)\cap\Omega} \vert u \vert dy \leq C_\alpha r^\alpha, \quad \mbox{$0<r\leq r_0$}.
		\]
	\end{definition}
	\begin{coro}
		\label{cor: order of zeros}
		Assume $\Omega$, $\Sigma$, $A(x)$, $\tau<\tau_0$, and $u\not\equiv0$ as in the statement of Theorem \ref{thm: u is positive besides a set of positive codimension}.
		There exists some small constant $\epsilon_2>0$ depending on $d$, the Lipschitz constant $\tau$ of $\Sigma$, and the ellipticity $\Lambda_A$ of the matrix $A(x)$ such that for all $x\in\Sigma$ outside a set of Hausdorff dimension $d-1-\epsilon_1$, the vanishing order of $u$ at $x$ is smaller than $1+\epsilon_2$.
		Moreover, for all $x \in \Sigma$, the vanishing order of $u$ at $x$ is greater than $1-\epsilon_2$.
	\end{coro}
	This corollary is proved by comparing $u$ locally (in the neighborhoods where it does not change sign) with the Green function of a certain cone with angular opening related to the Lipschitz constant $\tau$ of $\Sigma$. 
	
	{In Section \ref{section: counterexample to hausdorff estimates} we provide an example showing that Corollary \ref{cor: unique continuation} cannot be improved in the sense of Hausdorff dimension estimates. This contrasts with the higher regularity case ($C^{1,\text{Dini}}$) where the set $\mathcal S_\Sigma(u)$ is known to be $(d-2)$-rectifiable and, a fortiori, has Hausdorff dimension at most $d-2$ (see \cite{KZ1}). Finally, in Section \ref{section: equality singular set in Dini dmn}, we prove the following proposition relating $\mathcal S_\Sigma(u)$ and $\mathcal S_\Sigma'(u)$ in the smooth case.}
	\begin{propo}
	\label{prop: equality singular set} 
	{Let $\Omega \subset \mathbb R^d$ be a $C^{1,\operatorname{Dini}}$ domain, $B$ be a ball centered in $\partial\Omega$, and $\Sigma = B\cap\partial\Omega$. Let $u$ be a harmonic function defined in $\Omega$, continuous in $\overline\Omega$ that vanishes in $\Sigma$.
	Then $\mathcal S_{\Sigma}(u)$ coincides with $\mathcal S_{\Sigma}'(u)$.}
	\end{propo}
	{The proof of the proposition follows from a local expansion of $u$ as the sum of a homogeneous harmonic polynomial and an error term of higher degree proved in \cite{KZ2}.}
	\vspace{4mm}
	
	{\noindent\textbf{Further questions.} 
	We present some open questions related to the previous results:}
	\begin{itemize}
		\item {Is it true that $\operatorname{dim}_{\overline{\mathcal M}} \mathcal S_\Sigma'(u)\cap K$ is at most $d-2$ for any compact $K\subset \Sigma$ ($\epsilon_1 \equiv 1$)? What about the (slightly) easier case where $\Omega$ is a $C^1$ domain?}
		\item {Is the set $\mathcal S_\Sigma'(u)$ $(d-2)$-rectifiable?}
		\item {Do these estimates hold for general Lipschitz domains? For the moment, even the conjecture of Lin is still open.}
		\item {Can there exist points with vanishing order $\infty$ in $\Sigma$ in the Lipschitz case? Thanks to the results of \cite{AE,KN,KZ1,KZ2} we know this cannot happen in the $C^{1,\operatorname{Dini}}$ case.}
	\end{itemize}
	\vspace{3mm}
	
	{\noindent\textbf{Outline of the paper.}} In Sections \ref{section: domain and PDE} and \ref{section: frequency function}, we present some notation, tools and ideas that will be used throughout the paper (often without reference). The main aim of Section \ref{section: behavior of frequency} is the proof of Lemma \ref{lemma: Key lemma}, although in Section \ref{section: whitney structure} we construct a Whitney cube structure to $\Omega$ that will be used during the sections following after. Section \ref{section: balls without zeros} is devoted to Lemma \ref{lemma: 2nd hyperplane lemma}. Both lemmas are then combined in a combinatorial argument in Section \ref{section: pf of thm 1.1} to prove Theorems \ref{thm: u is positive besides a set of positive codimension} and \ref{thm: measure of nodal set near the boundary}. In Section \ref{section: pf of thm 1.1}, we also briefly discuss the application of Theorem \ref{thm: measure of nodal set near the boundary} to the study of the zero set of certain class of eigenfunctions. The rest of the paper is spent on the proofs of Corollary \ref{cor: unique continuation}, Corollary \ref{cor: order of zeros}, the example of a Lipschitz domain and harmonic function $u$ with ``large" $\mathcal S_\Sigma(u)$, and the equality $\mathcal S_{\Sigma}(u) = \mathcal S_{\Sigma}'(u)$ in the $C^{1,\operatorname{Dini}}$ case. This last part does not require the Whitney cube structure or Lemmas \ref{lemma: Key lemma} and \ref{lemma: 2nd hyperplane lemma}. In Appendix \ref{appendix: nontangential limits}, we discuss the existence of the non-tangential limit of $\nabla u$ in $\Sigma$.

	\vspace{3mm}
	\noindent\textbf{Acknowledgments} 
	Part of this work was carried out while the author was visiting the Hausdorff Research Institute for Mathematics in Bonn during the research trimester \textit{Interactions between Geometric measure theory, Singular integrals, and PDE}. The author thanks this institution and its staff for their hospitality.
	The author is also grateful to Xavier Tolsa for his guidance and advice, to Jaume de Dios for some useful discussions, {and to the anonymous referee that helped improve the readability of the paper.}

\section{Lipschitz domains with small constant and some properties of elliptic PDEs}
	\label{section: domain and PDE}
	\noindent\textbf{Notation:} the letters $C, c, c', \tilde c$ are used to denote positive constants that depend on the dimension $d$ and whose values may change on different proofs. The constants $c_H$ and $C_N$ retain their values.
	The notation $A\lesssim B$ is equivalent to $A \leq CB$, and $A \sim B$ is equivalent to $A \lesssim B \lesssim A$.
	Sometimes, we will also use the notation $A(x) = B(x) + O(x)$ to denote that $|A(x) - B(x)| \lesssim |x|$.
	
	In the whole paper, we assume that $\Omega$, $\Sigma$, and $u\not\equiv0$ are as in Theorem \ref{thm: u is positive besides a set of positive codimension}. Moreover, we assume that $\Sigma$ is a Lipschitz graph with Lipschitz constant $\tau$ with respect to the hyperplane $H_0 := \{x_d = 0\}$ and that locally $\Omega$ lies above $\Sigma$. 
	
	\begin{remark}
		Note that a $C^1$ domain is a Lipschitz domain with local Lipschitz constant as small as we need. In particular, Theorem \ref{thm: u is positive besides a set of positive codimension}, Theorem \ref{thm: measure of nodal set near the boundary}, and its corollaries are valid for $C^1$ domains. 
	\end{remark}

	\begin{remark}
		The Lipschitz constant of $\Omega$ is invariant by rescalings. 
		If we consider a more general (anisotropic) scaling given by multiplication by a positive definite symmetric matrix with ellipticity $\tilde \Lambda$, then the Lipschitz constant of $\Sigma$ changes by a factor of at most $\tilde\Lambda^2$.
	\end{remark}

	\subsection{Divergence form elliptic PDEs with Lipschitz coefficients}
	\label{section: conditions on A(x)}
	The function $u$ we study solves $\operatorname{div}(A(x)\nabla u(x)) = 0$ weakly in $\Omega$ where the matrix $A(x)$ satisfies that 
	\begin{itemize}
		\item $A(x)$ is symmetric.
		\item There exists $\Lambda_A>1$ such that $\Lambda_A^{-1} \vert y \vert^2 \leq(A(x)y,y) \leq \Lambda_A\vert y\vert^2$ for all $x\in\overline\Omega, y\in \mathbb R^d$ (\textit{uniform ellipticity}).
		\item There exists a Lipschitz constant $L_A > 0$ such that $|A_{ij}(x)-A_{ij}(y)| \leq L_A |x-y|$ (\textit{Lipschitz coefficients}).
	\end{itemize}
	By standard elliptic PDE theory ({see \cite{GT, FR} for example}), we know that $u \in C^{1,\alpha}$ in any ball with closure inside $\Omega$ for any $0<\alpha <1$. {Also $u \in W^{1,2}(\Omega)$ and $u \in W^{2,2}(\Omega')$ for any compactly embedded subdomain $\Omega' \subset \subset \Omega$.}
	
	We extend the function $u$ by $0$ outside of $\overline \Omega$ so that it is continuous through $\Sigma$. We also extend the matrix $A(x)$ in a way that it preserves the ellipticity $\Lambda_A$ and Lipschitz $L_A$ constants up to a constant factor (the particular extension we choose will not matter). In particular, note that the absolute value of the extended function $|u|$ is a subsolution in balls $B$ such that $B\cap\partial\Omega\subset \Sigma$. This means that
	\[
	\int_B (A \nabla |u|, \nabla\phi) \,dx \leq 0,\quad \mbox{$\forall \phi \in C_c^1{(B)}$}.
	\]
	
	\subsection{Modifying the domain and \texorpdfstring{$A(x)$}{A(x)}}
	\label{section: modifying the domain}
	\begin{remark}
		\label{rmk: small LA and LambdaA}
		Throughout this paper, we will require at different points the constants $\max(\Lambda_A-1,1-\Lambda_A^{-1})$ and $L_A$ to be very small.
		We can obtain this by exploiting the fact that our considerations are local. 
		Indeed, we can cover our initial domain $\Omega$ by a finite family of ``small" domains and prove the results on the introduction on each one separately. 
		
		By ``zooming" on a small domain, we decrease the Lipschitz constant $L_A$ of the matrix.
		By rescaling the domain by multiplication with an adequate matrix, we can force $A(x) = I$ at a particular point $x$. This, in addition to the small Lipschitz constant $L_A$ of the new matrix, implies small ellipticity $\Lambda_A$ (we will show this below). Note, though, that this last operation {changes the Lipschitz constant $\tau$ of $\Sigma$}. For this reason, in the statement of Theorem \ref{thm: u is positive besides a set of positive codimension}, $\tau_0$ depends on the ellipticity of $A(x)$. Intuitively, we use the Lipschitz regularity of $A(x)$ to exploit that, in small scales, $A(x)$ is very close to a constant matrix.
	\end{remark}
	Let us show the effects of ``zooming" and rescaling the domain when the domain is a ball.
	Suppose $u$ solves weakly $\operatorname{div}(A\nabla u)=0$ in a ball $B_R$ for some $R<1$, that is, for all $\phi \in  C_c^1 (B_R)$ we have
	\[
	\int_{B_R} (A(x) \nabla u(x),\nabla \phi(x)) \,dx = 0. 
	\]
	If we ``zoom" in on the origin by considering the function $\tilde u(y) = u(Ry)$ defined on $B_1$, we can show it satisfies
	\[
	\int_{B_1} (\tilde A(y) \nabla \tilde u(y),\nabla \tilde \phi(y)) \,dy = 
	0, \quad \forall \tilde \phi \in C_c^1(B_1),
	\]
	where $\tilde A(y) = A(Ry)$. Notice that $\tilde A(y)$ has improved Lipschitz constant $R \cdot L_A < L_A$.
			
	Now suppose that $A(0) \neq I$ and we want to change the function, the domain and the equation so that ``$A(0) = I$" for the solution of the new equation.

	Consider $\tilde u(x) = u(\tilde Sx)$ defined in $\tilde S^{-1}B_R$ where $\tilde S$ is the symmetric positive definite square root of $A(0)$ and $B_R$ is a ball centered at the origin of radius $R>0$.
	Then, for all $\tilde \phi(x)  \in  C_c^\infty (\tilde S^{-1}B_R)$, we have
	\begin{align*}
	\int_{\tilde S^{-1}B_R} (\tilde S^{-1}A(\tilde Sx)\tilde S^{-1} \nabla \tilde u(x),\nabla \tilde \phi(x)) dx &= \int_{\tilde S^{-1}B_R} (A( \tilde Sx)  \tilde S^{-1}\nabla \tilde u(x),\tilde S^{-1}\nabla \tilde \phi(x)) dx \\ 
	&= \int_{\tilde S^{-1}B_R} (A(\tilde Sx)  \nabla u(\tilde Sx),\nabla \phi(\tilde Sx)) dx \\
	&= (\det \tilde S)^{-1} \int_{B_R} (A(y)  \nabla u(y),\nabla \phi(y)) dy \\
	&= 0
	\end{align*}
	where $\phi(\tilde Sx) = \tilde \phi(x)$.
	This implies that $\tilde u$ is a weak solution of the equation $ \operatorname{div}(A_{\tilde S}(x)\nabla \tilde u )=0$ in $\tilde S^{-1}{B_R}$ where $A_{\tilde S}(x) := \tilde S^{-1}A(\tilde Sx)\tilde S^{-1}$. 
	\vspace{5mm}
	
	\noindent \textbf{Properties of the new matrix $A_{\tilde S}(x)$:}
	\begin{itemize}
		\item Observe that $A_{\tilde S}(0) = \tilde S^{-1}A(0)\tilde S^{-1} = I$.    
				
		\item The coefficients of $A_{\tilde S}(x)$ are also Lipschitz with constant $L_{A_{\tilde S}}$ and the ratio between the Lipschitz constants of $A_{\tilde S}$ and $A$ can be bounded above and below by positive constants depending only on $d$ and $\Lambda_A$.
		
		\item $A_{\tilde S}(x)$ is uniformly elliptic with ellipticity constant bounded by $\min(\Lambda_A^2,L_{A_{\tilde S}} \text{diam}(\tilde S^{-1}B_R) d)$.
		Indeed, the largest and smallest eigenvalues $\lambda_\text{max}$ and $\lambda_\text{min}$ of ${A_{\tilde S}(z)}$ at a point $z\in \tilde S^{-1}B_R$ satisfy 
		\begin{align*}
			\max(\lambda_\text{max} - 1, 1-\lambda_\text{min}) &= \Vert A_{\tilde S}(z) - I \Vert_2 \leq \Vert A_{\tilde S}(z) - I \Vert_F =	
	 		\sqrt{\sum_{i,j} |a_{\tilde S,i,j}(z) - \delta_{i,j}|^2}  \\
			&\leq  L_{A_{\tilde S}} |z| d  
			\leq L_{A_{\tilde S}} \text{diam}(\tilde S^{-1}B_R) d
		\end{align*}
		where $\Vert \cdot \Vert_F$ is the Frobenius norm, $\Vert \cdot \Vert_2$ is the spectral norm, and $\delta_{i,j}$ is the Kronecker delta.
	
	\end{itemize}

\section{Frequency function for solutions of elliptic PDEs in divergence form}
	\label{section: frequency function}

	Let $x \in \Omega \cup \Sigma$ such that $A(x) = I$. For $r>0$ such that $B_r(x) \cap \partial\Omega \subset \Sigma$, we denote $\mu_x(y) := (A(y) (y-x) , y-x)/\vert y-x \vert^2$ and
	\[
	H(x,r) := r^{1-d} \int_{\partial B_r(x)} \mu_x(z) \vert u(z) \vert^2 d\sigma(z)
	\]
	where $d\sigma$ is the surface measure of the ball.
	Note that the quantity $H(x,r)$ is nonnegative.
	To simplify the notation we will assume from now on that $x$ is the origin and we will write $H(r) := H(x,r)$ and $\mu(y) := \mu_x(y)$.

	\begin{remark} 
		\label{rmk: properties mu et al}
		Observe that
		$\Lambda_A^{-1}\leq \mu(y) \leq \Lambda_A$ and, since $A$ has Lipschitz coefficients, we have {$|\mu(y) - 1| \lesssim L_A |y|$}, $\Vert A(y) - I\Vert_2 \lesssim  L_A |y|$, and $|\text{tr}A(y) - d| \lesssim L_A |y|$.
	\end{remark}
	We also denote
	\[
	I(x,r) :=r^{1-d} \int_{B_r(x)\cap \Omega} (A\nabla u, \nabla u) dy.
	\]
	Note that $I(x,r)$ is also nonnegative since $A$ is positive definite. Moreover, $I(x,r)$ is finite thanks to Caccioppoli's inequality supposing $B_r(0)\cap\partial\Omega \subset \Sigma$. As before, we will write $I(r):=I(x,r)$.

	\begin{propo}
		\label{prop:exp(cr)H(r) is nondecreasing}
		Assume $0\in\Omega\cup\Sigma$ and $A(0) = I$. For $r>0$ such that $B_r \cap \partial\Omega \subset \Sigma$, the derivative of $H(r)$ is
		\[
		H'(r) = 2I(r) + O(L_A H(r)).
		\]
		As a consequence there exists some constant $c>0$ such that
		\[
		H'(r) \geq -c L_A H(r) .
		\]
	\end{propo}
	For the rest of this paper, we denote $c_H := c L_A$.
	In particular, $H(r)e^{c_H r}$ is a nondecreasing function of $r$. Note that in the harmonic case, $c_H = 0$ and $H$ is nondecreasing (which is also a corollary of the subharmonicity of $|u|^2$).
	
	Note that the proof of Proposition \ref{prop:exp(cr)H(r) is nondecreasing} is quite simple in the case $B(0,r) \cap \partial\Omega = \varnothing$. But in our setting we require some extra considerations to apply the divergence theorem in balls touching the boundary since $u$ only belongs in $W^{1,2}(B(0,r)\cap\Omega)$. To address this problem, we use Appendix \ref{appendix: nontangential limits} and follow the ideas of \cite{To} in the harmonic case.
	\begin{proof}
		First, using that $u \in W^{1,2}(B(0,r))$, we apply the divergence theorem to obtain
		\[
		H(r) = r^{-d} \int_{B_r }\operatorname{div}(|u|^2 A(x)x)dx.
		\]
		Using Remark \ref{rmk: properties mu et al}, we expand the term inside the integral as
		\begin{align*}
			\operatorname{div}(|u|^2 A(x) x) &= 2u (\nabla u, A(x)x) + |u|^2 \underbrace{\text{tr}(A(x))}_{d + O(L_A |x|)} + |u|^2 \underbrace{\sum_{i,j} (\partial_i a_{ij})x_j}_{O(L_A |x|)} \\
			&=: 2u (\nabla u, A(x)x) + d|u|^2 \mu(x) +  \underbrace{f(x)}_{O(L_A |x|)}|u|^2.
		\end{align*}
		
		We compute (formally) the derivative of $H$:
		\begin{align*}
		H'(r) &= -d r^{-1} H(r) + r^{-d} \int_{\partial B_r} \left (2u (\nabla u, A(x)x) + d|u|^2 \mu(x) + |u|^2 {f(x)} \right)d\sigma(x)\\
		&= r^{-d} \int_{\partial B_r} \left( 2u (\nabla u, A(x)x) + |u|^2 {f(x)} \right)d\sigma(x).
		\end{align*}
		Note that the only problematic term is $ r^{-d} \int_{\partial B_r} 2u (\nabla u, A(x)x) d\sigma(x)$. Observe though, that $\int_{\partial B_r} |\nabla u| d\sigma$ exists and is finite for almost every $r$ since $\int_0^b \int_{\partial B_r} |\nabla u| \,d\sigma\, dr = \int_{B_r} |\nabla u| dx$ is finite.
		Nonetheless, using the divergence theorem, we show that for all $[a,b]\subset [0,\operatorname{dist}(0,\partial\Omega\backslash\Sigma))$:
		\begin{align*}
		\int_{a}^b H'(s) ds &= \int_{a}^b s^{-d}\int_{\partial B_s} \left(2u (\nabla u, A(x)x) + |u|^2 {f(x)} \right)d\sigma(x)\, ds \\
		&= \int_{\mathcal A(0,a,b)} \left(\frac{1}{|x|^d}2u (\nabla u, A(x)x) + \frac{1}{|x|^d}|u|^2 {f(x)} \right)dx \\
		&= \int_{\mathcal A(0,a,b)} \frac{1}{|x|^d}\operatorname{div}(|u|^2A(x)x) dx - \int_{\mathcal A(0,a,b)} \frac{d}{|x|^d}|u|^2 {\mu(x)} \,dx \\
		&= \int_{\mathcal A(0,a,b)} \operatorname{div}\left(|u|^2A(x)\frac{x}{|x|^d}\right) dx \\	
		&= H(b) - H(a).
		\end{align*}
		where $\mathcal A(0,a,b)$ is the annulus $B(0,b)\backslash B(0,a)$ for $0<a<b$. Thus, we have the identity for $H'(r)$ for almost every $r$.
		
		Now we want to prove that
		\[
		r^{-d} \int_{\partial B_r} 2u (\nabla u, A(x)x)\, d\sigma(x) = 2r^{1-d} \int_{B_r} (A\nabla u,\nabla u) dx = 2I(r).
		\]
		We denote $\Sigma_\epsilon = \Sigma + \epsilon e_d$ and $\Omega_\epsilon = \Omega + \epsilon e_d$ for small $\epsilon>0$ where $e_d = (0,\hdots,0,1)$ and we have
		\[
		r^{-d} \int_{\partial B_r\cap\Omega} 2u (\nabla u, A(x)x) d\sigma(x)=\lim_{\epsilon\to0} r^{-d} \int_{\partial (B_r\cap\Omega_\epsilon)} 2u (\nabla u, A(x)x) d\sigma(x)
		\]
		from the fact that $u$ vanishes continuously in $\Sigma$ and $\nabla u$ converges as a non-tangential limit in $L^2_{\operatorname{loc}}(\Sigma)$, as shown in Appendix \ref{appendix: nontangential limits}.
		
		Now, since $u\in W^{2,2}(B_r\cap\Omega_\epsilon)$, we may use divergence's theorem and that $u$ solves $\operatorname{div}(A\nabla u)=0$ to obtain
		\begin{align*}
		r^{-d} \int_{\partial (B_r\cap\Omega_\epsilon)} 2u (\nabla u, A(x)x) d\sigma(x) 
		&= 
		r^{-d} \int_{ B_r\cap\Omega_\epsilon} 2r \operatorname{div}(u A(x) \nabla u) dx \\
		&= 
		r^{-d} \int_{ B_r\cap\Omega_\epsilon} 2r (A(x)\nabla u, \nabla u) dx .
		\end{align*}
		Finally, the limit as $\epsilon\to0$ of the last term exists because $u\in W^{1,2}(B_r)$.
		
		Summing up, we have
		\[
		H'(r) = 2I(r) + r^{-d} \int_{\partial B_r} |u|^2 f(x) d\sigma(x)
		\]
		with the second term being $O(L_A H(r))$.
	\end{proof}

	\begin{definition}
		\label{def: frequency function}
		Assume $x\in\Omega \cup \Sigma$ and $A(x) = I$. For $r>0$ such that $B_r(x) \cap \partial\Omega \subset \Sigma$ we define the\textit{ Almgren's frequency function} as
		\[
		N(x,r) := r\frac{I(x,r)}{H(x,r)}.
		\]
	\end{definition}
	We will assume, as before, that $x$ is the origin and we will write $N(r) := N(x,r)$.

	The following geometric lemma is essential to ensure good behavior for $N(x,r)$ in balls intersecting $\Sigma$.
		\begin{lemma}
		\label{lemma: conditions on boundary term so that it is positive}
		Suppose that $x \in \Omega$ and $A(x) = I$.
		Assuming $B_r(x)\cap\partial\Omega\subset\Sigma$, $r$ is small enough (depending on $\operatorname{dist}(x,\Sigma)$ and the Lipschitz constant $L_A$ of $A(y)$), and the Lipschitz constant $\tau$ of the domain is also small enough (depending on $\operatorname{dist}(x,\Sigma)/r$ and $L_A r$), 
		then
		\[
			(\nu(y), A(y) (y-x)) \geq 0
		\]
		for $\mathcal H^{d-1}$-a.e. $y\in\Sigma\cap B_r(x)$ where $\nu$ is the outer unit normal to $\partial\Omega$.
		In particular, if $x$ is the origin, we have
		\[
		\int_{B_r \cap \partial\Omega} \mu(y)^{-1}(\nu, A y)(A\nu,\nu) |\partial_\nu u|^2 d\sigma(y) \geq 0 .
		\]
	\end{lemma}
	\begin{proof}		
	Without loss of generality, we assume $x$ is the origin, denote $T=\operatorname{dist}(0,\Sigma)$ and {suppose $T\leq r$. Otherwise, $\Sigma \cap B_r = \varnothing$.}
	
	First, note that $\Vert A(y) - I \Vert_2 \leq C L_A r$ for all $y \in B_r \cap \partial\Omega$ by the Lipschitz continuity of $A$.
	Using this, we obtain
	\begin{align*}
	(\nu(y), A(y) y) &= (\nu(y),y) + (\nu(y),(A-I)y)\\
	&\geq (\nu(y),y) - |y|C L_A r \\
	&= |y| \cos(\alpha) - |y|C L_A r
	\end{align*}
	where $\alpha$ is the angle between $y$ and $\nu(y)$.
	
	If the domain were flat ($\tau = 0$ and $\Sigma$ were an open subset of a hyperplane) then all the points would have the same fixed normal vector $\tilde \nu$ and we would have $\cos \alpha \geq \frac T r $. Since the domain is not flat ($\tau \neq 0$), we have that the normal vector $\nu$ makes at most an angle of $\arctan(\tau)$  with $\tilde \nu$. Using this    
	we can bound from above $\alpha$ by $\arccos \frac T r + \arctan \tau$. 
	
	Now, assuming $r$ is small enough so that 
	\[
	\frac T r > CL_A r
	\]
	and $\tau$ is also small enough so that
	\[
	\cos\left ( \arccos \frac T r + \arctan \tau \right) \geq C L_A r,
	\]
	we obtain the desired inequality. 
\end{proof}    

	\begin{definition}
		\label{def: x r admissible}
		We will say that the origin $0\in\Omega\cup\Sigma$ (assuming $A(0)=I$) and a radius $r$ are \textit{admissible} if they satisfy the assumptions of the previous lemma. 
		If $A(0)\neq I$, we will say that $0\in\Omega\cup\Sigma$ and $r$ are \textit{admissible} if $0$ and $\Lambda_A^{1/2} r$ are admissible for the transformed domain $\tilde S^{-1}\Omega$ and the matrix $A_{\tilde S}$ (see Remark \ref{rmk: small LA and LambdaA}).
		We can extend this definition to a point $x\neq0$ by translating the domain.
	\end{definition}
	~
	\begin{remark}~
		\begin{enumerate}

			\item If $ \operatorname{dist}(x,\Sigma) > r$, we have that $B_r \cap \partial\Omega = \varnothing$ and the integral is $0$. In this case we also say that $x$ and $r$ are admissible.
			
			\item If $x$ and $r'$ are admissible, then $x$ and $r$ are admissible for all $0\leq r < r'$.
			
			\item In the case $A(x) \neq I$, we can ensure admissibility if we impose
			\[
			\cos\left ( \arccos \left(\Lambda_A^{-1} \frac {\operatorname{dist}(x,\Sigma)} r\right) + \arctan (\Lambda_A\tau) \right) \geq C L_A \Lambda_A^{1/2} r.
			\]
			The reason is that, in the transformed domain, the Lipschitz constant $\tau$ increases at most by a factor $\Lambda_A$ and the distance from the point to $\Sigma$ decreases at most by a factor $\Lambda_A^{1/2}$.
		\end{enumerate}
	\end{remark}

	\begin{remark}
		\label{rmk: existence of normal derivative}
		Since the outer unit normal vector $\nu$ is only defined $\sigma$-a.e. the derivative $\partial_\nu u$ may not exist everywhere.
		Its existence, the non-tangential convergence, and the fact that $\nabla u = (\nabla u, \nu) \nu$ in $L^2(\sigma)$ are proven in Appendix \ref{appendix: nontangential limits}.
	\end{remark}

	We prove an almost monotonicity property for $N(r)$ which was first observed in \cite{GL} in the interior of the domain.
	\begin{propo}
		\label{prop:exp(Cr)N(r) is nondecreasing}
		Assume $0$ and $r'$ are admissible, $A(0) = I$, and the {ellipticity constant $\Lambda_A$ of $A(x)$ is smaller than $2$}. Then there exists $C_N>0$ depending on $L_A$ such that $e^{r C_N} N(r)$ is nondecreasing in the interval $(0,r')$.
	\end{propo}
	For the proof of this proposition we will only keep track of the relation between the constant $C_N$ and the Lipschitz constant $L_A$. $C_N$ also depends on the ellipticity constant $\Lambda_A$ but, by assuming (for example) $\Lambda_A < 2$, we can omit it. We may do so thanks to Remark \ref{rmk: small LA and LambdaA}.
	Notice also that $C_N \equiv 0$ in the harmonic case.
	
	Our proof is an adaptation of the one in \cite{LM1} but with the inclusion of the case $B_r\cap\partial\Omega \neq \varnothing$. We need special care when $B_r\cap\partial\Omega \neq \varnothing$ as with the proof of Proposition \ref{prop:exp(cr)H(r) is nondecreasing}. Again, we use Appendix \ref{appendix: nontangential limits} and follow the ideas of \cite{To} in the harmonic case to circumvent these problems.
	\begin{proof}
		Fix a compact interval $I \subset (0,r')$. We will show that the derivative of $N(r)$ is positive a.e. $r\in I$.
		
		Since 
		\[
		I(r) = r^{1-d} \int_{0}^r \int_{\partial B_s} (A\nabla u, \nabla u)\, d\sigma ds
		\]
		and $r$ is bounded away from $0$, we have that $I$ is absolutely continuous. Also note that we are only considering the case $0\in\Omega$ (since we ask for admissibility). For this reason, $H$ is of class $C^1$ and bounded away from $0$ and $N$ is also absolutely continuous.
		The derivative is
		\[
		N'(r) = \frac{( I(r)+r I'(r)) H(r) - r I(r) H'(r)}{ H(r)^2}.
		\]
		Let's compute
		\[
		I'(r) = (1-d) r^{-1} I(r) + r^{1-d} \underbrace{\int_{\partial B_r\cap\Omega} (A \nabla u, \nabla u) d\sigma}_{\boxed A}.
		\]	
		The previous identity is true for a.e. $r\in I$.
		
		Let $w(y) := \mu(y)^{-1} A(y)y$ be a vector field in $\Omega$. Observe that $(w(y), y) = \vert y \vert^2$ and that $y/r = \nu$, the normal vector in $\partial B_r$. We can rewrite \boxed A as
		\begin{align*}
		\boxed A &= \frac{1}{r} \int_{\partial B_r\cap\Omega} (A \nabla u, \nabla u) (\nu(y), w(y)) d\sigma(y) \\
		&= \frac{1}{r}\underbrace{ \int_{\partial (B_r\cap\Omega)} (A \nabla u, \nabla u) (\nu(y), w(y)) d\sigma}_{\boxed B} - \frac{1}{r} \underbrace{\int_{ B_r\cap\partial\Omega} (A \nabla u, \nabla u) (\nu(y), w(y)) d\sigma}_{\boxed C}.
		\end{align*}
		To study \boxed C we can use that $B_r\cap\partial\Omega \subset \Sigma$ and, thus, $\nabla u = (\nabla u,\nu)\nu$ on $\Sigma$ (in the sense of non-tangential limits, see Appendix \ref{appendix: nontangential limits}).

		Then, we have that
		\[
		\boxed C = \int_{B_r\cap\partial\Omega} (A \nu,\nu) |\partial_\nu u|^2 (\nu(y), w(y)) d\sigma(y) .
		\]
		{This term is a bit problematic and we will treat it later with the help of Lemma \ref{lemma: conditions on boundary term so that it is positive}.}
		
		Let's use the divergence theorem on \boxed B (we do this in $B_r\cap\Omega_\epsilon$ and then let $\epsilon\to0$, as in the proof of Proposition \ref{prop:exp(cr)H(r) is nondecreasing}), obtaining
		\[
		\boxed B = \int_{B_r\cap\Omega} \text{div}(w(x)(A \nabla u, \nabla u)) dx.
		\]    
		We have that
		\begin{align}\begin{split}
		\label{eq: terms we miss in the proof of monotonicity of N}
		\text{div}(w(x) (A \nabla u, \nabla u )) &=
		\text{div}(w) (A \nabla u, \nabla u) +
		(w, \nabla(A \nabla u, \nabla u)) \\
		&=\text{div}(w) (A \nabla u, \nabla u) +
		2(w, \text{Hess}(u)(A \nabla u)) + (A_{D,w}\nabla u, \nabla u)
		\end{split}\end{align}
		where $A_{D,w} = \{\sum_k (\partial_k a_{ij})w_k\}_{ij}$. Furthermore, $\operatorname{Hess}(u)$ is symmetric and
		\[
		\operatorname{Hess}(u)(w) = \nabla (\nabla u, w) - (Dw)\nabla u.
		\]
		Let's compute the integrals with all these terms in \eqref{eq: terms we miss in the proof of monotonicity of N} one by one. First, we obtain
		\[
		\int_{B_r  \cap \Omega_\epsilon} (\operatorname{Hess}(u)w, A \nabla u) dx=
		\int_{B_r  \cap \Omega_\epsilon} \left((\nabla (\nabla u, w), A \nabla u) - ((Dw)\nabla u, A \nabla u) \right)dx.
		\]
		Using that $u$ satisfies $\operatorname{div}(A\nabla u) = 0$ in $\Omega$, we get
		\begin{align*}
		\int_{B_r  \cap \Omega_\epsilon} (\operatorname{Hess}(u)w, A \nabla u) dx
		&=\int_{B_r  \cap \Omega_\epsilon} \left(\operatorname{div}((\nabla u, w)A\nabla u) - ((Dw)\nabla u, A \nabla u) \right)dx \\
		&= \underbrace{\int_{\partial(B_r  \cap \Omega_\epsilon)} (\nabla u, w) (A\nabla u, \nu) d\sigma}_{\boxed D} - \underbrace{\int_{B_r  \cap \Omega_\epsilon} ((Dw)\nabla u, A \nabla u) dx}_{\boxed E}.
		\end{align*}
		We can rewrite the previous equation using the divergence theorem:
		\begin{align*}
		\boxed D &= \int_{\partial B_r  \cap \Omega_\epsilon} (\nabla u, w) (A\nabla u, \nu) d\sigma + \int_{B_r  \cap \partial\Omega_\epsilon} (\nabla u, w) (A\nabla u, \nu) d\sigma 
		\end{align*}
		Using that $\nabla u|_{\Sigma_\epsilon}$ converges to $(\nabla u,\nu)\nu$ on $\Sigma$ in $L^2_{\operatorname{loc}}(\sigma)$ as $\epsilon \to 0^+$ (see Appendix \ref{appendix: nontangential limits}), we get
		\[
		\lim_{\epsilon\to0} \boxed D = \int_{\partial B_r  \cap \Omega} (\nabla u, w) (A\nabla u, \nu) d\sigma 
		+ \underbrace{\int_{B_r  \cap \partial\Omega} (\nu, w)(A\nu,\nu) |\partial_\nu u|^2 d\sigma}_{\boxed C}
		\]
		where the last term on the right hand side coincides with $\boxed C$ defined previously.

		Remember that $w(x) = \mu(x)^{-1} A(x)x$. Thus, we have the approximations $Dw = I + O(L_A|x|)$, and $\operatorname{div}(w) = d + O(L_A|x|)$ for some $a>0$. 
		We can use these estimates to show that 
		\[
		\lim_{\epsilon\to0}\boxed E = r^{d-1}I(r) + O(L_A r^d I(r)).
		\]
		One of the terms we missed in \eqref{eq: terms we miss in the proof of monotonicity of N} behaves as
		\[
		\int_{B_r \cap \Omega} \text{div}(w)(A \nabla u, \nabla u) dx = d r^{d-1} I(r) + O(L_A r^d I(r))
		\]
		and the other as
		\[
		\int_{B_r \cap \Omega} (A_{D,w} \nabla u, \nabla u) dx \lesssim {L_A} \int_{B_r \cap \Omega} r \vert \nabla u \vert^2 = O(L_A r^d I(r)).
		\]
		Summing everything up,
		\[
		I'(r) = - r^{-1} I(r) +  O(L_A I(r)) + 2 r^{1-d} \int_{\partial B_r \cap \Omega} \mu(y)^{-1}(A \nabla u, \nu)^2 d\sigma + r^{-d} \boxed C
		\]	
		and we obtain
		\begin{align*}
		N'(r)N(r)^{-1} &= (rI'(r)+I(r)) (rI(r))^{-1}- H'(r) (H(r))^{-1} \\
		&=  
		\left (2 r^{1-d} \int_{\partial B_r \cap \Omega} \mu(y)^{-1}(A \nabla u, \nu)^2 d\sigma + r^{-d} \boxed C \right ) I(r)^{-1} 
		- 2 r^{-1}N(r) 
		+ O(L_A) \\
		&=  
		\left (2 r^{1-d} H(r) \int_{\partial B_r \cap \Omega} \mu(y)^{-1}(A \nabla u, \nu)^2 d\sigma + 
		r^{-d} H(r) \boxed C - 
		2I(r)^2\right) (H(r)I(r))^{-1} 
		+ O(L_A).
		\end{align*}
		Since $\boxed C \geq 0$ thanks to the fact that $0$ and $r'$ are admissible, and Lemma \ref{lemma: conditions on boundary term so that it is positive}, we can use Cauchy-Schwarz inequality to show that the whole first term is positive. Indeed, note that from the proof of Proposition \ref{prop:exp(cr)H(r) is nondecreasing}, we have
		\[
		I(r) = r^{1-d} \int_{\partial B_r\cap\Omega} u (\nabla u, A(x)\nu) \,d\sigma(x).
		\]
		Thus, using Cauchy-Schwarz, we obtain
		\[
		I(r)^2 \leq r^{2-2d} \int_{\partial B_r\cap\Omega} \mu |u|^2\,d\sigma \int_{\partial B_r \cap \Omega} \mu(y)^{-1}(A \nabla u, \nu)^2 d\sigma = r^{1-d}H(r)\int_{\partial B_r \cap \Omega} \mu(y)^{-1}(A \nabla u, \nu)^2 d\sigma.
		\]
		Therefore, there exists $ C \geq 0$ such that 
		\[
		N'(r) \geq - C L_A N(r)
		\]
		and we denote $C_N := C L_A$.
	\end{proof}
	
\subsection{Frequency function centered at arbitrary points}
	We have only considered $H(x,r)$ and $N(x,r)$ centered at points $x\in\Omega \cup \Sigma$ where $A(x) = I$.
	We can treat general points by making a change of variables such as the one in Section \ref{section: modifying the domain}.
	
	Assume $A(0) \neq I$. {Let $\tilde S$ be the symmetric positive definite square root of $A(0)$, $\tilde u(x) = u(\tilde Sx)$, and $A_{\tilde S} = \tilde S^{-1}A(\tilde Sx)\tilde S^{-1}$.} For the transformed equation $\operatorname{div}(A_{\tilde S}\nabla \tilde u)=0$ with $A_{\tilde S}(0) = I$, we can compute
	\[
	H(r) = r^{-1-d} \int_{\partial B_r} (A_{\tilde S}(x)x,x) |\tilde u(x)|^2 d\sigma(x).
	\]
	After some computations and a change of variables, we can check that it is equal (in the original domain) to
	\[
	H(r) = (\det \tilde S)^{-1} r^{-d} \int_{\partial(\tilde SB_r) } \vert u(y) \vert^2 |A(0)^{-1} y| (A(y)\nu(y), \nu(y)) d\sigma(y).
	\]
	\begin{remark}
		\label{rmk: H in general points}
		Note that $\det \tilde S, (A(y)\nu(y), \nu(y))$ and $r^{-1}|A(0)^{-1} y|$ can be upper and lower bounded by a constant depending only on the ellipticity constant $\Lambda_A$.
		
		In particular, by assuming $\Lambda_A$ bounded, we have
		\[
		H(r) \approx r^{1-d} \int_{\partial(\tilde SB_r)} \vert u(y) \vert^2 d\sigma(y) .
		\]
	\end{remark}
	On the other hand, we have
	\[
	I(r) = r^{1-d} \int_{B_r} (A_{\tilde S}(x) \nabla \tilde u, \nabla \tilde u) dx
	\]
	which, in the original domain, is equal to
	\[
	I(r) = (\det \tilde S)^{-1} r^{1-d} \int_{\tilde SB_r} (A(y) \nabla u(y), \nabla u(y)) dy \approx r^{1-d} \int_{\tilde SB_r} |\nabla u|^2 dy.
	\]
	This allows us to compute $N(x,r)$ for general points $x$. Beware that $A_{\tilde S}$ may have different Lipschitz and ellipticity constants but this is not a problem since the change can be controlled as discussed in Section \ref{section: modifying the domain}.
	
\subsection{Auxiliar lemmas on the behavior of \texorpdfstring{$H(r)$}{H(r)} and \texorpdfstring{$N(r)$}{N(r)}}
	First, we will present a lemma that controls the growth of $H(r)$ using $N(r)$.
	\begin{lemma}
	\label{lemma: lemma 2.2 of Tolsa}
	Suppose $0\in\Omega\cup\Sigma$, $A(0)=I$, and $\alpha >1$.
	Then
	\[
	\int_{\rho}^{\alpha\rho} \left (2\frac{N(r)}{r} - c_H \right) dr\leq\log\left(\frac{H(\alpha\rho)}{H(\rho)}\right) \leq \int_{\rho}^{\alpha\rho} \left(2\frac{N(r)}{r} + c_H \right)dr.
 	\] 
	Moreover, if $0$ and $\alpha \rho$ are admissible, we have
	\begin{equation*}
		2N(\rho) (\log\alpha) e^{-C_N(\alpha-1) \rho} -c_H(\alpha-1)\rho 
		\leq 
		\log \left (\frac{H(\alpha\rho)}{H(\rho)} \right) 
		\leq 
		2 N(\alpha \rho)(\log \alpha) e^{C_N(\alpha-1) \rho} + c_H(\alpha-1)\rho.
	\end{equation*}
	\end{lemma}
	\begin{proof}
		Using Propositions  \ref{prop:exp(cr)H(r) is nondecreasing} and \ref{prop:exp(Cr)N(r) is nondecreasing} on the interval $[\rho, \alpha\rho]$, we can control
		\[
		H'(r) \leq 2I(r) + c_H H(r) = (2r^{-1}N(r) + c_H)H(r) \leq
		(2r^{-1}N(\alpha \rho)e^{C_N(\alpha\rho-r)} + c_H)H(r).
		\]
		Analogously,
		\[
		H'(r) \geq 2I(r) - c_H H(r) \geq (2r^{-1}N(\rho) e^{C_N(\rho-r)} - c_H)H(r).
		\]
		Now we simply integrate $H'/H$ in the interval $[\rho, \alpha\rho]$.
	\end{proof}
	The next lemma bounds $L^2$ norms in annuli by $H(r)$. We denote an annulus centered at $x$ of outer radius $r_2$ and inner radius $r_1$ by $\mathcal A(x,r_1,r_2):= B(x,r_2)\backslash B(x,r_1)$.
	\begin{lemma}
		\label{lemma: annulus bounds}
		Suppose $0\in\Omega\cup\Sigma$, $r,\delta\geq0$, $B(0,r+\delta)\cap\partial\Omega \subset \Sigma$ and $A(0)=I$. Then we have
		\begin{equation}
			\label{integration over annulus}
			e^{c_Hr}H(r) \frac{1}{d \vert B_1 \vert} 
			\leq
			\fint_{\mathcal A(0, r, r+\delta)} e^{c_H\vert x \vert}\mu(x) \vert u(x) \vert^2 dx  
			\leq  
			e^{c_H(r+\delta)}H(r+\delta) \frac{1}{d \vert B_1 \vert}.
		\end{equation}
	\end{lemma}
	\begin{proof}
	Using polar coordinates, write
	\[
	\int_{\mathcal A(0, r, r+\delta)} e^{c_H \vert x \vert}\mu(x) \vert u(x) \vert^2 dx =
	\int_{r}^{r+\delta} s^{d-1} e^{c_H s} H(s) ds,
	\] 
	and use that $e^{c_H r} H(r)$ is nondecreasing (Proposition \ref{prop:exp(cr)H(r) is nondecreasing}), and $|\mathcal A(0,r,r+\delta)| = |B_1|((r+\delta)^d-r^d)$.
	\end{proof}
	
	In a similar fashion, if $A(x_0)\neq I$ and we assume $\Lambda_A$ bounded we obtain the following result.
	\begin{lemma} 
		\label{lemma: bounding balls by tilde H}
		Suppose $x_0 \in \Omega\cup\Sigma$ and $B(x_0, \Lambda_A^{1/2} r) \cap \partial\Omega \subset \Sigma$. Then we have
		\[
		\fint_{B(x_0,r)} \vert u \vert^2 dx  
		\lesssim e^{c_H \Lambda_A^{1/2} r}  H(x_0, \Lambda_A^{1/2} r).
		\]
	\end{lemma}
	\begin{proof}
		Make a change of variables so that $A(x_0) = I$. The ball $B(x_0, r)$ is sent to an ellipsoide contained in the ball $B(x_0, \Lambda_A^{1/2}r)$. Proceed as in the previous lemma by using that $\mu \approx 1$.
	\end{proof}
	
	The next lemma is a perturbation result for $H(z,r)$: it shows that we can bound $H(0,r)$ by $C H(z,r')$ if $0$ and $z$ are close compared to $r$. Moreover, it does not assume that $A(z) = I$.
	\begin{lemma}
		\label{lemma: bounding H by H at a different point}
		Assume that $0 \in \Omega \cup \Sigma$, $A(0) = I$, $\Lambda_A-1$ is small, $z \in \overline \Omega$ such that $|z| \leq \gamma r$ with $\gamma \in(0,1)$ and $B_{100r}(z)\cap\partial\Omega \subset \Sigma$.
		Then for any $\delta \in (0,10)$, we have 
		\[
		H(0,r) \leq C(\gamma, \delta, d, r, L_A) H(z,\Lambda_A^{1/2}r(1+\gamma+\delta))
		\]
		for some constant $C>0$ depending on $\gamma, \delta, d, r$, and $L_A$.
	\end{lemma}
	We omit the dependence of $C$ on $\Lambda_A$ in this lemma.
	\begin{proof}
	Let $\delta \in (0,10)$, then using Lemma \ref{lemma: annulus bounds}, we get
	\[
	e^{c_Hr} H(0,r) \leq \frac{d \vert B_1 \vert}{\vert \mathcal A(0, r, r(1+\delta))\vert} \int_{\mathcal A(0, r, r(1+\delta))} e^{c_H\vert x \vert}\mu(x) \vert u(x) \vert^2 dx  .
	\]
	
	Let $\tilde S := \sqrt{A(z)}$, $\lambda_\text{min}$ be the minimum eigenvalue of $\tilde S$, $\lambda_{\text{max}}$ be the maximum eigenvalue of $\tilde S$, $\mathcal A_0 := \mathcal A(0, r, r(1+\delta))$, and $\mathcal A_z = \mathcal A(0, \lambda_\text{max}^{-1}(1-\gamma)r, \lambda_\text{min}^{-1}(1+\delta+\gamma)r)$. 
	These two annuli are defined so that $\mathcal{A}_0 \subset \{z\} + \tilde S\mathcal A_z$.
	
	Moreover, we have the following estimates
	\[
	\lambda_\text{min} \geq \max(\Lambda_A^{-1/2}, 1 - O(L_A \gamma r))
	\]
	and
	\[
	\lambda_\text{max} \leq \min(\Lambda_A^{1/2}, 1 + O(L_A \gamma r))
	\]
	which will be useful in the proof of the next lemma.

	Now we can bound
	\begin{align*}
	 \frac{\vert \mathcal A_0\vert}{d \vert B_1 \vert}  e^{c_Hr} H(0,r) &\leq  \int_{\mathcal A_0} e^{c_H\vert x \vert}\mu(x) \vert u(x) \vert^2 dx  \leq
	 (1+O(L_A \gamma r)+O(c_H r(1+\delta)) \int_{\mathcal A_0} \vert u \vert^2 dx \\
	 &\leq (1+O(L_A r(1+\delta))) \int_{\{ z\} + \tilde S\mathcal A_z} \vert u \vert^2 dx 
	\end{align*}	
	where we have used that $\mathcal A_0 \subset \{z\} + \tilde S\mathcal A_z$ and $e^{c_H |x|} = 1+O(L_A|x|)$.

	We make the change of variables $x = \tilde Sy$, $dx = (\det \tilde S) dy = (1+O(L_A \gamma r)) \,dy$ and integrate in polar coordinates to get
	\[
	\int_{\{ z\} + \tilde S\mathcal A_z} \vert u \vert^2 dx  = (1+O(L_A \gamma r))  \int_{\{z\}+\mathcal A_z} \vert u (\tilde Sy) \vert^2 dy \leq
	(1+O(L_A \gamma r))
	\int_{\lambda_{\text{max}}^{-1}(1 - \gamma)r}^{\lambda_{\text{min}}^{-1}(1+\gamma+\delta)r} s^{d-1} H(z, s) ds .
	\]
	Finally, we use that $e^{c_H r} H(r)$ is increasing to see that
	\begin{align*}
	\int_{\lambda_{\text{max}}^{-1}(1 - \gamma)r}^{\lambda_{\text{min}}^{-1}(1+\gamma+\delta)r} s^{d-1} H(z, s) ds \leq 
	&(1+O({L_A r (1+\delta)})) H(z, \lambda_{\text{min}}^{-1}(1+\gamma+\delta)r)  \\
	&\frac{r^d[((\lambda_\text{min}^{-1}(1+\delta))^d - (\lambda_\text{max}^{-1}(1-\gamma))^d]}{d} .
	\end{align*}
	Summing up, we obtain
	\[
	H(0,r) 
	\leq
	(1+O({L_A r (1+\delta)}))
	\frac{[((\lambda_\text{min}^{-1}(1+\gamma+\delta))^d - (\lambda_\text{max}^{-1}(1-\gamma))^d]}{(1+\delta)^d - 1}
	H(z, \lambda_{\text{min}}^{-1}(1+\gamma+\delta)r) . 
	\]
	\end{proof}

	Finally, we prove a perturbation result for the frequency function $N$.
	
	\begin{lemma}
		\label{lemma: perturbation of frequency function with small gamma}
		Let $r>0$ and $z \in \Omega$ with $\vert z \vert \leq \gamma r$ for $\gamma >0$ small enough. Assume $L_A r$ is small enough, $0\in\Omega$, $A(0) = I$, the ellipticity $\Lambda_A$ of $A(x)$ is small enough, and the point $z$ and distance $4r$ are admissible. Then we have the following bound
		\[
		N(0,r) \leq O(\sqrt {L_A r}) + O(\sqrt \gamma) +  N(z, 4r)\left (1+O(\sqrt {L_A r})+O(\sqrt \gamma) \right ).
		\]
		If $\gamma$ merely satisfies $0<\gamma<(\Lambda_A+1)^{-1}$, we obtain
		\[
		N(0,r) \leq C +  CN(z, 4r)
		\]
		for some constant $C>0$.
	\end{lemma}
	\begin{proof}    
		Let $\delta \in (0,1)$ to be chosen later. Using Lemma \ref{lemma: annulus bounds} we get the following upper and lower bounds
		\[
		e^{c_H r} H(0,2r) \leq  \frac{d \vert B_1 \vert}{\vert \mathcal A(0, 2r, r(2+\delta))\vert} \int_{\mathcal A(0, 2r, r(2+\delta))} e^{c_H\vert x \vert}\mu(x) \vert u(x) \vert^2 dx  
		\]
		and 
		\[
		e^{c_H r} H(0,r) \geq  \frac{d \vert B_1 \vert}{\vert \mathcal A(0, r(1-\delta), r)\vert} \int_{\mathcal A(0, r(1-\delta), r)} e^{c_H \vert x \vert}\mu(x) \vert u(x) \vert^2 dx  .
		\]
		By Lemma \ref{lemma: lemma 2.2 of Tolsa}, we have
		\[
		2N(0, r) (\log 2) e^{-C_N r} -c_H r
		\leq
		\log \frac{H(0,2r)}{H(0,r)}.
		\]
		We aim to upper bound this quantity by something of the form $\log(H(z, r_1) /H(z,r_2))$. To do this we will proceed as in Lemma \ref{lemma: bounding H by H at a different point}.
		
		Since $A(0) = I$, we have $\lambda_\text{max}(A(z)) = 1 + O(L_A\gamma r)$ and $\lambda_\text{min} (A(z)) = 1 - O( L_A\gamma r)$ (maximum and minimum eigenvalues of $A(z)$).
		Let $\tilde S_z$ be the positive definite symmetric square root of ${A(z)}$. From now on, we will denote $\lambda_\text{max}(\tilde S_z) =: \lambda_\text{max}$ and $\lambda_\text{min}(\tilde S_z) =: \lambda_\text{min}$, both depending on $z$.

		We want to find an ellipsoidal annulus (with shape given by $\tilde S_z$) centered at $z$ that contains
		\[
		\mathcal A^2_0 := \mathcal A(0, 2r, r(2+\delta)) 
		\]
		and another one that is contained in
		\[
		\mathcal A^1_0 := \mathcal A(0, r(1-\delta), r).
		\]
		If we take an annulus $\mathcal A(0,r_1, r_2)$ and deform it by $\tilde S_z$, we get an ``ellipsoidal annulus" $\tilde S_z \mathcal A(0,r_1, r_2)$ such that
		\[
		\mathcal A(0,\lambda_{\text{max}}r_1, \lambda_{\text{min}}r_2) \subset
		\tilde S_z\mathcal A(0,r_1, r_2) \subset
		\mathcal A(0, \lambda_{\text{min}}r_1, \lambda_{\text{max}}r_2).
		\]
		Using this, we can choose the following annuli 
		\[
		\mathcal A^2_0 \subset \{z\} + \tilde S_z\mathcal A(0, \lambda_\text{max}^{-1}(2-\gamma)r, \lambda_\text{min}^{-1}(2+\delta+\gamma)r) =: \{z\} + \tilde S_z\mathcal \mathcal A^2_z
		\]
		and this other one (recall $|z| \leq \gamma r$)
		\[
		\mathcal A^1_0 \supset \{z\} + \tilde S_z\mathcal A(0, \lambda_\text{min}^{-1} (1-\delta+\gamma)r, \lambda_\text{max}^{-1}(1-\gamma)r) 
		=: \{z\} + \tilde S_z\mathcal A^1_z.
		\]    
		For this last annulus to be well defined, we need
		\[
		\lambda_\text{min}^{-1} (1-\delta+\gamma) <
		\lambda_\text{max}^{-1}(1-\gamma) \iff
		\delta > (1+\gamma) - 
		\left ( \frac{\lambda_\text{max}}{\lambda_\text{min}} \right)^{-1}
		(1 - \gamma).
		\]
		We also require
		\[
		\gamma \leq \frac{1}
		{\frac{\lambda_\text{max}}{\lambda_\text{min}}  + 1} < \frac 1 2 
		\]
		so that $\delta$ can satisfy $\delta \leq 1$.
		
		Now we can proceed in the exact same way as in Lemma \ref{lemma: bounding H by H at a different point} to get
		\[
		H(0,2r) \leq (1+O({L_A r (1+\delta)}))
		\frac{\vert \mathcal A_z^2 \vert}{\vert \mathcal A_0^2 \vert}
		H(z, \lambda_{\text{min}}^{-1}(2+\gamma+\delta)r).
		\]
		
		In an analogous way, we can lower bound
		\[
		H(0,r) \geq (1+O({L_A r (1+\delta)}))
		\frac{\vert \mathcal A_z^1 \vert}{\vert \mathcal A_0^1 \vert} H(z, \lambda_\text{min}^{-1}(1-\delta+\gamma)r).
		\]
		Putting both expressions together
		\begin{equation} \label{eq: log H}
			\begin{split}
				\log \left ( \frac{H(0,2r)}{H(0,r)} \right )\leq 
				(1+O(L_A r(1+\delta))) 
				+\log\left ( \frac{|\mathcal A_z^2| |\mathcal A_0^1| }{|\mathcal A_z^1| |\mathcal A_0^2|} \right) +
				\log\left ( \frac{H(z, \lambda_{\text{min}}^{-1}(2+\gamma+\delta)r)}{  H (z, \lambda_\text{min}^{-1} (1-\delta+\gamma)r)} \right ).
			\end{split}
		\end{equation}

		\item Using Lemma \ref{lemma: lemma 2.2 of Tolsa} again, we can upper bound $\log\frac{H(z,r_1) }{ H(z, r_2)}$ in terms of $N$ as
		\begin{align} \label{eq: log tilde H}
			\log\left ( \frac{ H(z, \lambda_{\text{min}}^{-1}(2+\gamma+\delta)r)}{  H (z, \lambda_\text{min}^{-1} (1-\delta+\gamma)r)} \right )
			\leq~
			&2  N(z, \lambda_{\text{min}}^{-1}(2+\gamma+\delta)r) 
			\log \left (\frac{2+\gamma+\delta}{1-\delta+\gamma}\right )  \\
			&(1+O(L_A r (1+\delta))) +
			O(L_A r (1+\delta)) \nonumber.
		\end{align}
		
		\item Now we need to choose $\delta$. Remember that $\delta$ has to satisfy
		\begin{equation}
			\label{eq: condition on delta}
			1 > \delta > (1+\gamma) - 
			\left ( \frac{\lambda_\text{min}}{\lambda_\text{max}} \right)
			(1 - \gamma)  = 2 \gamma + O(L_A r) (1 - \gamma).
		\end{equation}
		We will choose $\delta$ equal to the geometric mean of the left hand side and the right hand side of inequality \eqref{eq: condition on delta}, that is
		\[
		\delta = \sqrt{(1+\gamma) - 
			\left ( \frac{\lambda_\text{min}}{\lambda_\text{max}} \right)
			(1 - \gamma)}.
		\]
		On the other hand, $\gamma$ has to satisfy
		\[
		0 < \gamma <  \frac{1}{\frac{\lambda_\text{max}}{\lambda_\text{min}}  + 1} 
		= \frac{1}{2+O(L_A r)} = \frac 1 2 - O(L_A r).
		\]
		\textbf{Notation:} From now on, for this proof, we will write $1 - \epsilon := \frac{\lambda_\text{min}}{\lambda_\text{max}}$. Thus $\epsilon = O(L_A r)$.
		
		Notice that
		\[
		\sqrt{\gamma + \epsilon}\leq \delta = \sqrt{2\gamma + \epsilon - \epsilon\gamma} \leq \sqrt{2\gamma + \epsilon } \leq O(\sqrt \gamma) + O(\sqrt {L_A r}) 
		\]
		for $\epsilon$ small enough ($L_A r$ small enough).
		Now let's bound every term that has appeared before on Equations \eqref{eq: log H} and \eqref{eq: log tilde H}.
		First we bound
		\begin{align*}
		\log \left( \frac{2+\gamma + \delta}{1 - \delta+\gamma} \right)
		&=
		\log  \left( {\frac {2+\sqrt { \left( 2-\epsilon \right) \gamma+
					\epsilon}+\gamma}{1-\sqrt { \left( 2-\epsilon \right) \gamma+\epsilon}
				+\gamma}} \right) 
		\leq
		\log  \left( {\frac {2+\sqrt { \left( 2-\epsilon \right) \gamma+
					\epsilon}}{1-\sqrt { \left( 2-\epsilon \right) \gamma+\epsilon}
		}} \right) \\
		&\leq \log \left ( 
		\frac{2 +\sqrt{2\gamma + \epsilon}}
		{1 - \sqrt{2\gamma + \epsilon}}\right )
		= \log(2)+ O(\sqrt{\gamma+\epsilon}) \leq \log(2) + O(\sqrt \gamma) + O(\sqrt {L_A r}).
		\end{align*}
		To bound
		\[
		\log \left ( \frac{|\mathcal A_0^1| |\mathcal A_z^2|}{|\mathcal A_0^2||\mathcal A_z^1|} \right ) 
		=
		\log \left ( 
		\frac{(1-(1-\delta)^d) ((\lambda_\text{min}^{-1}(2+\delta+\gamma))^d-(\lambda_\text{max}^{-1}(2-\gamma))^d)}
		{((2+\delta)^d-2^d) ((\lambda_\text{max}^{-1}(1-\gamma))^d-(\lambda_\text{min}^{-1}(1-\delta+\gamma))^d)}
		\right )
		\]
		we separate it in two terms,
		\[
		\log \boxed A = \log \left ( \frac{1-(1-\delta)^d}
		{((1-\epsilon)(1-\gamma))^d-(1-\delta+\gamma)^d} \right )
		\]
		and
		\[
		\log \boxed B = \log \left ( \frac{(2+\delta+\gamma)^d-((1-\epsilon)(2-\gamma))^d}
		{(2+\delta)^d-2^d} \right ).
		\]
		
		Let's write the first order expansion of the terms in $\delta$:
		\begin{itemize}
			\item $(1-\delta)^d = 1 - d\delta + O(\delta^2)$
			\item $(1-\epsilon)^d = 1 -O(L_A r)$
			\item $(1-\gamma)^d \geq 1 -O(\delta^2)$
			\item $(1-\delta+\gamma)^d =(1-\delta)^d + O(\gamma) = 1-d\delta + O(\delta^2) $
		\end{itemize}
		Now we can use this in the first term
		\begin{align*}
		\boxed A 
		&=
		\frac{ d \delta + O(\delta^2)}
		{(1-O(L_A r)) (1 - O(\gamma)) - (1-d\delta + O(\delta^2))} \\
		&\leq
		\frac{ d \delta + O(\delta^2)}
		{(1-O(L_A r)) (1 - O(\delta^2)) - (1-d\delta + O(\delta^2))} \\ 
		&=
		\frac{d\delta + O(\delta^2)}{d\delta -O(L_A r) - O(\delta^2)} \\
		&= 1 + O(L_A r/\delta) + O(\delta) \\
		&\leq 1 + O(\sqrt {L_A r}) + O(\sqrt \gamma).
		\end{align*}
		
		As for the other term, we proceed similarly
		\begin{align*}
		\boxed B
		&=
		\frac{(2+\delta)^d+O(\gamma)-(1-O(L_A r))(2^d-O(\gamma))}
		{d2^{d-1}\delta + O(\delta^2)}\\
		&=
		\frac{d2^{d-1}\delta + O(\delta^2) +O(\gamma) +O(L_A r)}
		{d2^{d-1}\delta + O(\delta^2)}\\ 
		&= 
		1 + O(\delta) + O(\gamma/\delta) + O(L_A r/\delta) \\
		&\leq 1 + O(\sqrt {L_A r}) + O(\sqrt \gamma).
		\end{align*}        
		Using all the bounds obtained in the last paragraphs together with Equations \eqref{eq: log H} and \eqref{eq: log tilde H}, we get
		\[
		N(0,r) \leq O(\sqrt {L_A r}) + O(\sqrt \gamma) + \tilde N(z, \lambda_{\text{min}}^{-1}(2+\gamma+\delta)r)\left (1+O(\sqrt {L_A r})+O(\sqrt \gamma) \right )
		\]
		for $L_A r$ and $\gamma$ small enough.   
		
		Finally, we can bound
		\[
		\lambda_{\text{min}}^{-1}(2+\gamma+\delta)r \leq 4r
		\]
		and use that $e^{C_N r}N(r)$ is increasing to get a simpler expression.
		
	\begin{remark}
		To prove the second part of the lemma ($\gamma$ not necessarily small), we just need to choose $\delta$ as the arithmetic mean of the left hand side and the right hand side of inequality \eqref{eq: condition on delta}. The rest of the proof is straightforward.
	\end{remark}
	\end{proof}

\section{Behavior of the frequency function on cubes near the boundary}
	\label{section: behavior of frequency}
	The aim of this section is to prove the first technical lemma concerning the behavior of $N$ near the boundary.
	For an analogous proof in the harmonic case, see Section 3 of \cite{To} or Sections 4.1 and 4.2 of \cite{LMNN}.

\subsection{Whitney cube structure on \texorpdfstring{$\Omega$}{Omega}}
	\label{section: whitney structure}
	We will consider the same Whitney cube structure in $\Omega$ as \cite{To}.
	
	Let $H_0$ be the horizontal hyperplane through the origin, and $B_0$ be a ball centered in $\Sigma$ such that $M B_0 \cap \partial\Omega \subset \Sigma$ for some very large $M$. We also assume that $MB_0 \cap \partial\Omega$ is a Lipschitz graph with slope $\tau$ small enough with respect to $H_0$.

	We consider the following Whitney decomposition of $\Omega$: a family $\mathcal W$ of dyadic cubes in $\mathbb R^d$ with disjoint interiors and constants $W>20$ and $D_0 \geq 1$ such that
	\begin{enumerate}
		\item $\bigcup_{Q\in\mathcal W} Q = \Omega$,
		\item $10Q \subset \Omega$, $\forall Q \in \mathcal W$,
		\item $WQ \cap \partial\Omega \neq \varnothing$, $\forall Q \in \mathcal W$,
		\item there are at most $D_0$ cubes $Q' \in \mathcal W$ such that $10Q \cap 10Q' \neq \varnothing$, $\forall Q \in \mathcal W$. Further, for such cubes $Q'$ we have $\frac 1 2 \ell(Q') \leq \ell(Q) \leq 2 \ell(Q')$.
	\end{enumerate}
	We will denote by $\ell(Q)$ the side length of $Q$ and by $x_Q$ the center of the cube $Q$.
	From these properties it is clear that $\operatorname{dist}(Q,\partial\Omega) \approx \ell(Q)$. Also we consider the cubes small enough so that $\text{diam}(Q) < \frac 1 {20} \operatorname{dist}(Q, \partial\Omega)$.
	
	Now we will introduce a ``tree" structure of parents, children and generations to this Whitney cube decomposition.
	
	Let $\Pi$ denote the orthogonal projection on $H_0$ and choose $R_0 \in \mathcal W$ such that $R_0 \subset \frac M 2 B_0$. It will be the root of the tree and we define $\mathcal D^0_\mathcal W (R_0) = \{R_0\}$ (that is the set of cubes of generation $0$ of the rooted tree). To characterize the generations $\mathcal D^k_\mathcal W(R_0)$ for $k\geq 1$, we define first
	\[
	J(R_0) = \{\Pi(Q) : Q \in \mathcal W \text{ such that } \Pi(Q) \subset \Pi(R_0) \text{ and } Q \text{ is {below} } R_0\}.
	\]
	We have that $J(R_0)$ is a family of $d-1$ dimensional dyadic cubes in $H_0$, all of them contained in $\Pi(R_0)$. Let $J_k(R_0) \subset J(R_0)$ be the subfamily of $(d-1)$-dimensional dyadic cubes in $H_0$ with side length equal to $2^{-k} \ell(R_0)$. To each $Q' \in J_k(R_0)$ we assign some $Q\in \mathcal W$ such that $\Pi(Q) = Q'$ and such that $Q$ is {below} $R_0$ (notice that there may be more than one choice for $Q$ but the choice is irrelevant), {see \cite[Lemma B.2]{To}}, and we write $s(Q') = Q$. Then we define
	\[
	\mathcal D^k_\mathcal W(R_0) := \{s(Q') : Q' \in J_k(R_0)\}
	\]
	and 
	\[
	\mathcal D_\mathcal W (R_0) = \bigcup_{k \geq 0} \mathcal D^k_\mathcal W(R_0).
	\]
	Finally, for each $R \in \mathcal{D}_{\mathcal{W}}^{k}\left(R_{0}\right)$
	and $j \geq 1$, we denote
	\[
	\mathcal{D}_{\mathcal{W}}^{j}(R)=\left\{Q \in \mathcal{D}_{\mathcal{W}}^{k+j}\left(R_{0}\right): \Pi(Q) \subset \Pi(R)\right\} .
	\]
	By the properties of the Whitney cubes, we can observe that
	\[
	Q \in \mathcal{D}_{\mathcal{W}}(R_{0}) \Rightarrow \operatorname{dist}(Q, \Sigma)=\operatorname{dist}(Q, \partial \Omega) \approx \ell(Q).
	\]
	Further, for any $Q \in \mathcal W$, we denote its center by $x_Q$, its associated cylinder by
	\[
	\mathcal C(Q) := \Pi^{-1}(\Pi(Q)),
	\]
	and
	the $(d-1)$-dimensional Lebesgue measure on the hyperplane $H_{0}$ by $m_{d-1}$.
	In Appendix B of \cite{To} one can find more details about the construction of this Whitney cube structure and its projections.
	
\subsection{Lemma on the behavior of the frequency in the Whitney tree}~
	Now we can present the first main lemma required in the proofs of Theorem \ref{thm: u is positive besides a set of positive codimension} and Theorem \ref{thm: measure of nodal set near the boundary}. This lemma controls probabilistically the behavior of the frequency function in the tree of Whitney cubes defined in the last section.
	See \cite[Lemma 3.1]{To} for a version of this lemma for harmonic functions. Our proof is very similar. The reader only needs to consider that the properties of the frequency function for elliptic PDEs are slightly worse than those of the frequency function for harmonic functions, and that $A(x)$ is a perturbation of the identity matrix. Also this Lemma should be compared with the (interior) Hyperplane lemma of \cite{Lo1} and \cite[Lemma 7]{LMNN}. Note that, in what follows, we refer to the frequency function $N$ of a solution of $\operatorname{div}(A\nabla u)=0$ in $\Omega$ as in the statement of Theorem \ref{thm: u is positive besides a set of positive codimension}.
	
	\begin{lemma}
		\label{lemma: Key lemma}
		Let $N_0 > 1$ be big enough. There exists some absolute constant $\delta_0 > 0$ such that for all $S \gg 1$ big enough the following holds, assuming also the Lipschitz constant $\tau$ of $\Sigma$ is small enough. Let
		$R$ be a cube in $\mathcal D_\mathcal{W}(R_0)$ with $\ell(R)$ small enough depending on $S$ and $L_A$ that satisfies $N(x_R, S \ell(R)) \geq N_0$. Then, there exists some positive integer $K = K(S)$
		big enough such that if we let
		\[
		\mathcal{G}_{K}(R)=\left\{Q \in \mathcal{D}_{\mathcal{W}}^{K}(R): N(x_{Q}, S \ell(Q)) \leq  \frac{1}{2}   N(x_{R}, S \ell(R))\right\}
		\]
		then:
		\begin{enumerate}
			\item $m_{d-1}\left(\bigcup_{Q \in \mathcal{G}_{K}(R)} \Pi(Q)\right) \geq \delta_{0} m_{d-1}(\Pi(R))$,
			\item for $Q \in \mathcal D_\mathcal{W}^K(R)$, it holds
			\[
			N(x_Q, S\ell(Q)) \leq (1+CS^{-1/2}) N(x_R, S\ell(R)) .
			\]
			Note that (2) does not require $N(x_R, S\ell(R))\geq N_0$.
		\end{enumerate}
	\end{lemma}
	It is important that $\delta_0$ does not depend on $S$. Other constants such as $M$, $K$, and the upper bound on $\tau$ do depend on $S$. Finally, $N_0$ only depends on the dimension $d$.
	
	\begin{remark}
		\label{rmk:star_shaped}
		Fix $S,T>0$ and $R \in \mathcal D_{\mathcal W}(R_0)$ with $\ell(R)$ small enough depending on $L_A$ and $T$. If $x\in\Omega$ satisfies
		\[
		\operatorname{dist}(x,R) \leq T \ell(R) \quad \text{ and }\quad \operatorname{dist}(x, \partial\Omega) \geq T^{-1} \ell(R), 
		\]
		then $x$ and $S \ell(R)$ are admissible, assuming $M \gg T, M \gg S$, and that $\tau$ is small enough. The proof is analogous to the proof of \cite[Remark 3.3]{To}.
	\end{remark}

	An important tool for the proof of the Lemma \ref{lemma: Key lemma} is the following quantitative Cauchy uniqueness theorem.
	
	\begin{theorem}[Quantitative Cauchy uniqueness]
		\label{thm: quantitative cauchy uniq}
		Let $u$ be a solution of $\operatorname{div}(A(x)\nabla u(x))=0$ in the half ball
		\[
		B_{+}=\left\{y=\left(y^{\prime}, y^{\prime \prime}\right) \in \mathbb{R}^{d-1} \times \mathbb{R} \,|\, \left|y^{\prime}\right|^{2}+|y^{\prime \prime }|^2<1,~ y^{\prime \prime}>0\right\} 
		\]
		and suppose $u$ is $C^1$ smooth up to the boundary and $A(x)$ is as discussed in Section \ref{section: conditions on A(x)}. Let 
		\[
		\Gamma := \{ x \in \mathbb R^d \,|\, x_d = 0,~ \vert x \vert \leq 3/4\}
		\]
		Suppose that $\Vert u \Vert_{L^2(B^+)} \leq 1$ and 
		$ \Vert u \Vert_{W^{1,\infty}(\Gamma)} \leq \epsilon$ for some $\epsilon \in (0,1)$.
		Then 
		\[
		\sup_{B((0, \hdots, 0, 1/2) ,1/4)} \vert u\vert \leq C \epsilon^\alpha ~~\text{ and } ~~\sup_{\frac 1 4 B_+} |u| \leq C' \epsilon^{\alpha'}
		\]
		where $C, C', \alpha,$ and $\alpha'$ are positive constants depending only on the ellipticity and the Lipschitz constant of the matrix $A(x)$ and the dimension $d$.
	\end{theorem}
	This result is proved in great generality in \cite[Theorem 1.7]{ARRV}. It will also be useful for the proof of Lemma \ref{lemma: 2nd hyperplane lemma}. Before starting the proof of Lemma \ref{lemma: Key lemma}, we note that we will require both $\Lambda_A-1$ and $L_A$ very small in what follows (see again Remark \ref{rmk: small LA and LambdaA} to see why we can do so).

	\begin{proof}[Proof of Lemma \ref{lemma: Key lemma}]
 		Let $S\gg 1$ and then choose $R \in \mathcal D_\mathcal W(R_0)$ with $\ell(R)$ small enough depending on $S$ and $L_A$. For some $j \gg 1$ independent of $S$ that will be fixed below, consider the hyperplane $L$ parallel to $H_0$  (and above $H_0$) such that 
 		\[
 		\operatorname{dist}(L, \Sigma \cap \mathcal C(R)) = 2^{-j}\ell(R) .
 		\]
 		From now on, we will denote by $J$ the family of cubes from $\mathcal W$ that intersect $L\cap\mathcal C(\frac 1 2 R)$. By our construction of the Whitney cubes, we have $\ell(Q) \approx 2^{-j}\ell(R)$ and $\Pi(Q) \subset \Pi(R)$ for all $Q\in J$. Notice that if $\tau$ is small enough (depeding on $j$), then 
 		\[
 		\operatorname{dist}(x, \Sigma \cap \mathcal C(10R)) \approx 2^{-j} \ell(R) \quad \text{for } x\in L \cap \mathcal C(10R) .
 		\]
 		
 		Denote by Adm$(2W Q)$ the set of points $x\in \Omega \cap 2WQ$ such that the interval $(0, \text{diam}(25WQ))$ is admissible for $x$. We assume that the Lipschitz constant of the domain $\tau$ is small enough so that $3Q \subset \text{Adm}(2WQ)$ (using Remark \ref{rmk:star_shaped}). Then by Lemma \ref{lemma: perturbation of frequency function with small gamma}, for $Q$ small enough
 		\begin{equation}
 			\label{eq: Tolsa's 3.4}
 			\sup _{x \in \text{Adm}(2WQ)}  N(x, \operatorname{diam}(5 W Q)) \leq C_{0} N(x_{Q}, \operatorname{diam}(20 W Q))+C_{0}
 		\end{equation}
 		where $C_0$ is an absolute constant. 
 		Note that $|x-x_Q| \leq \operatorname{diam}(WQ) < \frac{1}{4} \operatorname{diam}(5WQ)$ since $x\in 2WQ$, hence it satisfies the conditions required by Lemma \ref{lemma: perturbation of frequency function with small gamma} (we use that the ellipticity constant is small enough).
 		
 		\begin{claim} {There exists some $Q \in J$ such that}
 		\begin{equation}
 			\label{eq 3.5 in Tolsa}
 			N(x_Q, \operatorname{diam}(20 W Q)) \leq \frac{N(x_R, S\ell(R))}{4C_0}
 		\end{equation}
 		{if $j$ is big enough (but independent of $S$) and we assume that $\tau_0$ is small enough depending on $j$, and also $N_0$ is big enough.}
 		\end{claim}
 		\begin{proof}[Proof of the claim]
 		From now on, we denote $N = N(x_R, S\ell(R))$. Our aim is to prove the claim using Theorem \ref{thm: quantitative cauchy uniq} in a small half-ball centered at $z_R$, the projection of $x_R$ onto the hyperplane $L$. Set 
 		\[
 		B_+ := \{ x\in B(z_R, \ell(R)/4) \,|\, x_d > (z_R)_d \} \subset \Omega.
 		\]
 		Also, let $\tilde z_R := z_R + (0, \hdots, 0, \ell(R)/8) \in B_+$. Note that, rescaling $B_+$, $\tilde z_R$ corresponds to the point $(0,\hdots, 0, 1/2)$ in the statement of Theorem \ref{thm: quantitative cauchy uniq}.
 		
 		We aim for a contradiction, so we assume that $ N(x_Q, \text{diam}(20WQ)) > N/(4C_0)$ for all $Q \in J$, where $N =  N(x_R, S\ell(R))$.
 		For each $Q\in J$, we have
 		\[
 		\sup_{2Q} \vert u \vert^2 {\lesssim}
 		\fint_{B(x_Q, \text{diam}(3Q))} \vert u \vert^2 dx
 		\]
 		by standard properties of solutions of elliptic PDEs. Then, by Lemma \ref{lemma: bounding balls by tilde H}, we obtain
 		\[
 		\fint_{B(x_Q, \text{diam}(3Q))} \vert u \vert^2 dx\lesssim
 		H(x_Q, \text{diam}(20 W Q)) e^{c_H \text{diam}{(20WQ)}}
 		\]
 		where $c_H$ is the constant from Lemma \ref{lemma: lemma 2.2 of Tolsa}.
 		Note that $e^{c_H \text{diam}{(20WQ)}} = 1+O(L_A \ell(Q))$, omiting the dependence on $\Lambda_A$ (which we may assume is very close to $1$).
 		Using Lemma \ref{lemma: lemma 2.2 of Tolsa} we obtain the following bound using the frequency function:
 		\[
 		H(x_Q, \text{diam}(20WQ)) 
 		\leq 
 		H(x_Q, \ell(R)) \left ( \frac{\text{diam}(20WQ)}{\ell(R)} \right )^{2 N(x_Q, \text{diam}(20WQ))(1+O(L_A \ell(Q)))} (1+O(L_A \ell(Q))) .
 		\]
 		Note that for the previous step we need $\tau$ small enough depending on $j$. At other points of the proof we will require $\tau$ small enough but without further reference. 
 		
 		Now we estimate $H(x_Q, \ell(R))$ as follows 
 		\begin{align*}
 		H(x_Q, \ell(R)) &\approx \ell(R)^{1-d}\int_{\partial B({\ell(R)})} \vert u(x_Q + A(x_Q)^{1/2}y) \vert^2 d\sigma(y) \\
 		&\lesssim \frac{(1+O(L_A \ell(R)))}{\vert \mathcal A(x_Q, \ell(R), 2\ell(R))\vert} \int_{\mathcal A(0, \ell(R), 2\ell(R))} \vert u(x_Q + A(x_Q)^{1/2}y) \vert^2 e^{c_H\vert y\vert} dy \\
 		&\lesssim \frac{(1+O(L_A \ell(R)))}{\vert B(\tilde z_R,C_1 \ell(R) )\vert} \int_{B(0,C_1 \ell(R))} \vert u(\tilde z_R + A(\tilde z_R)^{1/2}y) \vert^2 e^{c_H \vert y\vert} dy
 		\end{align*}
 		where we have used Remark \ref{rmk: H in general points}, and that for some fixed $C_1>0$ we have that $A(x_Q)^{1/2} \mathcal A(x_Q, \ell(R), 2\ell(R)) \subset \mathcal A({\tilde z_R})^{1/2}B(\tilde z_R,C_1 \ell(R))$. Finally, using Lemma \ref{lemma: annulus bounds} we can bound 
 		\[
 		H(x_Q, \ell(R)) \lesssim (1+O(L_A \ell(R))) H(\tilde z_R, C_1\ell(R)).
 		\]
 		Moreover, using again Lemma \ref{lemma: lemma 2.2 of Tolsa}, we can further bound as follows
 		\begin{align*}
 		H(\tilde z_R, C_1 \ell(R)) 
 		&\leq 
 		H(\tilde z_R, \ell(R)/16) (16C_1)^{2 N(\tilde z_R, C_1 \ell(R)) (1+O(L_A \ell(R)))} (1+O(L_A \ell(R))) \\
 		&\lesssim \frac{(1+O(L_A \ell(R)))}{\vert B(\tilde z_R, \ell(R)/8) \vert }
 		\int_{B(0, \ell(R)/8)} e^{c_H|y|} \vert u(\tilde z_R + A(\tilde z_R)^{1/2}y) \vert^2 dy \\
 		&\quad
 		(16C_1)^{2 N(\tilde z_R, C_1 \ell(R))(1+O(L_A \ell(R)))} 
 		\end{align*}
 		and recalling that $B(\tilde z_R, \ell(R)/8) \subset B_+$ even after considering the reescaling by $A(\tilde z_R)^{1/2}$, we obtain
 		\[
 		H(\tilde z_R, C_1 \ell(R))
 		\lesssim \frac{(1+O(L_A \ell(R)))}{\vert B_+ \vert }
 		\int_{B_+} \vert u \vert^2 dy \cdot
 		(16C_1)^{2\tilde N(\tilde z_R, C_1 \ell(R)) (1+O(L_A \ell(R)))} .
 		\]
 		Now, using that $L_A \ell(R)$ is very small, we can bound all the terms $O(L_A \ell(R))$ by $1$, for example.
 		Thus, summing up all the computations we have done, we can write
 		\begin{equation}
 			\label{eq 3.6 in Tolsa}
 			\sup_{2Q} \vert u \vert^2 \lesssim \left ( \frac{\text{diam}(20WQ)}{\ell(R)} \right )^{N(x_Q, \text{diam}(20WQ))}
 			(16C_1)^{4 N(\tilde z_R, C_1 \ell(R))}
 			\fint_{B_+} \vert u \vert^2 dy.
 		\end{equation}
 		By Lemma \ref{lemma: perturbation of frequency function with small gamma}, we have
 		\begin{equation}
 			\label{eq 3.7 in Tolsa}
 			 N(\tilde z_R, C_1 \ell(R)) \leq C  N (x_R, S\ell(R)) + C \leq C'N
 		\end{equation}
 		for a suitable constant $C>2C_1$ if $S$ is large enough, and $N_0$ too.
 		
 		Recalling that we assumed for all $Q \in J$ that $N(x_Q, \text{diam}(20WQ)) > \frac{N}{4C_0}$, we get
 		\begin{align*}
 		\sup_{2Q} \vert u \vert^2
 		&\lesssim
 		(16C_1)^{C'N} \left ( \frac{\text{diam}(20WQ)}{\ell(R)} \right )^{N/4C_0} 
 		\fint_{B_+} \vert u \vert^2 dy\\ 
 		&= 2^{-jC''N+C'''N} \fint_{B_+} \vert u\vert^2 dy
 		\end{align*}
 		for some positive constants $C''$ and $C'''$. We have used that $\text{diam}(20 W Q) < \ell(R)$ by choosing $j$ large enough.

		Using interior estimates for solutions of elliptic PDEs
 		\[
 		\sup_{\frac 3 2 Q} \vert \nabla u \vert^2 \lesssim
 		\frac{\sup_{2Q}\vert u \vert^2}{\ell(Q)^2}
 		\lesssim
 		\frac{2^{2j}}{\ell(R)^2}
 		2^{-jC''N+C'''N} 
 		\fint_{B_+} \vert u\vert^2 dy .
 		\]
 		
 		From the last two estimates we deduce that if $j$ is big enough (depending on the absolute constants $C''$ and $C'''$)
 		and $N_0$ (and thus also $N$) is big enough too, then there exists some $c'>0$ such that
 		\[
 		\sup_{\frac 3 2 Q} \left ( \vert u \vert^2 + \ell(R)^2 \vert \nabla u \vert^2 \right )
 		\lesssim
 		2^{-jc'N}
 		\fint_{B_+} \vert u\vert^2 dy .
 		\]
 		Since the cubes $\frac 3 2 Q$ with $Q \in J$ cover the flat part of the boundary of $B_+$, we can apply a rescaled version of Theorem \ref{thm: quantitative cauchy uniq} to $B_+$ to get
 		\begin{equation} \label{eq: after quant Cauchy uniq}
 			\sup_{B(\tilde z_R, \ell(R)/16)} \vert u \vert^2 
 			\lesssim
 			2^{-jc'N\alpha}
 			\fint_{B_+} \vert u\vert^2 dy
 		\end{equation}
 		for some $\alpha > 0$.
 		Observe that we can lower bound
 		\[
 		\sup_{B(\tilde z_R, \ell(R)/16)} \vert u \vert^2 \gtrsim H(\tilde z_R, \ell(R)/16) 
 		\]
 		and upper bound 
 		\[
 		2^{-jc'N\alpha}
 		\fint_{B_+} \vert u \vert^2 dy 
 		\lesssim
 		2^{-jc'N\alpha}
 		H(\tilde z_R, \ell(R)) .
 		\]
 		Using these bounds in \eqref{eq: after quant Cauchy uniq} we obtain
 		\[
 		H(\tilde z_R, \ell(R)/16) \lesssim 
 		2^{-jc'N\alpha}
 		H(\tilde z_R, \ell(R)).
 		\]
 		By Lemma \ref{lemma: lemma 2.2 of Tolsa}, this implies
 		\[
 		2 N(\tilde z_R, \ell(R))(\log 16) (1+O(L_A \ell(R))) + O(L_A \ell(R))
 		\geq
 		\log\left ( \frac{ H(\tilde z_R, \ell(R))}{ H(\tilde z_R, \ell(R)/16)}  \right )
 		\gtrsim c'jN\alpha
 		\]
 		for some fixed $c'>0$. 
 		But for $j$ big enough this contradicts the fact that $ N(\tilde z_R, \ell(R)) \lesssim N$ by \eqref{eq 3.7 in Tolsa}. Observe though that $j$ ``big enough" does not depend on the election of $S$.
 		\end{proof}
 		Now we may introduce the set $\mathcal G_K(R)$. Fix $Q_0 \in J$ such that \eqref{eq 3.5 in Tolsa} holds for $Q_0$. Notice that, by \eqref{eq: Tolsa's 3.4}
 		\begin{equation}
 			\label{eq: 3.8 of Tolsa}
 			\sup _{x \in \text{Adm}(2WQ_0)}  N(x, \operatorname{diam}(5 W Q_0)) \leq C_{0}  N(x_{Q_0}, \operatorname{diam}(20W Q_0))+C_{0} \leq \frac N 2 \cdot \left ( \frac{10}{11} \right )^2
 		\end{equation}
 		since $N \geq N_0$ and we assume $N_0$ big enough. The precise value of the constant $\left ( \frac{10}{11} \right )^2$ is not important. 
 		Finally, we can define
 		\[
 		\mathcal G_K(R) = \{Q \in \mathcal D_{\mathcal W}^{j+k}(R) : \Pi(Q) \subset \Pi(Q_0) \}
 		\]
 		with $k = \lceil\log_2 S\rceil$. Thus, we have $\mathcal G_K(R)\subset \mathcal D^K_\mathcal W(R)$ with $K = j+k$ and it holds $\ell(Q) = 2^{-k}\ell(Q_0)$ for every $Q \in \mathcal G_K(R)$.
 		
 		The property (1) follows from \eqref{eq: 3.8 of Tolsa}. Indeed if $P \in \mathcal G_K(R)$, then taking into account that $x_P \in \text{Adm}(2WQ_0)$ for $\tau$ small enough (depending on $S$) and using Proposition \ref{prop:exp(Cr)N(r) is nondecreasing} we get
 		\[
 		N(x_P, S\ell(P)) \leq
 		\frac{11}{10} N(x_P, \ell(Q_0)) \leq
 		\left(\frac{11}{10}\right)^2
 		N(x_P, W \ell(Q_0)) \leq \frac N 2 
 		\]
 		where we have bounded the terms $1+O(L_A \ell(Q_0))$ by $11/10$.
 		Notice also that
 		\[
 		m_{d-1}\left(\bigcup_{Q \in \mathcal{G}_{K}(R)} \Pi(Q)\right) =
 		l\left(Q_{0}\right)^{d-1} \approx
 		\left(2^{-j} \ell(R)\right)^{d-1}
 		\]
 		and recall that $j$ is independent of $S$. So (1) holds with $\delta_0 \approx 2^{-j(d-1)}$. 
 		
 		The property (2) is a consequence of Lemma \ref{lemma: perturbation of frequency function with small gamma}. Indeed for any $P \in \mathcal D^K_\mathcal W (R)$, since $\vert x_P - x_R \vert \lesssim \ell(R)$, taking $\gamma \approx S^{-1}$ in {the Lemma \ref{lemma: perturbation of frequency function with small gamma}}, we deduce
 		\begin{align*}
 		N(x_P, S\ell(P)) 
 		&\leq (1 + O(L_A S\ell(R))) N(x_P, S\ell(R)/3) \\
 		&\leq (1 + O(L_A S\ell(R))) \left[ O(\sqrt{L_A S \ell(R)}) + O (\sqrt{S^{-1/2}}) \right ] \\
 		& + (1+O(L_A S \ell(R))) N(x_R, S \ell(R)) \left [1 + O(\sqrt{L_A S \ell(R)}) + O (\sqrt{S^{-1/2}})) \right] .
 		\end{align*}
 		Assuming that $\ell(R) \lesssim S^{-2}$ and $N_0$ large enough, we obtain
 		\[
 		N(x_P, S \ell(P)) \leq  (1 + CS^{-1/2})N(x_R, S \ell(R))
 		\]
 		for certain constant $C$.
 	\end{proof}
	
\section{Balls without zeros near the boundary}
	\label{section: balls without zeros}
	In this section we will prove the second main lemma concerning the behavior of $N$ near the boundary. This lemma shows that if we have a ball near the boundary with bounded frequency, then we can find a smaller ball centered at the boundary where $u$ does not change sign. The following lemma should be compared with \cite[Lemma 8]{LMNN} that treats the harmonic case. Note that, in what follows, we refer to the frequency function $N(x,r)$ of a solution of $\operatorname{div}(A\nabla u)=0$ in $\Omega$ as in the statement of Theorem \ref{thm: u is positive besides a set of positive codimension}. Moreover, we consider $\Omega$ with the Whitney structure defined in Section \ref{section: whitney structure}.
	
	\begin{lemma}
		\label{lemma: 2nd hyperplane lemma} For any $N>0$ and $S \gg 1$ large enough there exist positive constants $\tau_0(N,S)$ and $\rho(N,S)$ such that the following statement holds. Suppose the Lipschitz constant $\tau$ of $\Sigma$ is smaller than $\tau_0$ and $Q$ is a cube in $\mathcal D_\mathcal W(R)$ such that $N(x_Q, S \ell(Q)) \leq N$. Then there exists a ball $B$ centered in $\Sigma \cap \mathcal C(Q)$ with radius $\rho\ell(Q)$ such that $u$ does not vanish in $B \cap \Omega$.
	\end{lemma}
		
	First, we will prove a ``toy" version of this lemma on the half ball $B_+$ for harmonic functions. The following lemma is essentially \cite[Lemma 9]{LMNN} but formulated using the frequency function instead of the doubling index (a closely related quantity used in \cite{LMNN}).
	
	\begin{lemma}
		\label{lemma: Toy lemma}
		Let $B$ be the unit ball in $\mathbb R^d$ and let $B_+$ be the half ball, 
		\[
		B_{+}=\left\{y=\left(y', y''\right) \in \mathbb{R}^{d-1} \times \mathbb{R}\, | \,|y'|^{2}+|y''|^2<1, y''>0\right\} .
		\]
		Let $u$ be a function harmonic in $B_{+}$ such that $u \in C(\overline{B_{+}})$, $u=0$ on $\Gamma := \partial{B}_{+} \cap\left\{y^{\prime \prime}=0\right\}$, and
		\[
		\sup_{\frac{1}{4} B_{+}}|u|=1.
		\]
		For any $N>0$ and $0<r_0<1/16$, there exist $\rho=\rho(N, r_0) >0 $ and $c_{0}=c_{0}(N, r_0)>0$ such that if {$N(0,1/2) \leq N$}, then there is $x^{\prime} \in \mathbb{R}^{d-1}$ with $\left|x^{\prime}\right|< r_0$ such that
		\[
		|u(y)| \geq c_{0} y^{\prime \prime}, \quad \text { for any } \quad y=(y^{\prime}, y^{\prime \prime}) \in B((x^{\prime}, 0), \rho) \cap B_{+} .
		\]
		In particular, $u$ does not vanish in $B\left(\left(x^{\prime}, 0\right), \rho\right) \cap B_{+}$.
	\end{lemma}
	The notation $\frac 1 n B_+$ used in the previous statement stands for $\{y \in \mathbb R^d \, |\, n y \in B_+ \}$.

	\begin{proof}
	Let $B_-$ be the reflection of the half-ball $B_+$ with respect to $\Gamma =\{y''=0\}\cap\partial B_+$. Since $u$ vanishes on $\Gamma$, $u$ can be extended to a harmonic function in $B$ by the Schwarz reflection principle. We also denote this extension by $u$.
	
	Using Cauchy estimates we can uniformly bound every partial derivative of $u$ inside $B(0,1/8)$ obtaining
	\[
	\sup_{x \in B(0,1/8)}\vert \nabla u(x) \vert \lesssim \sup_{B(0, 1/4)} \vert u \vert = 1.
	\]
	Let $\delta := \max_{\substack{x' \in \mathbb R^{d-1},\\ \vert x'\vert \leq r_0}} \vert \nabla u(x',0)\vert$.
	Then, Theorem \ref{thm: quantitative cauchy uniq} applied to $r_0 B_+$ implies that
	\[
	\sup_{B(0,r_0/4)}\vert u \vert \leq C \delta^\gamma
	\]
	for some positive $C$ and $\gamma \in (0,1)$.
	Then $\fint_{\partial B(0,r_0/4)} u^2 d\sigma \leq C \delta^{2\gamma}$ and by subharmonicity $\fint_{\partial B(0,1/2)} u^2 d\sigma \geq c\sup_{B(0,1/4)} u^2 = c$.
	
	By the monotonicity of $N$ in the harmonic case, we also have
	\[
	\log_{2/r_0} \frac{\fint_{\partial B(0,1/2)} u^2 d\sigma}{\fint_{\partial B(0,r_0/4)} u^2 d\sigma} = \log_{2/r_0} \frac{H(0,1/2)}{H(0,r_0/4)} \leq N(0,1/2) \leq N.
	\]
	We can conclude that
	\[
	\log_{2/r_0} \frac{c}{C \delta^{2\gamma}} \leq N \implies \delta \geq c' \left ( \frac{2}{r_0}\right)^{-\frac{N}{2\gamma}}.
	\]
	
	Now, choose $x_*' \in \mathbb R^{d-1}, \vert x_*'\vert \leq r_0$ such that $\vert \nabla u(x_*',0) \vert = \delta$. Clearly, at this point, $\vert \nabla u(x_*',0) \vert = \vert \partial_d u(x_*',0) \vert$ (the derivative in the direction normal to $\Gamma$). Without loss of generality, assume that $\partial_d u(x_*',0) = \delta$. 
	Observe that the second derivatives of $u$ are uniformly bounded in $B(0,1/8)$ using Cauchy estimates (as we have done before). Thus, we have  
	\[
	\partial_d u(y) > \delta/2 \quad\mbox{ when } \operatorname{dist}(y,(x_*',0)) < \rho = \min\{c_0 \delta, r_0\}
	\] 
	where $c_0$ only depends on the bound on the second derivatives, and thus is an absolute constant.
	Using this, we finally get
	\[
	u(y) \geq \frac \delta 2 y^{\prime \prime}  \geq c'' \left( \frac{2}{r_0} \right)^{-\frac{N}{2\gamma}} y^{\prime \prime}, \\
	\]
	for $y=\left(y^{\prime}, y^{\prime \prime}\right) \in B\left(\left(x_{*}^{\prime}, 0\right), \rho\right) \cap B_+$. 
	\end{proof}	

	Now we will prove an elliptic extension of the previous lemma. Unfortunately, the proof presents some complications since we do not have an adequate substitute to Schwarz's reflection principle. {To overcome this, we assume that our function $u$ is a solution of $\operatorname{div}(A(x)\nabla u)=0$ where $A(x)$ is a small (Lipschitz) perturbation of the identity matrix, and we show that there is a harmonic function $v$ very close to $u$ in $C^1$ norm for which the previous lemma holds. Then, we obtain that there is a smaller ball where $u$ does not vanish either.} 
	\begin{lemma}
		\label{lemma: toy lemma elliptic version}
	Let $u$ be a solution of $\operatorname{div}(A\nabla u)=0$ in $B_{+}$ such that $u \in C(\overline{{B}_{+}})$, $u=0$ on $\Gamma := \partial{B}_{+} \cap\left\{y^{\prime \prime}=0\right\}$, and
	\[
	\sup_{\frac{1}{4} B_{+}}|u|=1
	\]
	where $\frac{1}{4} B_{+} = \left(\frac 1 4 B\right)_+$.
	For any $N>0$ and $0<r_0<1/32$, there exist $\rho=\rho(N, r_0) >0 $ and $c_{0}=c_{0}(N, r_0)>0$ such that if {$N(0,1/2) \leq N$}, then there is $x^{\prime} \in \mathbb{R}^{d-1}$ with $\left|x^{\prime}\right|< r_0$ such that
	\[
	|u(y)| \geq c_{0} y^{\prime \prime}, \quad \text { for any } \quad y=\left(y^{\prime}, y^{\prime \prime}\right) \in B\left(\left(x^{\prime}, 0\right), \rho\right) \cap B_{+} 
	\]
	assuming that $L_A$ and $\Lambda_A-1$ are small enough depending on $N$ and $r_0$ (where $L_A$ and $\Lambda_A$ are the Lipschitz and ellipticity constants of $A(x)$, respectively).
	In particular, $u$ does not have zeros in $B\left(\left(x^{\prime}, 0\right), \rho\right) \cap B_{+}$.
	\end{lemma}
	\begin{proof} 
	Let $v$ be the harmonic extension of $u|_{\partial (\frac 1 2 B_+)}$ defined in $\frac 1 2B_+$.
	We intend to use Lemma \ref{lemma: Toy lemma} to find a ball $B$ such that $|v(y',y'')| \gtrsim y''$ for $(y',y'')\in B\cap\Omega$. Afterwards, we will see that if $L_A$ and $\Lambda_A-1$ are small enough, then the difference $v-u$ is arbitrarily small in $W^{1,\infty}(B\cap\Omega)$ which will prove the lemma (for a smaller concentric ball).
	
	First, we bound the frequency (associated to $\Delta$) of $v$ as follows 
	\[
	N^v(0,1/2) = \frac 1 2\frac{\int_{\frac 1 2B_+} \vert \nabla v \vert^2 dx}{\int_{\partial\frac 1 2 B_+} v^2 d\sigma} 
	\lesssim
	\frac{\int_{\frac 1 2 B_+} \vert \nabla u \vert^2 dx}{\int_{\partial \frac 1 2 B_+} u^2 d\sigma} 
	\]
	using that $v$ is the minimizer of the Dirichlet energy for the boundary condition $u|_{\partial (B_+/2)}$.
	We also have that
	\[
	\frac 1 2\frac{\int_{\frac 1 2 B_+} \vert \nabla u \vert^2 dx}{\int_{\partial\frac 1 2 B_+} u^2 d\sigma} 
	\approx
	\frac 1 2  \frac{\int_{\frac 1 2B_+} (A(x)\nabla u, \nabla u)  dx}{\int_{\partial \frac 1 2B_+} \mu u^2 d\sigma} =
	N^u(0,1/2),
	\]
	obtaining an upper bound for $N^v(0,1/2)$. Analogously, we may also obtain a lower bound for $N^v(0,1/2)$ in terms of $N^u(0,1/2)$ using that $u$ minimizes a weighted Dirichlet energy.
	
	Consider the function $h = v-u$ defined in $\frac 1 2B_+$. Note that $h$ is a solution of
	\[
	\begin{cases}
		\operatorname{div}(A\nabla h)= \operatorname{div}((A-I)\nabla v),  &\text{in } \frac 1 2B_+,\\ 
		h= 0,  &\text{on } \partial \frac 1 2 B_+.
	\end{cases}
	\]
	
	We aim to bound $h(x)$ and $\nabla h(x)$ in $\frac 1 {16} B_+$ in terms of $N, L_A$, and $\Lambda_A$. 
	Using the Green's function $G_A(x,y)$ for the elliptic operator $\operatorname{div}(A(x)\nabla \cdot)$ in $\frac 1 2B_+$, we can represent $h(x)$ as
	\[
	h(x)  =
	- \int_{\frac 1 2B_+} (\nabla_y G_A(x,y), (A-I)(y) \nabla v(y)) dy
	\]
	for $x \in B_+/16$.
	We split this integral in two parts and take absolute values:
	\begin{align*}
	|h(x)| \leq \int_{\frac 1 8 B_+} |\nabla_y G_A(x,y)| |(A-I)(y)| {|\nabla v(y)|} dy
	+ \int_{\frac 1 2 B_+\backslash \frac 1 8 B_+} |\nabla_y G_A(x,y)| |(A-I)(y)| |\nabla v(y)| dy .
	\end{align*}
	In both integrals we bound $|A-I|$ by $\Lambda_A-1$.
	Also, in the first integral, we are going to bound $|\nabla v(y)|\lesssim2^{cN}$ for some $c>0$ using Cauchy estimates. To this end, we consider $v$ extended to $\frac 1 2 B$ using Schwarz reflection principle, and then, for $y \in \frac 1 8 B$, we obtain 
	\begin{align*}
		|\nabla v(y)|^2 \lesssim \int_{\partial \frac 1 4 B} v^2 d\sigma \approx H^v(0,1/4) \leq H^v(0,1/2)\approx H^u(0,1/2) \lesssim H^u(0,1/4)\, 2^{cN} \lesssim 1 \cdot 2^{cN}
	\end{align*} 
	 where we have also used $H^u(0,1/4) \lesssim \sup_{\frac 1 4 B_+} |u|^2 = 1$. Remember that $H^v(x,r) = r^{1-d} \int_{\partial B_r(x)} \vert u(z) \vert^2 d\sigma(z)$, as $v$ is harmonic.
	
	Then, using that $\nabla_y G_A(x,\cdot)$ has weak $L^{\frac d {d-1}}$ norm bounded by a constant depending only on $\Lambda_A$ and $d$ (see \cite[estimate (1.6)]{GW} together with the symmetry of $A$), we get
	\[
	\int_{\frac 1 8 B_+} |\nabla_y G_A(x,y)|  dy \leq C(d, \Lambda_A).
	\]
	Thus, we can bound the first integral by
	\[
	\int_{\frac 1 8 B_+} |\nabla_y G_A(x,y)| |(A-I)(y)| {|\nabla v(y)|} dy \lesssim C'(d,\Lambda_A) 2^{cN}.
	\]
	
	For the other integral, we use Cauchy-Schwarz to obtain
	\[
	\int_{\frac 1 2 B_+\backslash \frac 1 8 B_+} |\nabla_y G_A(x,y)| |\nabla v(y)| dy 
	\leq 
	\underbrace{\left (\int_{\frac 1 2 B_+\backslash \frac 1 8 B_+} |\nabla_y G_A(x,y)|^2  dy \right )^{1/2}}_{\boxed A} 
	\underbrace{\left ( \int_{\frac 1 2 B_+}  |\nabla v(y)|^2 dy \right)^{1/2}}_{\boxed B}.
	\]
	Invoking \cite[Theorem 3.3]{GW}, we have that  $|\nabla_y G_A(x,y)|\leq C |x-y|^{1-d} \lesssim 1$ for $x \in \frac 1 {16}B_+$ and $y \in \frac 1 2 B_+ \backslash \frac 1 8 B_+$, which allows us to bound \boxed A.
	We estimate \boxed B as follows:
	\[
	\int_{\frac 1 2 B_+}  |\nabla v(y)|^2 dy = 2 N^v(0,1/2) \int_{\partial \frac 1 2 B_+} v^2 d\sigma \lesssim 2 N 2^{cN},
	\]
	where we have used that $H^v(0,1/2) \lesssim 2^{cN}$, as shown before.
	
	Summing up all the previous estimates, we have obtained $|h(x)| \lesssim C(N, \Lambda_A)(\Lambda_A -1)$ for $x\in \frac 1 {16}B_+$.
	Now, \cite[Corollary 8.36]{GT} gives us
	\[
	\Vert h \Vert_{C^{1,\alpha}({\frac{1}{32}B_{+}})} \lesssim \Vert h \Vert_{L^\infty(\frac 1 {16} B_+)} + \Vert (A-I)\nabla v\Vert_{C^{0,\alpha}(\frac 1 {16} B_+)}
	\]
	for some $\alpha \in (0,1)$. 
	Using the product rule for derivatives, we get 
	\[
	\Vert (A-I)\nabla v\Vert_{C^{0,\alpha}(\frac 1 {16} B_+)} 
	\lesssim
	\Vert (A-I)\nabla v\Vert_{C^{0,1}(\frac 1 {16} B_+)} 
	\lesssim
	L_A \Vert \nabla v \Vert_{L^\infty(\frac 1 {16} B_+)} + (\Lambda_A-1) \Vert D^2 v \Vert_{L^\infty(\frac 1 {16} B_+)}.
	\]
	Finally, using interior Cauchy estimates to estimate $\Vert D^2 v \Vert_{L^\infty(\frac 1 {16} B_+)}$ and $\Vert \nabla v \Vert_{L^\infty(\frac 1 {16} B_+)}$ in terms of $N$, we get
	\[
	|\nabla h(x)| \leq C(N) \left (L_A + (\Lambda_A-1) \right ), \quad x \in \frac 1 {32}B_+ .
	\]
	
	To end the proof, we apply Lemma \ref{lemma: Toy lemma} to $v$ to get a ball where 
	\[
	|v(y)| \geq c_{0} y^{\prime \prime}, \quad \text { for any } \quad y=\left(y^{\prime}, y^{\prime \prime}\right) \in B\left(\left(x^{\prime}, 0\right), \rho\right) \cap B_{+}
	\]
	and we make $L_A$ and $\Lambda_A-1$ small so that $\partial_d h(x) \leq c_0/2$ in $\frac 1 {32}B_+$. This implies $|h(y',y'')| \leq \frac{c_0} 2 y''$ and, since $|u| \geq |v|-|h|$, it finishes the proof.
	\end{proof}

	We need to state a last lemma before proving Lemma \ref{lemma: 2nd hyperplane lemma}.
	\begin{lemma}
		\label{lemma: silva-savin}
		Let $\Sigma \subset \mathbb R^d$ be a Lipschitz graph with respect to the hyperplane $H_0$ with Lipschitz constant $\tau<1/2$, assume that $0\in\Sigma$, and let $u$ be a solution of $\operatorname{div}(A\nabla u)=0$ in the domain $\Omega= \{x + r e_d \,|\, x \in \Sigma\cap B(0,1), r \in (0,1)\}$. Then, there exist positive constants $M$ and $\delta$ depending on the Lipschitz constant $L_A$ and the ellipticity constant $\Lambda_A$ of $A(x)$ such that the following statement holds.
		If $u \geq -1$ in $\Omega\cap B(0,1)$, $u\equiv 0$ on $\Sigma$ and $u \geq M$ on $\{x + r e_d \,|\, x \in \Sigma\cap B(0,1), r \in (\delta,1)\}\cap B(0,1)$,
		then $u\geq 0$ in $\Omega \cap B(0,1/2)$.
	\end{lemma}
		The proof of this lemma is Step 1 of the proof of \cite[Theorem 1.1]{DS}. For a simpler proof in the harmonic case (in Lipschitz domains with small Lipschitz constant), see Appendix B in \cite{FR}.

	In the following proof, we will assume again that $L_A$ and $\Lambda_A-1$ are very small (see Remark \ref{rmk: small LA and LambdaA}).
	\begin{proof}[Proof of Lemma \ref{lemma: 2nd hyperplane lemma}]	
		Let $S \gg 1$ and $Q$ be a cube of our Whitney cube structure such that $x_Q$  (the center of $Q$) and $S \ell(Q)$ are admissible and $N(x_Q, S \ell(Q)) \leq N$. Note that, to attain this, we need the Lipschitz constant of the domain $\tau$ small enough depending on $S$.
		Further, we assume that $S \ell(Q) = 8$ by rescaling the domain and the Whitney cube structure.
		This rescaling changes the Lipschitz constant $L_A$ of the matrix $A(x)$ corresponding to the elliptic operator. But if the cube $Q$ is small enough, the rescaling improves it, that is, makes $L_A$ smaller.

		Let $\tilde x$ be the projection (in the direction $e_d$) of $x$ on $\Sigma$. Then, if $S$ is big enough, we have
		\begin{equation}
			\label{eq: doubling index for H}
			\log \frac{H(\tilde x,S\ell(Q)/2^{k})}{H(\tilde x,S\ell(Q)/2^{k+1})} \leq \log c + \log \frac{H(x_Q,S\ell(Q)/2^{k-1})}{H(x_Q,S\ell(Q)/2^{k+2})} \leq \log c + N c' , ~ \forall k\in \{1,\hdots,5\}.
		\end{equation}
		by Lemma \ref{lemma: lemma 2.2 of Tolsa} and Lemma \ref{lemma: bounding H by H at a different point} (since we do not assume $A(x_Q)=I$ or $A(\tilde x)=I$, we use first Remark \ref{rmk: H in general points}).
		We need $S$ large enough so that $B(\tilde x, \Lambda_A S\ell(Q)/2^k) \subset B(x_Q, S\ell(Q)/2^{k-1})$ and $B(\tilde x, S\ell(Q)/2^{k+1}) \supset B(x_Q, \Lambda_A S\ell(Q)/2^{k+2})$ (we are using that $\Lambda_A$ is very close to $1$) for $k=1,\hdots,5$. From now on, $S$ is fixed.

		We also fix the following normalization for $u$:
		\begin{equation}
			\label{eq: suprem de u}
			\sup_{B(\tilde x,3)\cap\Omega} \vert u \vert = 1 
		\end{equation}
		(remember that $S\ell(Q) = 8$).
		
		Let $x_1 = \tilde x - 3 \tau e_d$.
		Observe that
		\[
		\Gamma_0 = \{ x = (x', x'') \in B(x_1, 2) : x'' = x_1'' \}
		\]
		doesn't intersect $\overline\Omega$. Also let $B_1:=B(x_1, 1), B_2 :=B(x_1, 2)$ and $B_{k,+}$ the upper half of $B_k$, $k=1,2$ (the half of $B_k$ that intersects $\Omega$). 
		
		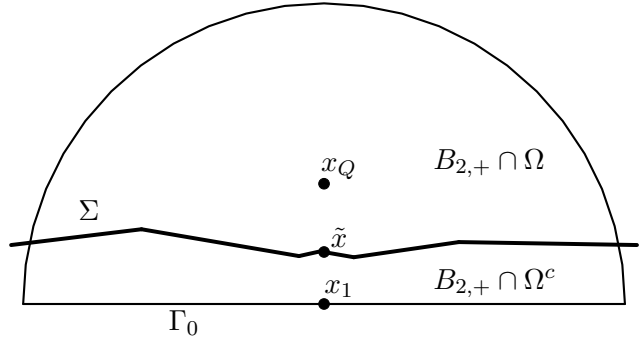
\begin{figure}[h]
			\centering
			\label{figure: lemma 5.1}
			\begin{tikzpicture}[line cap=round,line join=round,>=triangle 45,x=4cm,y=4cm]
				\draw [line width=0.8pt] (-1,0)-- (1,0); 
				\draw [line width=1.5pt] (-1.04,0.19611613513818385)-- (-0.606648968671746,0.24859672580237538);
				\draw [line width=1.5pt] (-0.606648968671746,0.24859672580237538)-- (-0.08316688419400843,0.15941088918764998);
				\draw [line width=1.5pt] (-0.08316688419400843,0.15941088918764998)-- (-0.013369272930310087,0.17492146946847179);
				\draw [line width=1.5pt] (-0.013369272930310087,0.17492146946847179)-- (0.09908243410564835,0.1555332441174445);
				\draw [line width=1.5pt] (0.09908243410564835,0.1555332441174445)-- (0.44807049042414004,0.2059426300301154);
				\draw [line width=1.5pt] (0.44807049042414004,0.2059426300301154)-- (1.04,0.19611613513818404);
				
				\draw(0,0.4) circle[radius = 2pt];
				\fill(0,0.4) circle[radius = 2pt];
				\draw(0.05,0.45) node[]{$x_Q$}; 
				\draw(0,0.173) circle[radius = 2pt];
				\fill(0,0.173) circle[radius = 2pt]; \draw(0.05,0.223) node[]{$\tilde x$};
				\draw(0,0) circle[radius = 2pt];
				\fill(0,0) circle[radius = 2pt]; \draw(0.05,0.05) node[]{$x_1$};
				
				\draw (-0.85,0.38) node[anchor=north west] {$\Sigma$};
				\draw (0.33,0.55) node[anchor=north west] {$B_{2,+}\cap\Omega$};
				\draw (0.33,0.15) node[anchor=north west] {$B_{2,+}\cap\Omega^c$};
				\draw (-0.55,0.012) node[anchor=north west] {$\Gamma_0$};
				\draw [shift={(0,0)},line width=0.8pt]  plot[domain=0:3.141592653589793,variable=\t]({1*1*cos(\t r)+0*1*sin(\t r)},{0*1*cos(\t r)+1*1*sin(\t r)}); 
			\end{tikzpicture}
			\caption{The domain $\Omega$, the Lipschitz graph $\Sigma$, and the half ball $B_{2,+}$.}
		\end{figure}
		Let $g_0$ be the solution of $\operatorname{div}(A\nabla g_0)=0$ on $B_{2,+}$ such that $g_0 \equiv 1$ on $\partial B_{2,+}\backslash \Gamma_0$ and $g_0 \equiv 0$ on $\Gamma_0$. By the maximum principle $g_0 \geq 0$ on $B_{2,+}$ and $g_0 \geq \vert u \vert$ on $\Omega \cap B_2 \subset B_{2,+}$ because of the normalization \eqref{eq: suprem de u} of $u$. 
		Notice that, moreover, we have the bound $$g_0(x) \leq C_1(x''-x_1'')$$ for $x\in B_{1,+}$ which gives us a bound for $\vert u \vert$ in $\Omega \cap B_1$. 
		In the case $A \equiv I$, this follows from reflection and interior Cauchy estimates for $\nabla u$. In the general case, we may use that $g_0(x)$ is comparable to the Green function $G_A(x, y)$ of the domain $B_{2,+}$ (with pole $y=(0,\hdots,0,1.5)$ for example) by the boundary Harnack inequality  inside $B_{1,+}$. 
		By \cite[Theorem 3.3]{GW}, since $B_{2,+}$ satisfies an exterior sphere condition, we have that $G(x,y) \lesssim \operatorname{dist}(x,\partial B_{2,+}) |x-y|^{1-d} \approx \operatorname{dist}(x,\partial B_{2,+})$.
		Further, for $x \in B_{1,+}$ we have that $\operatorname{dist}(x,\partial B_{2,+}) = x'' - x_1''$ which gives us the desired bound.

		Let $g$ be the solution of $\operatorname{div}(A\nabla g)=0$ in $B_{1,+}$ with Dirichlet boundary conditions $g \equiv u$ on $\partial B_{1,+} \cap \Omega$ and $g \equiv 0$ on $\partial B_{1,+} \backslash \Omega$.
		Since $|u| \leq g_0$ in $B_{2,+}\cap\Omega$, we have $\vert g \vert \leq g_0 \leq C_1(x''-x_1'')$. Also, 
		\begin{equation}		
		\label{eq: g and u lemma zeros}
			|g-u| = 
			\begin{cases}
				0, \quad &\text{on } \Omega \cap\partial B_{1,+}, \\
				|g| \leq 4C_1 \tau, \quad &\text{on } \partial\Omega \cap B_{1,+},
			\end{cases}
		\end{equation}
		because of the bound of $|g|\leq g_0\leq C_1(x''-x_1'')$. 
		
		By \eqref{eq: doubling index for H}, Remark \ref{rmk: H in general points}, Lemma \ref{lemma: bounding balls by tilde H}, Lemma \ref{lemma: annulus bounds}, Proposition \ref{prop:exp(cr)H(r) is nondecreasing}, and \eqref{eq: suprem de u}, we have
		\begin{align*}
		\int_{\partial B(\tilde x, \frac 1 8)\cap \Omega} \vert u \vert^2 d\sigma &\gtrsim
		e^{-c' N}\int_{\partial B(\tilde x, 4)\cap \Omega} \vert u \vert^2 d\sigma \approx e^{-c' N} H(\tilde x,4) 
		\gtrsim 
		e^{-c' N} \int_{B(\tilde x,4)}\vert u \vert^2 dy\\ 
		&\geq e^{-c' N}\int_{B(x_{\sup{}},1)} \vert u \vert^2 dy \gtrsim e^{-c' N}|u(x_{\sup{}})|^2 = e^{-c' N} 
		\end{align*}
		where $x_{\sup{}}$ is a point in $\partial B(\tilde x,3)$ where $|u(x_{\sup{}})|=1$.
		Assuming that $\tau$ is small enough, we have $B(\tilde x, \frac 1 8) \cap \Omega \subset \frac 1 4 B_{1,+}$. Using \eqref{eq: g and u lemma zeros} and Lemma \ref{lemma: bounding H by H at a different point}, we get
		\begin{align*}
			\left(\int_{\partial \frac{1}{4} B_{1,+}} g^{2} d\sigma \right)^{1 / 2} 
			&\geq
			\left(\int_{\partial \frac{1}{4} B_{1,+}} u^{2} d\sigma \right)^{1 / 2} - C_2 \tau
			\\
			&\geq
			\left(\tilde c\int_{\partial B\left(\tilde x, \frac{1}{8}\right)} u^{2} d\sigma \right)^{1 / 2}-C_{2} \tau .
		\end{align*}
		If we assume that $\tau$ (depending on $N$) is small enough, we conclude that
		\begin{equation}
			\label{eq: bound on g by N}
			\Vert g \Vert_{L^2\left(\partial \left (\frac 1 4 B_{1,+}\right)\right)} \geq c_1 e^{-c'N/2}.
		\end{equation}		
	
		Now let's estimate the doubling of $H^g(x_1,r)$ for $g$ by using \eqref{eq: suprem de u} and \eqref{eq: bound on g by N}:
		\[
		\log \frac{\int_{\partial \frac 1 2 B_{1,+}} g^2 d \sigma}
		{\int_{\partial \frac 1 4 B_{1,+}} g^2 d \sigma}
		\leq
		\log \frac{\sigma(\partial \frac 1 2 B_{1})}{c_1e^{-c'N/2}} = \log c'' + c'N/2 \leq C(N+1) .
		\]
		This gives us an upper bound for $N^g(x_1,1/4)$, the frequency for $g$ at $x_1$.
		Also note that \eqref{eq: bound on g by N} implies 
		\[
		\sup_{\frac 1 4 B_{1,+}} \vert g \vert \geq c e^{-c' N/2}.
		\]
		Then, by Lemma \ref{lemma: toy lemma elliptic version}, there exists $x_* \in \Gamma_0 \cap  B_{1/S}$, $c_2 = c_2(C(N+1), S) >0$, and $\rho = \rho(C(N+1), S)$ such that
		\[
		|g(x)| \geq c_{2}\left(x^{\prime \prime}-x_{1}^{\prime \prime}\right) \text { for } x=\left(x^{\prime}, x^{\prime \prime}\right) \in B\left(x_{*}, \rho\right) \cap B_{1,+} .
		\]
		We may assume that $g>0$ in $B\left(x_{*}, \rho\right) \cap B_{1,+}$, otherwise we consider $-u$ and $-g$ in place of $u$ and $g$. From \eqref{eq: g and u lemma zeros}, we obtain
		\begin{equation}
		\label{eq: last equation of lemma of zeros}
		u(x) \geq g(x)-4 C_1 \tau \geq c_{2}\left(x^{\prime \prime}-x_{1}^{\prime \prime}\right)-4 C_1 \tau \quad \text { in } \quad B\left(x_{*}, \rho\right) \cap \Omega .
		\end{equation}
		Note that $\rho$ does not depend on $\tau$ and for $\tau$ small enough we have $B\left(x_{*}, \frac{\rho}{4}\right) \cap \partial \Omega \neq \varnothing .$ 
		
		Our goal is to show that $u > 0$ on $B(x_*, \frac \rho 2) \cap \Omega$ and we will use Lemma \ref{lemma: silva-savin} to this end (from now on, constants $\delta$ and $M$ come from the statement of Lemma \ref{lemma: silva-savin}). 
		Restrict $\tau$ to be small enough so that $u>0$ in $B(x_*, \rho) \cap \Omega \cap \{ (x', x'') | x'' > x''_* + \delta \rho\}$, which we can do thanks to \eqref{eq: last equation of lemma of zeros}. Now, we  multiply the function $u$ by a very large constant $K$ so that $Ku \geq M$ in the same set $B(x_*, \rho) \cap \Omega \cap \{ (x', x'') | x'' > x''_* + \delta \rho\}$. Finally, we make $\tau$ even smaller so that $Ku \geq -1$ in $B(x_*, \rho) \cap \Omega$, thanks again to \eqref{eq: last equation of lemma of zeros}. Now, by a rescaled version of Lemma \ref{lemma: silva-savin} to $B(x_*,\rho)\cap\Omega$, the function $Ku$ restricted to $B(x_*, \rho/2)\cap\Omega$ is positive, which implies that $u$ is positive too and ends the proof. 
	\end{proof}

\section{Proof of Theorems \ref{thm: u is positive besides a set of positive codimension} and \ref{thm: measure of nodal set near the boundary}}
\label{section: pf of thm 1.1}
We will combine Lemmas \ref{lemma: Key lemma} and \ref{lemma: 2nd hyperplane lemma} to prove Theorems \ref{thm: u is positive besides a set of positive codimension} and \ref{thm: measure of nodal set near the boundary}.

Note that, in the present section, we refer to the frequency function $N$ of a solution of $\operatorname{div}(A\nabla u)=0$ in $\Omega$ as in the statement of Theorem \ref{thm: u is positive besides a set of positive codimension}. Without further mention, we will use that the constants $L_A$ and $\Lambda_A -1$ from the matrix $A(x)$ are very small thanks to Remark \ref{rmk: small LA and LambdaA}. Moreover, we consider $\Omega$ with the Whitney structure defined in Section \ref{section: whitney structure}.

First, we give a corollary to Lemma \ref{lemma: 2nd hyperplane lemma} in a language closer to that of Lemma \ref{lemma: Key lemma}.
\begin{coro}
	\label{coro: 2nd hyperplane} 
	For any $N>0$ and $S \gg 1$ large enough there exist positive constants $\tau_0(N,S)$ and $K(N,S)$ such that the following statement holds. Suppose $\Omega \subset \mathbb R^d$ has Lipschitz constant $\tau < \tau_0$ and $Q$ is a cube in $\mathcal D_\mathcal W(R)$ such that $N(x_Q, S \ell(Q)) \leq N$. Then, for all $\tilde K \geq K$, there exist 
	cubes $Q''_1, \hdots, Q''_{2^{(d-1)(\tilde K-K)}}$ such that, for all $j$,
	\begin{enumerate}
		\item the center of $Q_j''$ lies in $\Sigma$,
		\item  $u|_{Q_j''\cap\Omega}$ does not have zeros,
		\item there exists $Q_j' \in \mathcal D^{\tilde K}_\mathcal W(Q)$ such that $Q''_j$ is a vertical translation of $Q_j'$.
	\end{enumerate} 
	In particular, there exists $\delta_0(N,S)>0$ such that 
	\[
	m_{d-1}\left (\bigcup_{i} \Pi(Q_i')\right) \geq \delta_0 m_{d-1}(\Pi(Q)) .
	\]
\end{coro}
These cubes $Q''_j$ are the cubes that are contained in the ball $B$ given by Lemma \ref{lemma: 2nd hyperplane lemma} and are vertical translation of cubes in $\mathcal D^{\tilde K}_\mathcal W(Q)$. From now on, given a cube $Q \in \mathcal D_\mathcal W(R)$, we denote by $t(Q)$ the unique cube $Q'$ such that its center lies on $\Sigma$ and $Q'$ is a vertical translation of $Q$.

Next, we present a modified frequency function for which we prove good behavior as a consequence of Lemmas \ref{lemma: Key lemma} and \ref{lemma: 2nd hyperplane lemma}. 

\subsection{Modified frequency function}
\label{section: modified freq function}
Let $R$ be a cube in $\mathcal W$ such that it satisfies the conditions of Lemma \ref{lemma: Key lemma} with $S = S_1 \gg 1$ large enough so that $CS_1^{-1/2}$ (also from the statement of Lemma \ref{lemma: Key lemma}) is small enough (we will specify the precise relation later). The use of this lemma gives us constants $K_1 := K(S_1)$ and $\delta_1 := \delta_0$ (that does not depend on $S_1$).

Consider also Corollary \ref{coro: 2nd hyperplane} with fixed $N=2N_0+1$ (where $N_0$ is the constant of Lemma \ref{lemma: Key lemma}) and $S=S_2$ large enough but smaller than $S_1$ (note that both constants are independent). The use of this corollary gives us constants $K_2 := K(N, S_2)$ and $\delta_2 := \delta_0(N,S_2)$.
In particular, we may assume that $K_2$ is smaller or equal than $K_1$ and that $CS_1^{-1/2}$ is small enough depending on $\min(\delta_1,\delta_2)$. 

We remark that the use of Lemma \ref{lemma: Key lemma} and Corollary \ref{coro: 2nd hyperplane} with constants $S_1$ and $S_2$ respectively requires that the domain has small enough Lipschitz constant $\tau$. 
For the rest of this section, we will denote $\epsilon := CS_1^{-1/2}$, $K:=K_1$, and $\delta_0=\min(\delta_1,\delta_2)$ (in particular, $\delta_0$ and $S_1$ are independent constants).

Define
\[
N(Q) := N(x_Q, S_1 \ell(Q))
\]
for any $Q \in \mathcal D_{\mathcal W}^{jK} (R)$ for $j\geq 0$.
Notice that $1+2N(Q) \geq N(x_Q, S_2 \ell(Q))$ thanks to Proposition \ref{prop:exp(Cr)N(r) is nondecreasing} (we assume $x_Q$ and $S_1 \ell(Q)$ are admissible for all children $Q$ of $R$ in the Whitney tree structure, thanks to $\tau$ being small enough).

We define the modified frequency function $N'(Q)$ for $Q \in \mathcal D_\mathcal W^{jK}(R)$, $j \geq 0$, inductively. For $j=0$, we define $N'(R) = \max(N(R),N_0/2)$. Assume we have $N'(P)$ defined for all cubes $P \in \mathcal D_\mathcal W^{iK}(R)$ for $0\leq i < j$. Fix $\widehat Q \in \mathcal D_\mathcal W^{(j-1)K}(R)$ and consider its vertical translation $t(\widehat Q)$ centered on $\Sigma$. Then:
\begin{itemize}
	\item[(a)] if $u$ restricted to $t(\widehat Q)\cap\Omega$ has no zeros, define $N'(Q) = N'(\widehat Q)/2$ for $\lceil\delta_0\cdot \operatorname{card}\{\mathcal D^K_\mathcal W(\widehat Q)\}\rceil$ of the cubes $Q$ in $\mathcal D^K_\mathcal W(\widehat Q)$ and $N'(Q) = (1+\epsilon)N'(\widehat Q)$ for the rest of cubes in $\mathcal D^K_\mathcal W(\widehat Q)$ (the particular choice is irrelevant),
	
	\item[(b)] if $u$ restricted to $t(\widehat Q)\cap\Omega$ has zeros, choose $Q\in\mathcal D^K_\mathcal W(\widehat Q)$, and 
	\begin{itemize}
		\item[1.] if its vertical translation $t(Q)$ satisfies that $u$ restricted to $t(Q)\cap\Omega$ has no zeros, define $N'(Q) = N'(\widehat Q)/2$,
		\item[2.] else define $N'(Q) = \max (N(Q), N_0/2)$.
	\end{itemize}
\end{itemize} 
Note that if a cube $Q$ satisfies that $u$ restricted to  $t(Q)\cap\Omega$ has no zeros, then all its descendants in the Whitney cube structure will satisfy the same property and (a) applies to them. Alternatively, if a cube $Q$ satisfies that $u$ restricted to  $t(Q)\cap\Omega$ changes sign, then all its predecessors in the Whitney cube structure will satisfy the same and (b2) applies to them.

Now, a combination of Lemma \ref{lemma: Key lemma} and Corollary \ref{coro: 2nd hyperplane} yields the following behavior for $N'(Q)$ for $Q \in\mathcal D_\mathcal W^{jK}(R), ~j\geq0$. Consider a cube $\widehat Q$ and its vertical translation $t(\widehat Q)$. Then:
\begin{itemize}
	\item If $u$ restricted to $t(\widehat Q)\cap\Omega$ has zeros and $N(\widehat Q) \geq N_0$, then Lemma \ref{lemma: Key lemma} tells us that at least $\lceil\delta_0\cdot \operatorname{card}\{\mathcal D^K_\mathcal W(\widehat Q)\}\rceil$ cubes in $\mathcal D^K_\mathcal W(\widehat Q)$ satisfy $N(Q) \leq N(\widehat Q)/2$. Moreover, in this case, $N'(Q) = \max(N(Q), N_0/2) \leq N(\widehat Q)/2 = N'(\widehat Q)/2$ where we have used that $N(\widehat Q)>N_0$ and that (b) applies. For the rest of the cubes $Q$ in $\mathcal D^K_\mathcal W(\widehat Q)$, we have $N'(Q) = \max(N(Q),N_0/2) \leq (1+\epsilon)N(\widehat Q) = (1+\epsilon) N'(\widehat Q)$  where we have used again that $N(\widehat Q)>N_0$ and (b) applies.
	
	\item If $u$ restricted to $t(\widehat Q)\cap\Omega$ has zeros and $N(\widehat Q) < N_0$, then Corollary \ref{coro: 2nd hyperplane} enters in play and it tells us that at least $\lceil\delta_0\cdot \operatorname{card}\{\mathcal D^K_\mathcal W(\widehat Q)\}\rceil$ cubes $Q$ in $\mathcal D^K_\mathcal W(\widehat Q)$ satisfy that $t(Q)\cap\Omega$ does not contain zeros of $u$. For these cubes, $N'(Q) = N'(\widehat Q)/2 = \max(N(\widehat Q),N_0/2)/2<N_0/2$. For the rest of the cubes $Q$ in $\mathcal D^K_\mathcal W(\widehat Q)$, we have $N'(Q) = \max(N(Q), N_0/2) \leq \max((1+\epsilon)N(\widehat Q), N_0/2) \leq (1+\epsilon)N'(\widehat Q)$.
	
	\item If $u$ restricted to $t(\widehat Q)\cap\Omega$ has no zeros, then we have defined $N'(Q) = N'(\widehat Q)/2$ for $\lceil\delta_0\cdot \operatorname{card}\{\mathcal D^K_\mathcal W(\widehat Q)\}\rceil$ cubes $Q$ in $\mathcal D^K_\mathcal W(\widehat Q)$. We have defined $N'(Q) = (1+\epsilon)N'(\widehat Q)$ for the rest of cubes $Q$ in $\mathcal D^K_\mathcal W(\widehat Q)$ .
\end{itemize}	
Summing up, for $\widehat Q \in \mathcal D_\mathcal W(R)$ and random $Q \in \mathcal D_\mathcal W^K(\widehat Q)$, we have
\[
N'(Q) \leq 
\begin{cases}
	N'(\widehat Q)/2, & \mbox{ with probability at least $\delta_0$, } \\
	N'(\widehat Q)(1+\epsilon), &  \mbox{ with probability at most $1-\delta_0$.}
\end{cases}
\]
This is similar to the behavior of the frequency function $N$ given by Lemma \ref{lemma: Key lemma} but without the restriction $N'(\widehat Q) \geq N_0$.

Let's summarize the dependence of the constants that have appeared in this section (omitting its dependence on the dimension $d$ and on $L_A$ and $\Lambda_A$ by assuming we are in a setting like the one described in Remark \ref{rmk: small LA and LambdaA}).
On one hand, the constants $K_2$ and $\delta_2$ given by Corollary \ref{coro: 2nd hyperplane} are absolute since  they depend on $S_2$ and $N = 2N_0 + 1$ which also are absolute constants. We do require $\tau$ small enough to use Corollary \ref{coro: 2nd hyperplane}.
On the other hand, we have $K_1$ depending on $S_1$, $\epsilon = CS_1^{-1/2}$, and an absolute constant $\delta_1$ given by Lemma \ref{lemma: Key lemma}. To use Lemma \ref{lemma: Key lemma}, we require $\tau$ small enough depending on $S_1$. For the arguments that follow, we need $\epsilon$ small enough depending on $\delta_0 = \min(\delta_1, \delta_2)$. Though, since both constants ($\epsilon$ and $\delta_0$) are independent, we can choose $S_1$ large enough and, thus, we require $\tau$ small enough (depending on $S_1$) to use Lemma \ref{lemma: Key lemma}.

\subsection{Proof of Theorem \ref{thm: u is positive besides a set of positive codimension}}
The idea behind the proof is that Lemma \ref{lemma: Key lemma} allows us to use that most cubes in any generation of the Whitney tree satisfy $N(x_Q, S \ell(Q)) \leq N_0$. Then, we can apply Lemma \ref{lemma: 2nd hyperplane lemma} to these cubes, thus covering most of $\Sigma$ by balls where $u$ does not change sign. 

\begin{proof}[Proof of Theorem \ref{thm: u is positive besides a set of positive codimension}]
	We will prove the result for the projection of a single cube $\Pi(R)$. Afterwards, we can cover any compact in $\Sigma$ by a finite union of such cubes which leaves stable the Minkowski dimension estimate.
	
	For $x \in \Pi(R)$, we denote by $Q_j(x)$ the unique cube $Q_j \in \mathcal D_\mathcal W^{jK}(R)$ such that $x \in \Pi(Q_j)$ for {some integer $K$ large enough that will be fixed later.} 
	We will say that $Q_j(x)$ is a \textit{good} cube if $N'(Q_j) \leq \frac{1}{2} N'(Q_{j-1})$ and that it is \textit{bad} otherwise.
	
	\begin{remark}
		\label{rmk:reducing to N_0/2}
		Note that, with the previous definitions, $N'(Q) < N_0/2$ implies that for all $x \in t(Q)\cap\Sigma$ there is a neighborhood where $u$ does not vanish in $\Omega$.
		Thus we only need to study the Minkowski dimension of the set of points $x \in \Pi(R)$ such that they are not in $\Pi(Q)$ for some $Q\in\mathcal D_{\mathcal W}(R)$ with $N'(Q) < N_0/2$.  	
		Also, notice that the map $\Pi: \Sigma \cap \Pi^{-1}(\Pi(R)) \mapsto \Pi(R)$ is biLipschitz and thus it preserves Minkowski dimensions.
	\end{remark}
	We define the \textit{goodness frequency} of a point $x \in \Pi(R)$ as
	\[
	F_j(x) = \frac{1}{j} \#\{ \text{good cubes in } Q_1(x), \hdots, Q_j(x) \}
	\]
	for $j\in\mathbb N$. We define $\overline F(x) = \limsup_{n \rightarrow \infty} F_n(x)$. Let $\alpha(\delta_0)>0$ be such that 
	\[
	\frac{\delta_0}{1-\delta_0} \frac{1-\alpha}{\alpha} = 3	
	\]
	(in particular, $\alpha<\delta_0$) and $\epsilon_0(\alpha)>0$ such that
	\[
	\alpha  = \frac{\log (1+\epsilon_0)}{\log(1+\epsilon_0) + \log 2}.
	\]
	Note that for any $0<\epsilon<\epsilon_0$, we have
	\begin{equation}
	\label{eq:def_comb_eps_alpha}
	\left ( \frac 1 2 \right)^\alpha (1+\epsilon)^{1-\alpha}<1.
	\end{equation}
	For all $j>1$, define 
	\[
	\mu_j=\frac 1 j \log_2 \left ( \frac{2N'(R)}{ N_0}\right)\geq0.
	\]
	\begin{claim} For all $j>1$, the following holds		\[
		F_j(x) \geq \alpha + \mu_j \implies N'(Q_j(x)) < \frac{N_0}{2}. 
	\]
	\end{claim}
	\begin{proof}[Proof of claim]
		We have that
		\[
		N'(Q_j(x)) \leq \left ( \frac 1 2 \right)^{j(\alpha+\mu_j)}(1+ \epsilon)^{j(1-\alpha-\mu_j)} N'(R) < 
		\left ( \frac 1 2 \right)^{j \mu_j} (1+\epsilon)^{-j \mu_j} N'(R)
		< \frac{N_0}{2}
		\]
		{by \eqref{eq:def_comb_eps_alpha}, and the definition of $F_j(x)$ and $\mu_j$.}
	\end{proof}

	Now, thanks to the previous claim and Remark \ref{rmk:reducing to N_0/2}, we can reduce the problem to studying the Minkowski dimension of the set of points 
	\[
	E = \{ x \in \Pi(R) ~|~ F_j(x) < \alpha + \mu_j, ~\forall j \in \mathbb N \}.
	\]
	
	{If we consider a random sequence of cubes $(Q_j)_j$ with $Q_0 = R$ and $Q_j \in \mathcal D_{\mathcal W}^{K}(Q_{j-1})$, and let $x \in \bigcap_{j \geq 0}\Pi(Q_j)$, the probability that $F_j(x) \leq \beta_j$ for $j \in \mathbb N$ and $\beta_j \in (0,1)$ is bounded above by}
	\[
	\sum_{i=0}^{\lfloor j \beta_j \rfloor} {j \choose i} \delta_0^i (1-\delta_0)^{j-i} .
	\]
	Note that choosing randomly such a sequence is equivalent to choosing a random $x \in \Pi(R)$ uniformly.
	{In what follows we will assume that $\beta_j$ satisfy $2<\frac{\delta_0}{1-\delta_0} \frac{1-\beta_j}{\beta_j}<4$ for all $j>0$, in particular $\beta<\delta_0$.} Let's find an upper bound for the previous quantity for very large $j$:
	\[
	\boxed A_j := \sum_{i=0}^{{\lfloor j \beta_j \rfloor}}
	{j \choose i} (1-\delta_0)^{j-i} \delta_0^i
	=
	(1-\delta_0)^j\sum_{i=0}^{{\lfloor j \beta_j \rfloor}}
	{j \choose i} \left ( \frac{\delta_0}{1-\delta_0} \right )^i.
	\]
	Observe that for $\beta_j<1/2$ we have
	\[
	{j \choose k-1} < \frac{\beta_j}{1-\beta_j} {j \choose k}, \quad\text{for } 0 < k \leq \lfloor \beta_j j \rfloor.
	\]
	This is because
	\[
	\frac{{j \choose k-1}}{{j \choose k}} = \frac{k! (j-k)!}{(k-1)! (j-k+1)!} = \frac{k}{j-k+1} < \frac{\beta_j}{1-\beta_j}.
	\]
	Iterating this inequality, we obtain
	\[
	{j \choose i} < \left ( \frac{\beta_j}{1-\beta_j} \right)^{\lfloor j \beta_j\rfloor-i} {j \choose \lfloor j \beta_j \rfloor}, \quad \mbox{for } i<\lfloor j \beta_j \rfloor.
	\]
	Using this observation we can bound  $\boxed A_j$ by
	\[
	(1-\delta_0)^j\sum_{i=0}^{{\lfloor j\beta_j \rfloor}}
	{j \choose i} \left ( \frac{\delta_0}{1-\delta_0} \right )^i \leq 
	(1-\delta_0)^j \left(\frac{\beta_j}{1-\beta_j} \right)^{\lfloor j\beta_j\rfloor } {j \choose \lfloor j\beta_j \rfloor} 
	\sum_{i=0}^{{\lfloor j\beta_j \rfloor}}
	\left ( \frac{\delta_0}{1-\delta_0}  \frac{1-\beta_j}{\beta_j}\right )^i .
	\]
	We use Stirling's formula to approximate
	\[
	{j \choose \lfloor j\beta_j \rfloor } \approx \frac{\sqrt{2\pi j} \left (\frac{j}{e} \right)^{j}}
	{\sqrt{2\pi \beta_j j} \left (\frac{\beta_j j}{e} \right)^{\beta_j j}
		\sqrt{2\pi (1-\beta_j) j} \left (\frac{(1-\beta_j )j}{e} \right)^{(1-\beta_j )j}}
	=
	\frac{1}{\sqrt{2\pi\beta_j(1-\beta_j)j}} \left(\frac{1}{\beta_j^{\beta_j} (1-\beta_j)^{1-\beta_j}}\right)^j.
	\]
	We also estimate
	\[
	\left(\frac {\beta_j} {1-\beta_j}\right)^{\lfloor j \beta_j \rfloor} \approx \left(\frac {\beta_j} {1-\beta_j}\right)^{ j \beta_j } 
	\,\,\text{ and }\,\,\,\,
	\sum_{i=0}^{{\lfloor j\beta_j \rfloor}}
	\left (\frac{\delta_0}{1-\delta_0}  \frac{1-\beta_j}{\beta_j}\right )^i \approx  \left(\frac{\delta_0}{1-\delta_0}  \frac{1-\beta_j}{\beta_j}\right )^{j\beta_j} .
	\]
	Observe that the comparability constants in the previous approximations depend on the upper and lower bounds of $\beta_j$ but not on $j$ for $j$ large enough depending on $\delta_0$.
	
	Summing everything up, we obtain
	\[
	\boxed A_j \lesssim \frac{2}{\sqrt{2\pi\beta_j(1-\beta_j)j}}
	\underbrace{\left ( 
		\frac{(1-\delta_0)^{1-\beta_j} {\delta_0}^{\beta_j}}{\beta_j^{\beta_j} (1-\beta_j)^{1-\beta_j}}
		\right)^j}_{z(\beta_j)^j},
	\]
	for $j$ large enough.
	Observe that $z(\beta_j)<1$ for $\beta_j<\delta_0$.
	
	Now, we will choose a suitable covering of the set $E$ by projections of cubes in $\mathcal D^{jK}_{\mathcal W}(R)$. First, pick $\epsilon<\epsilon_0$ (equivalently $S_1$ large enough, recall the discussion in Section \ref{section: modified freq function}). Observe that by choosing $\epsilon$ we also fix $K$. For $j \geq 1$, set 
	\[
	E_{j} := \{ x \in \Pi(R) ~|~ F_j(x) \leq \alpha+\mu_j \},
	\]
	so that $E = \bigcap_j E_j$.
	Let's upper bound the ($K$-adic) Minkowski dimension of $E$ by finding a certain cover of $E_j$ by projections of cubes in $\mathcal D^{jK}_\mathcal W(R)$ (note that there are $M = 2^{(d-1)K}$ cubes in $\mathcal D^K_\mathcal W(R)$).
	
	Using the previous asymptotics (setting $\beta_j = \alpha + \mu_j$), we can cover $E_j$ (for $j$ large enough) with 
	\[
	C M^j\frac{2}{\sqrt{2\pi(\alpha+\mu_j)(1-\alpha - \mu_j)j}} z(\alpha+\mu_j)^j
	\]
	projections of cubes in $\mathcal D^{jK}_\mathcal W(R)$ and each of those cubes has side length $M^{-j/(d-1)}$.
	
	Now we are ready to upper bound the Minkowski dimension of the set $E$. We will use the following definition of upper Minkowski dimension
	\[
	\operatorname{dim}_{\overline{\mathcal M}} E = \limsup_{j\to\infty} \frac{\log \#(\{\mbox{$K$-adic cubes $Q$ of side length $K^{-j}$ that satisfy $Q \cap E \neq \varnothing$}\})}{j\log K}.
	\]
	which is equivalent to the dyadic one in \eqref{eq: Minkowski dimension}.
	By covering a single set $E_j \supset E$ and making $j\to\infty$, we obtain the following upper Minkowski dimension estimate
	\begin{align*}
		\operatorname{dim}_{\overline{\mathcal M}} (E) &\leq \lim_{j \rightarrow\infty} \frac{ j (\ln M + \ln z(\alpha + \mu_j)) + \ln \left(\frac{2}{\sqrt{2\pi(\alpha+\mu_j)(1-\alpha - \mu_j)j} } \right )}{  {(d-1)^{-1}j } \ln M} \\
		&= (d-1) \frac{\ln M + \ln z(\alpha)}{\ln M} < d-1 
	\end{align*}
	since $z(\alpha)<1$.
	
\end{proof}

\subsection{Planar case of Theorem \ref{thm: u is positive besides a set of positive codimension}}~
\noindent In the planar case, Theorem \ref{thm: u is positive besides a set of positive codimension} asserts that we can cover any compact $K \subset \Sigma$ by balls where $u$ does not change sign inside apart from a finite set of points. Moreover, this is valid for general Lipschitz domains (and domains with worse boundary regularity).

\begin{proof}[Proof of Theorem \ref{thm: u is positive besides a set of positive codimension} in the planar case]
	Without loss of generality, suppose that $\Omega$ is simply connected and bounded.
	
	By \cite[Theorem 16.1.4]{AIM}, there exists a $K$-quasiconformal map $\phi : \mathbb C \rightarrow \mathbb C$ such that $u= w \circ \phi$ with $w$ harmonic in $\phi(\Omega)$.
	Since $\phi$ is biH\"older continuous it is enough to prove the desired result for harmonic functions in Jordan domains.
	
	Now, consider a conformal mapping between $\phi(\Omega)$ and the disk $\mathbb D$ which extends to a homeomorphism up the boundary, and let $v$ be the induced harmonic function in the disk.
	
	Denote by $\tilde\Sigma$ the open set in $\partial \mathbb D$ where $v$ vanishes and by $\tilde E$ the set of points in $\tilde\Sigma$ such that for every neighborhood $v$ changes sign. Then $\tilde E$ must be a discrete set.
	This is because it coincides with the zero set of $\nabla v$ which is holomorphic (we can locally extend $v$ to $\mathbb D^c$ by reflection using the Kelvin transform near $\tilde \Sigma$). Thus, since it is a discrete set, it is countable and finite inside any compact. Note that all maps we have considered are homeomorphisms, thus the set where $v$ changes sign in every neighborhood is transported to another countable discrete set in $\Sigma\subset \overline\Omega$.

\end{proof}

\subsection{Estimates on the measure of nodal sets in the interior of the domain}
Theorem 6.1 in \cite{Lo1} estimates the $(d-1)$-dimensional Hausdorff measure of the nodal set in a cube in terms of its \textit{doubling index}, which is a quantity intimately related to the frequency function used in this paper. We present the following reformulation of Logunov's theorem avoiding the use of doubling indices.
\begin{theorem}
	\label{thm: Logunov nodal set}
	There exist positive constants $r,R,C$ depending on the Lipschitz constant $L_A$ of $A(x)$, ellipticity constant $\Lambda_A$ of $A(x)$, and dimension $d$ such that the following statement holds. Let $u$ be a solution of $\operatorname{div}(A\nabla u)=0$ on $B(0,R) \subset \mathbb R^d$. Then, for any cube $Q \subset B(0,r)$, we have
	\[
	\mathcal H^{d-1}(\{u = 0\} \cap Q) \leq C \ell(Q)^{d-1} (N(x_Q, 16\operatorname{diam}(Q))+1)^\alpha
	\]
	for certain $\alpha = \alpha(d) > 0$ and where $x_Q$ is the center of $Q$.
\end{theorem}
\begin{remark}
	If we assume $N(x_Q, 16\operatorname{diam}(Q)) > N_0$ for some $N_0$ positive, we can rewrite the previous theorem as
	\[
	\mathcal H^{d-1}(\{u = 0\} \cap Q) \leq C \ell(Q)^{d-1} (N(x_Q, 16\operatorname{diam}(Q)))^\alpha
	\]
	where now $C$ depends also on $N_0$ and $\alpha$.
\end{remark}

\subsection{Proof of Theorem \ref{thm: measure of nodal set near the boundary}}
To prove Theorem \ref{thm: measure of nodal set near the boundary} we will follow the ideas of \cite{LMNN} but exchanging the use of the {Donnelly-Fefferman} estimate for the size of nodal sets (see \cite{DF}) by Logunov's estimate (Theorem \ref{thm: Logunov nodal set}). This gives rise to a worse estimate (polynomial in the frequency) than the one of \cite{LMNN} (which is linear in the frequency but only valid for harmonic functions).

\begin{proof}[Proof of Theorem \ref{thm: measure of nodal set near the boundary}]
	Let $Q$ be a small enough cube of the Whitney structure (so that we can use the modified frequency function defined in Section \ref{section: modified freq function}). Again, we will use the notation $t(Q)$ for the unique cube $Q'$ with center on $\Sigma$ such that $Q'$ is a vertical translation of $Q$.
	\begin{claim}	
		The following equation holds
		\begin{equation}
			\label{eq: nodal set in compact}
			\mathcal H^{d-1}(\{u = 0\} \cap t(Q) \cap K) \leq C_0 N'(Q)^\alpha \ell(Q)^{d-1}
		\end{equation}
		for any $K$ compact inside $\Omega$ with $C_0$ and $\alpha$ independent of $K$.
	\end{claim}
	
	\begin{proof}[Proof of claim]
		Note that Equation \eqref{eq: nodal set in compact} holds for all cubes $Q$ small enough, since $t(Q) \cap K = \varnothing$ if $Q$ is small enough. This is because $\operatorname{dist}(K,\Sigma)>0$. We will proceed to prove this estimate by induction going from small cubes to large cubes.
		
		Assume it holds for all small cubes with $\ell(Q)<s$ . Now choose a larger cube $P$ in the Whitney cube structure with $\ell(P) < 2^K \ell(Q)$ (with $K$ given by Lemma \ref{lemma: Key lemma} as discussed in Section \ref{section: modified freq function}).
		Given such a cube $P$, we can cover $t(P) \cap \Omega$ with small cubes $t(Q_i)$ (intersecting the boundary)
		where $Q_i \in \mathcal D_{\mathcal W}^K(P)$ and with small cubes $Q_i'$ far from the boundary (small enough so that we can apply Theorem \ref{thm: Logunov nodal set} on them).
		
		Using that the Lipschitz constant of $\partial\Omega$ is small, we can bound the number of small cubes $Q_i'$ necessary. Moreover, using Lemma \ref{lemma: perturbation of frequency function with small gamma}, we can bound $N(x_{Q'}, 16 \operatorname{diam}(Q')) < 2 N(P) + 1$. Note that $N(P) \leq N'(P)$ in the case that $t(P) \cap \Omega \cap \{ u = 0 \} \neq \varnothing$.
		This allows us to bound the size of the nodal set on $\bigcup Q_i' \cap \Omega$ using Theorem \ref{thm: Logunov nodal set} by
		\[
		\sum_{Q_i'} \mathcal H^{d-1} (\{u = 0\} \cap Q_i') \leq C_1 N'(P) \ell(P)^{d-1}.
		\]
		
		We still need to bound the size of the nodal set in the boundary cubes $t(Q_i)$, which satisfy $N'(Q_i) \leq (1+\epsilon) N'(P)$. Moreover, we know that at least for $\lceil\delta_0 \operatorname{card}\{\mathcal D_\mathcal W^K(P)\}\rceil $ cubes $Q_i$ we have $N'(Q_i) \leq \frac {N'(P)}{2}$.
		Now we can use the induction hypothesis
		\begin{align*}
			\mathcal H^{d-1}(\{u = 0\} \cap K \cap \bigcup_{Q_i} t(Q_i)) &\leq
			\sum_{ N'(Q_i) > N'(P)/2} \mathcal H^{d-1}(\{u = 0\} \cap K \cap t(Q_i)) \\
			&~~ + 
			\sum_{N'(Q_i)\leq N'(P)/2} \mathcal H^{d-1}(\{u = 0\} \cap K \cap t(Q_i)) \\
			&\leq \sum_{ N'(Q_i) > N'(P)/2} C_0 N'(Q_i)^\alpha \ell(Q_i)^{d-1}\\ 
			&~~+ \sum_{N'(Q_i)\leq N'(P)/2} C_0 N'(Q_i)^\alpha \ell(Q_i)^{d-1} \\
			&\leq C_0 N'(P)^\alpha \ell(P)^{d-1} \left (  (1+\epsilon)^\alpha (1 - \delta_0) + \frac {\delta_0} {2^\alpha} \right) .
		\end{align*}
		We choose $\epsilon$ small enough (by increasing $S_1$) so that
		\[
		(1+\epsilon)^\alpha (1 - \delta_0) + \frac {\delta_0} {2^\alpha}  < 1
		\]
		noting that $\delta_0$ does not depend on $S_1$ or $\epsilon$.
		Finally, we choose $C_0$ large enough so that it absorbs all terms, that is,
		\[
		C_0 \left ( (1+\epsilon)^\alpha (1 - \delta_0) + \frac {\delta_0} {2^\alpha} \right ) + C_1 < C_0 .
		\]
		Note that the previous estimates do not depend on the compact $K$ chosen and we prove the claim.
	\end{proof}
	
	To treat a general (small) ball $B$ centered at $\Sigma$, we first cover it by a comparable cube $Q$ centered at $\Sigma$.
	Now, we choose a translate and dilation of the dyadic cube structure of $\mathbb R^d$ such that $Q = t(P)$ for some $P$ in the Whitney cube structure and we apply the previous claim.
	
\end{proof}
\begin{remark}
	\label{rmk: x in the thm of measure of nodal set}
	In the statement of Theorem \ref{thm: measure of nodal set near the boundary}, the point $\tilde x$ that appears is the center of a particular Whitney cube appearing in the proof.
	Nonetheless, Lemma \ref{lemma: perturbation of frequency function with small gamma} (and that the Lipschitz constant $\tau$ of the boundary is small to preserve admissibility) gives us a lot of freedom to choose $\tilde x$.
\end{remark}
\subsection{\texorpdfstring{$(d-1)$}{(d-1)}-dimensional Hausdorff measure of Dirichlet eigenfunctions}
\label{section: Yau}
Theorem \ref{thm: measure of nodal set near the boundary} allows us to study the zero set of solutions of the Dirichlet eigenvalue problem
\[\begin{cases} 
	\operatorname{div}(A \nabla u_\lambda) = -\lambda u_\lambda,~ \text{ in } \Omega, \\
	u_\lambda = 0, ~ \text{ on } \partial\Omega.
\end{cases}\]
In fact, one can show that
\begin{equation}
	\label{eq: yaus upper estimate}
	\mathcal H^{d-1}(\{u_\lambda = 0\}) \leq C(\Omega, \Lambda_A, L_A) \lambda^{\frac \alpha 2}
\end{equation}
for bounded domains $\Omega$ with local Lipschitz constant small enough for some $\alpha = \alpha(d) > 1$. For a detailed account of the proof in the harmonic case (and a sharper result), see Section 6 in \cite{LMNN}. We will only briefly sketch the main ideas behind its proof. Also note that this problem is intimately related to \textit{Yau's conjecture on nodal sets of Laplace eigenfunctions in manifolds} (see \cite{LM2,Lo1,Lo2,DF}).

The first step consists in passing from eigenfunctions to solutions of $\operatorname{div}(A\nabla u)=0$.
Consider the function $u(x,t) = u_\lambda(x) e^{\sqrt\lambda t}$ in the cylinder domain $\Omega \times \mathbb R \subset \mathbb R^{d+1}$. Let
\[
\tilde A(x) = \begin{pmatrix}
	A(x) & 0 \\ 0 & 1
\end{pmatrix},
\]
then we have that $u$ solves
\[
\operatorname{div}(\tilde A \nabla  u) = \underbrace{\operatorname{div}(A \nabla  u_\lambda)}_{-\lambda u_\lambda(x)} e^{\sqrt \lambda t} + \lambda u_\lambda(x) e^{\sqrt \lambda t} = 0 .
\]
Clearly, we have $\{ u(x,t)  = 0 \} = \{ u_\lambda(x)  = 0 \}\times \mathbb R $. Thus, we can restrict us to the study of the nodal set of $u(x,t)$.

The next necessary step is the \textit{Donnelly-Fefferman frequency estimate} \cite{DF}. For small balls $B$ contained in the domain, it is shown in \cite{LM1, DF} that $N^u(x_B, r(B)) \leq C(\Omega, \Lambda_A, L_A) \sqrt{\lambda}$. The previous result is also true for balls intersecting the boundary (see Lemma 10 in \cite{LMNN} for a proof in the harmonic case).

Finally, using the previous estimate together with Theorems \ref{thm: measure of nodal set near the boundary} and \ref{thm: Logunov nodal set}, one obtains the result in \eqref{eq: yaus upper estimate}.

\section{Proof of Corollary \ref{cor: unique continuation}}
	\label{section: proof of unique cont}
	Theorem \ref{thm: u is positive besides a set of positive codimension} tells us that we can decompose $\Sigma$ in its intersection with a countable family of balls $(B_i)_i$ and a set of Hausdorff dimension smaller than $d-1$ by taking an exhaustion of $\Sigma$ by compacts.
	Thanks to countable additivity, we only need to prove that 
	\[
	\mathcal H^{d-1}\left (\{x \in \Sigma \,|\, \partial_\nu u(x) = 0\}\cap B \right ) = 0
	\] 
	for any ball $B\in(B_i)_i$ given by the decomposition of Theorem \ref{thm: u is positive besides a set of positive codimension}.
	
	Before starting the proof, we define the concept of $\mathcal A_\infty$ weight.
	\begin{definition}
		\label{def: Ainfty weight}
		We say that a measure $\omega \in \mathcal A_\infty(\sigma)$ if there exist $0<\gamma_1, \gamma_2 < 1$ such that for all balls $B$ and subsets $E \subset B$, $\sigma(E) \leq \gamma_1 \sigma(B)$ implies $\omega(E) \leq \gamma_2 \omega(B)$.
	\end{definition}
	
	\begin{proof}[Proof of Corollary \ref{cor: unique continuation}]
		Consider $B$ centered on $\Sigma$ such that $u|_{B\cap\Omega}$ does not change sign. Without loss of generality, we assume that $u$ is positive.
		
		By Dahlberg's theorem \cite{Dah}, harmonic measure for the domain $B\cap \Omega$ is an $\mathcal A_\infty$ weight with respect to surface measure.
		By \cite{FKP}, 
		since the matrix $A(x)$ is uniformly elliptic and has Lipschitz coefficients, its associated elliptic measure $\omega_A$ is another $\mathcal A_\infty$ weight.
		In particular, it is well known that this implies that the density $\frac{d \omega_A}{d\sigma}$ can only vanish in a set of zero surface measure.
		
		On the other hand, the density of elliptic measure is comparable with $(A\nabla g, \nu)$ at the boundary (where $g$ is the Green function with pole outside $2B$). By the boundary Harnack inequality (see \cite{DS} for example), since $u$ is positive in $B\cap\Omega$, we have that $A\nabla u$ on $\Sigma\cap B$ is comparable to $A\nabla g$. This finishes the proof.
	\end{proof}
	
	\begin{remark}
		There is also a different approach to proving Corollary \ref{cor: unique continuation}, which is the one adopted in \cite{To} for harmonic functions. The ingredients are Lemma \ref{lemma: Key lemma}, \cite[Lemma 0.2]{AE} (which is also valid for the type of PDEs we consider, see the paragraph below the proof of  \cite[Lemma 2.2]{AE}, also \cite[Lemma 4.3]{To}), and a modification of \cite[Lemma 4.1]{To}. In particular, the tools of Sections \ref{section: balls without zeros} and \ref{section: pf of thm 1.1} are not indispensable for this result.
	\end{remark}

\section{Proof of Corollary \ref{cor: order of zeros}}
	\label{section: proof of order of zeros}
	\begin{definition}
		We define a non-truncated cone $\tilde C_\tau$ with aperture $\tau \in \mathbb R$ and vertex at $0$ as
		\[
		\tilde C_\tau(0) = \{(y',y'') \in \mathbb R^{d-1} \times \mathbb R ~|~ y'' > \tau |y'|\}.
		\]
		Notice that when $\tau = 0$, then $\tilde C_0$ is a half space and when $\tau<0$, $\tilde C_\tau$ is a non-convex cone.
		A truncated cone $C_{\tau,s}$ is a non-truncated cone $\tilde C_\tau$ intersected with the ball $B(0,s)$.
	\end{definition}
	First we will prove the result for harmonic functions.
	\begin{proof}[Proof of Corollary \ref{cor: order of zeros} in the harmonic case]
		Let $x \in B\cap\Sigma$ for some ball $B$ where $u$ does not vanish given by Theorem \ref{thm: u is positive besides a set of positive codimension}.
		Assume without loss of generality that $u|_{B\cap\Omega} >0$.
		
		Since $\Sigma$ is Lipschitz with Lipschitz constant $\tau$, we can find a small truncated cone $C_{\tau,s}$ of aperture $\tau$ with vertex at $x$ and contained in $\Omega$. 
		
		Let $g$ be the Green function (for the Laplacian) with pole at infinity of the non-truncated cone $\tilde C_\tau$. Since $u|_{C_{\tau,s}}$ is nonnegative up to the boundary, we can lower bound it by some adequate multiple of $g$ (that is $u \gtrsim g$ in $C_{\tau,s}$) by boundary Harnack inequality.
		Notice that $g$ is of the form $g(x) = g_r(|x|)g_\theta(\frac{x}{|x|})$ where $g_\theta$ is the first Dirichlet eigenfunction of the Laplace-Beltrami operator $\Delta$ in the domain ${\tilde C_\tau \cap \partial B(0,1)}$ with eigenvalue $\lambda_\tau$ and $g_r(|x|) = |x|^{-\frac d 2 + 1 + \frac{\sqrt{(d-2)^2-4\lambda_\tau}}{2}}$ (see \cite[Theorem 1.1]{An}).
		Thus $g_r$ gives an upper bound on the order of vanishing at the origin.
		Also notice that $\lambda_0 = 1-d$ (when $\tilde C_0$ is a half space) and thus $g_r(|x|) = |x|$.
		Since the Dirichlet eigenvalues $\lambda_\tau$ of $\tilde C_\tau\cap \partial B(0,1)$ for the Laplace-Beltrami operator vary continuously with the aperture $\tau$ of $\tilde C_\tau$, we can ensure that the order of vanishing of $g$ is close enough to $1$ by making $\tau$ small enough. The continuity of the variation of the eigenvalues with the aperture can be easily shown using the Rayleigh quotient.
		
		For the lower bound on the order of the vanishing at the origin, consider a cone $\tilde C_{-\tau}$ of aperture $-\tau$ (concave cone) and its Green function $g$ with pole at the infinity. Now $g|_{\Omega\cap \tilde C} \gtrsim u$ and we can follow the same argument.
	\end{proof}
	Next, we deal with the elliptic case, but first, we need several lemmas.	
	\begin{lemma}
		\label{lemma: N and vanishing order}
		Let $u$ be a solution of $\operatorname{div}(A\nabla u) = 0$ in $\Omega$ and $x_0 \in \Omega\cup\Sigma$.
		Then the vanishing order $\alpha$ of $u$ at $x_0$ satisfies
		\[
		\alpha(x_0) = 
		\lim_{r \downarrow 0}  \log_2 \frac{\sqrt{H(x_0,2r)}}{\sqrt{H(x_0,r)}}
		\] 
		when the limit exists.
	\end{lemma}
	\begin{proof}
		The proof has two parts. First, we see that 
		\[
		\fint_{B(x_0,r)} |u|\, dx \leq C_\alpha r^\alpha
		\iff 
		\sqrt{H(x_0,r)} \leq C_\alpha' r^\alpha.
		\]
		Notice that for $r$ small enough, we have
		\[
		\fint_{B(x_0,\Lambda_A^{-1}r)} |u| dy 
		\leq
		\left( \fint_{B(x_0,\Lambda_A^{-1}r)}|u(y)|^2 dy \right)^{1/2} 
		\lesssim
		\sqrt{H(x_0,r)} 
		\]
		thanks to Cauchy-Schwarz and Lemma \ref{lemma: bounding balls by tilde H}.
		On the other hand, if we let $x_* \in \overline{B(x_0,r)}$ be such that $|u(x_*)| = \sup_{B(x_0,r)} |u|$, then
		\[
		\sqrt{H(x_0, r)} \lesssim
		\sup_{B(x_0,r)} |u|
		\lesssim \fint_{B(x_*, \Lambda_A^{-1}r)} |u| dy
		\lesssim  \fint_{B(x_0, 2\Lambda_A^{-1}r)} |u| dy 
		\]
		using standard estimates for solutions of elliptic PDEs.
		
		The second part consists in showing that
		\[
		\lim_{r \downarrow 0}  \log_2 \frac{\sqrt{H(x_0,2r)}}{\sqrt{H(x_0,r)}} = \alpha
		\] 
		implies
		\[
		\sqrt{H(x_0,r)} \leq C_\gamma r^\gamma, \quad 0<\gamma<\alpha, ~\text{ and }~ \lim_{r\to0} \frac{\sqrt{H(x_0,r)}}{r^\gamma} = \infty, \quad \gamma>\alpha.
		\]
		Fix $\gamma<\alpha$. Then there exists some $k_\gamma>0$ such that for all $k>k_\gamma$, we have
		\[
		 \log_2 \frac{\sqrt{H(x_0,2^{-k+1})}}{\sqrt{H(x_0,2^{-k})}} > \gamma.
		\]
		Then, for all $n$ large enough, we have
		\[
		\sqrt{H(x_0, 2^{-n})} = \sqrt{H(x_0,1)}\, \prod_{j=1}^n \frac{\sqrt{H(x_0,2^{-j})}}{\sqrt{H(x_0,2^{-j+1})}}
		<
		C 2^{-\gamma(n-k_\gamma)} = C' \left (2^{-n} \right)^\gamma.
		\]
		Now, fix $\gamma>\alpha$. As before,there exists some $k_\gamma>0$ such that for all $k>k_\gamma$, we have
		\[
		\log_2 \frac{\sqrt{H(x_0,2^{-k+1})}}{\sqrt{H(x_0,2^{-k})}} < \gamma - \delta.
		\]
		for some small $\delta>0$.
		Finally, for all $n$ large enough, we obtain
		\[
		\sqrt{H(x_0, 2^{-n})} = \sqrt{H(x_0,1)}\, \prod_{j=1}^n \frac{\sqrt{H(x_0,2^{-j})}}{\sqrt{H(x_0,2^{-j+1})}}
		>
		C 2^{-(\gamma+\delta)(n-k_\gamma)} > C' \left (2^{- n}\right)^\gamma 2^{\delta n}.
		\]
		Since $\lim_{n\to\infty}2^{\delta n} = \infty$, this finishes the proof.
	\end{proof}
	\begin{lemma}
		\label{lemma: frequency bounded for positive solutions}
		Let $u$ be a positive solution to $\operatorname{div}(A\nabla u) = 0$ in a cone $C_{\tau,s}$ such that $u = 0$ on $\partial C_{\tau,s}\cap B(0,s)$. Assume also that $A(0) = I$. Then, its frequency $N(0,r)$ is bounded above for $0<r<\frac s 2$.
	\end{lemma}
	\begin{proof}
		Remember that
		\[
		N(0,r) = r \frac{\int_{B(0,r)} (A \nabla u, \nabla u) dy}{\int_{\partial B(0,r)} \mu u^2 d\sigma} \approx  r \frac{\int_{B(0,r)} | \nabla u|^2 dy}{\int_{\partial B(0,r)} u^2 d\sigma}.
		\]
		Using that $u$ extended by $0$ outside of $C_{\tau,s}$ is a subsolution, by Caccioppoli's inequality we have
		\[
		\int_{B(0,r)} |\nabla u|^2 dx  \lesssim \frac{1}{r^2} \int_{B(0,2r)} u^2 dx.
		\]		
		By Remark 3 (page 953) in \cite{AE}, we obtain the doubling property
		\[
		\int_{B(0,4r)} u^2 \,dx \lesssim \int_{B(0,\Lambda_A^{-1} r)} u^2 \,dx.
		\]
		uniformly for all $r$ small enough. The remark is stated for convex domains, but the condition needed is that the domain is star-shaped with respect to $0$. In the case of cones, this is trivially true.
		Using  Lemma \ref{lemma: bounding balls by tilde H}, we obtain that 
		\[
		\fint_{B(0,\Lambda_A^{-1}r)} u^2 dx \lesssim \fint_{\partial B(0,r)} u^2 d\sigma.
		\]
		
		Summing up, we get 
		\[
		N(0,r) \approx r \frac{\int_{B(0,r)} |\nabla u|^2 dx}{\int_{\partial B(0,r)} u^2 dx}  \lesssim 1
		\]
		for all $r$ small enough.
	\end{proof}
	The following lemma shows that the blow-up of positive solutions in cones converges to the Green function for the Laplacian in the domain $\tilde C_\tau$ with pole at $\infty$ (see the proof of Corollary \ref{cor: order of zeros} in the harmonic case).
	\begin{lemma}
		\label{lemma: blowups in cones}
		Let $u$ be a positive solution of $\operatorname{div}(A\nabla u)=0$ in a truncated cone $C_{\tau,s}$ that vanishes on $\partial C_{\tau,s} \cap B(0,s)$ and assume that $A(0) = I$.
		Consider any sequence of radii $r_k \downarrow 0$ (such that $r_1 < |p|$).
		Let 
		\[
		u_k(x) := \frac{u(r_k x)}{\left ( \int_{\partial B_{1}} \vert u(r_k y) \vert^2 d\sigma(y) \right )^{1/2}} .
		\] 
		Then $u_k$ converges in $W^{1,2}_{loc}(\tilde C_\tau)$ and in $C^{1,\alpha}_{loc}\left (\overline{\tilde C_\tau} \backslash \{0\} \right)$ to a multiple of the Green function $g$ with pole at $\infty$ for the Laplacian in $\tilde C_\tau$ for some exponent $0<\alpha<1$.
		In particular, $u$ and $g$ have the same vanishing order at $0$. 
	\end{lemma}
	\begin{proof}
		Without loss of generality assume that $s>2$ and $r_1=1$.
		Clearly, $u \in W^{1,2}( C_{\tau,1})$ and $u_k \in W^{1,2}( C_{\tau, {1/r_k}})$ for all $k>0$. Also notice that $N^u(0,r_k) = N^{u_k}(0,1)$ where $N^{u_k}$ is the frequency for the new PDE.
		Moreover, each $u_k$ satisfies $1 \approx H^{u_k}(0,1)$ and $\Vert \nabla u_k \Vert^2_{L^2(C_{\tau,1})} \approx N^{u_k}(0,1)$. 
		Observe that we can also bound $\Vert \nabla u_k \Vert_{L^2(C_{\tau,1/r_j})}$ and $H^{u_k}(0,r_j)$ for any $j \leq k$ (with bound depending on $r_j$).
		
		Thanks to Lemma \ref{lemma: frequency bounded for positive solutions}, we have that the sequence $(u_k)_k$ is bounded in Sobolev norm in any bounded set. Moreover, by boundary Schauder estimates (see Lemma 6.18 in \cite{GT}), $u_k\in C^{1,\alpha}(K)$ for any compact $K \subset \overline{\tilde C_\tau} \backslash \{0\}$ (for $k$ large enough depending on $K$). By Arzel\`a-Ascoli, there is a subsequence $(u_{k_j})_j$ that converges in norm $C^1$ to $\tilde u$ in $\overline{ C_{\tau,1}}\backslash B_{1/8}$.
		
		Moreover, we can also assume that $\tilde u$ is a weak limit in $W^{1,2}_\text{loc}(\tilde C_\tau)$. Thus,  $\tilde u$ is harmonic in $\tilde C_\tau$. To see this, fix any compact $K \subset \subset \tilde C_\tau$. Then, for all $\phi \in C_C^1(K)$, we have
		\[
		\int_K (A(r_k x)\nabla u_k, \nabla \phi) dx = 0, \quad \forall k>0.
		\]
		On the other hand, since $\tilde u$ is the weak limit of $u_k$, we have
		\[
		\int_K (\nabla \tilde u, \nabla \phi) dx = \lim_{k\to\infty} \int_K (\nabla u_k, \nabla \phi) dx = \lim_{k\to\infty}
		\underbrace{\int_K (A(r_kx) \nabla u_k, \nabla \phi) dx}_{\boxed A}
		- \underbrace{\int_K ((A(r_kx)-I) \nabla u_k, \nabla \phi) dx}_{\boxed B} .
		\]
		The term \boxed A is zero by the definition of weak solution and we can estimate \boxed B as
		\[
		\int_K ((A(r_kx)-I) \nabla u_k, \nabla \phi) dx \lesssim
		L_A r_k (\operatorname{dist}(0,K) + \operatorname{diam}(K)) \Vert \nabla u_k \Vert_{L^2(K)} \Vert \nabla \phi \Vert_{L^2(K)} \to_{k\to\infty} 0.
		\]
		Summing up, we get that $\tilde u$ is harmonic in $K$.
		
		\begin{claim} The only functions which are positive and harmonic in $\tilde C_\tau$ are the multiples of the Green function $g$ for the Laplacian in $\tilde C_\tau$ with pole at $\infty$. 
		\end{claim}
		For the proof of the claim see \cite[Lemma 3.7]{KT}.
		
		Since we have $C^1$ convergence away from the pole and $\int_{\partial B_1} u_k^2 \,d\sigma=1$ for all $k>0$, the multiple of the Green function in the claim is fixed. Thus, the convergence does not depend on the sequence $r_k$ chosen and is not up to subsequences.
		For this reason, we assume without loss of generality that $r_k = 2^{-k+1}$.

		We will see now that the limit
		\[
		\lim_{r \downarrow 0} \log_2 \sqrt{H^u(x, 2r)} - \log_2 \sqrt{H^u(x,r)}
		\]
	 	exists and by Lemma \ref{lemma: N and vanishing order} this coincides with the order of vanishing at zero. 
	 	
	 	If $r\in (1/2^k, 1/2^{k-1})$, then
	 	\[
	 	\log_2 \sqrt{H^u(0, 2r)} - \log_2 \sqrt{H^u(0,r)}
	 	=
	 	\log_2 \frac{\sqrt{H^{u_{k-1}}(0, 2^{k-1}r)} }{ \sqrt{H^{u_{k-1}}(0,2^{k-2}r)}}.
	 	\] 
	 	Since $r\in (1/2^k, 1/2^{k-1})$, this happens in the region where $u_{k}$ converges to $\tilde u$ in $C^1$ norm. Thus $H^{u_k}(0,r) \rightarrow_{k \rightarrow \infty} H^{\tilde u}(0,r)$ for $1/8<r<1$. By Lemma \ref{lemma: N and vanishing order} this implies that both $u$ and $\tilde u$ have the same vanishing order. 	 	
	\end{proof} 
	The general elliptic case is a direct consequence of the previous lemma following the same proof as in the harmonic case.

\section{\texorpdfstring{Example with $\operatorname{dim}_{\mathcal H}(\{ x\in\Sigma ~| \lim_{r \downarrow 0} {\omega(B(x,r))}{r^{1-d}} = 0\}) > d-1-\epsilon$}{Example of domain with set of zero density of harmonic measure with small codimension}}
	\label{section: counterexample to hausdorff estimates}
	
	The following example by Xavier Tolsa shows that in a Lipschitz domain $\Omega$ (even with small Lipschitz constant) there is no hope for a (non-trivial) Hausdorff dimension bound of the set of points of $\Sigma$ where $\partial_\nu u(x) = 0$.
	\begin{remark}
		\label{rmk: non-tangential derivative and density of harmonic measure}
		We are interested in the normal derivative, but it may not exist at every point $x \in \Sigma$.
		Nonetheless, since we are in a Lipschitz domain $\Omega$ for every point $x \in \Sigma$ we can consider a non-tangential cone $C_\tau(x)$ contained in the domain. Thus, we can consider a non-tangential approach to the normal derivative.
		That is, for $x\in\Sigma$, we define
		\[
		\color{black}\nabla_{nt} u(x) = \limsup_{\substack{y \rightarrow x \\ y \in C_\tau(x)}} \frac{|u(y)-u(x)|}{|y-x|} .
		\]
		We will show that the set $\{x \in\Sigma \,|\, \nabla_{nt} u(x) = 0\}$ can be very large (in terms of Hausdorff dimension).
	\end{remark}
	The following lemma shows that non-tangential derivatives for positive harmonic functions are closely related with the density of harmonic measure.		
	\begin{lemma}
		Assume $u$ is positive and harmonic in $\Omega$ and vanishes in $\Sigma$. Then, for $x\in\Sigma$, $\nabla_{nt} u(x) = 0$ if and only if $\lim_{r\downarrow 0} \frac{\omega(B(x,r))}{r^{d-1}} = 0$ where $\omega$ is the harmonic measure for the Laplacian in $\Omega$.
	\end{lemma}

	\begin{proof}[Proof of Remark \ref{rmk: non-tangential derivative and density of harmonic measure}]
		Fix a point $p \in \Omega$ {far} from $\Sigma$ and consider the Green function $g$ with pole at $p$.
		Then, in a neighborhood of $\Sigma$, the boundary Harnack inequality {(see \cite{DS})} implies that the non-tangential derivative of $u$ is $0$ if and only if the non-tangential derivative of $g$ is $0$ wherever one of the two exists (if one exists and is $0$ the other exists too by boundary Harnack inequality). This does not depend on the pole $p$ chosen.
		
		Now, the non-tangential derivative of $g$ is $0$ if and only if $\lim_{r \downarrow 0} \frac{\omega(B(x,r))}{r^{d-1}} = 0$. This is a consequence of \cite[Lemma 1]{Dah} which shows
		\[
		r^{d-2} g(x+Cre_d) \approx \omega^p(B(x,r))
		\]
		for some $C$ depending on the dimension and the Lipschitz constant of the domain.
	\end{proof}
	
	Now we can start setting up an appropriate domain.
	Consider, for $\lambda \in (0,1)$, the $\lambda$-Cantor set defined as
	\[
	\mathcal C_\lambda = \bigcap_{n=1}^\infty E_n^\lambda
	\]
	where $E_0^\lambda = [0,1]$ and $E_k^\lambda = \frac{1-\lambda}{2}E_{k-1}^\lambda \cup (\frac{1+\lambda}{2}+\frac{1-\lambda}{2}E_{k-1}^\lambda)$ for $k\geq1$. {We will refer to the intervals in $[0,1]\backslash E_k^\lambda$ by \textit{gaps}.}
	
	\begin{remark}
		The set $\mathcal C_\lambda$ has Hausdorff dimension between $0$ and $1$ depending on $\lambda$, but by making $\lambda$ small enough we can obtain a set with dimension arbitrarily close to $1$. From now, on we will denote it by $s = s(\lambda) = \dim_{\mathcal H} \mathcal C_\lambda$.
	\end{remark}
	
	\begin{remark}
		If $\lambda=1/(2k+1)$ for some $k\in\mathbb N$, then the set $\mathcal C_\lambda$ coincides with the set of real numbers in $[0,1]$ such that its decimal expansion in basis $2k+1$ does not contain the digit $k$.
		Instead of choosing $\lambda$, we will choose $k$ large enough so that the Hausdorff dimension $s$ of $\mathcal C_\lambda$ is as close as we want to $1$.
	\end{remark}
	
	\begin{theorem}\label{thm:borel_normal_numbers}
	 Let $\lambda = 1/(2k+1)$ for some $k\in\mathbb N$, and $s=s(\lambda) = \operatorname{dim}_{\mathcal H}(\mathcal C_\lambda)$. Then, $\mathcal H^s$-a.e. $x \in \mathcal C_\lambda$ satisfies that each possible digit $\{0,1,\hdots,k-1,$  $k+1,  \hdots, 2k\}$ in its decimal representation in basis $2k+1$ appears with the {same asymptotic frequency $1/(2k)$.} 
	\end{theorem} 
	{The proof of the theorem is analogous to the proof of Borel's normal number theorem.}
	Also, we will say that the points $x \in \mathcal C_\lambda$ that satisfy this are $\mathcal C_\lambda$-\textit{normal}.

	Fix $a\in(0,1)$ and let
	\[
	\Omega_\lambda := B(0,2) \cap \bigcup_{p \in \mathcal C_\lambda} \{C_a(p)\} \subset \mathbb R^2,
	\]
	where $C_a(p)$ are vertical open cones of aperture $a$ with vertex at $p$.
	Thus, $\Omega_\lambda$ is the union of all the open cones with aperture $a$ centered at a point of the Cantor set $\mathcal C_\lambda$. The domain $\Omega_\lambda$ clearly has Lipschitz boundary and its Lipschitz constant can be made arbitrarily small as it depends only on the aperture of the cones $a$. 
	\begin{remark}
	\label{rmk:gaps in cantor sets}
	{Let the interval $(\xi_1, \xi_2)$ be a gap in $[0,1]\backslash E_j^\lambda$. Then $\partial \left (C_a(\xi_1) \cup C_a(\xi_2)\right) \cap \{(x,y)\in\mathbb R^2\,|\, x \in (\xi_1,\xi_2)\} \subset \partial\Omega_\lambda$. In other words,
	the boundary of $\partial\Omega_\lambda$ looks like a cone (pointing upwards) where the Cantor set $\mathcal C_\lambda$ has gaps.}
	
	{Note that, if $\lambda = 1/(2k+1)$, the gaps correspond to subintervals of $[0,1]$ where all the numbers have the digit $k$ in the position $j$ of the decimal expansion in base $2k+1$.}
	\end{remark}
	
	Now we will prove that every $\mathcal C_\lambda$-normal point $(x,0) \in \mathcal C_\lambda \subset \mathbb R \cap \overline \Omega_\lambda$ satisfies that the density of harmonic measure at $(x,0)$ is $0$.

	\begin{propo}
	\label{propo: upper density is 0 Hs ae}
	Let $\Omega_\lambda$ be the Lipschitz domain defined above. Let $(x,0)$ be a $\mathcal C_\lambda$-normal point in $\mathcal C_\lambda$. Then $\lim_{r\downarrow 0} \frac{\omega(B(x,r))}{r^{d-1}} = 0$ in $(x,0)$ where $\omega$ is the harmonic measure of $\Omega_\lambda$ with pole in $(x,1)$.
	\end{propo}

	\begin{proof}
		
	Theorem 6.2. in \cite[page 149]{GM} (called $dx/\theta(x)$ estimate) tells us
	\begin{equation}
		\label{eq: Ahlfors distortion}
	\omega(B((x,0),r)) \leq \frac 8 \pi \exp \left ( -\pi \int_r^1 \frac{ds}{s\Theta(s)} \right)
	\end{equation}
	for $r$ small enough where $s\Theta(s)$ is the length of the connected arc in $\Omega \cap \partial B((x,0),s)$ that separates $(x,0)$ and the pole $(x,1)$.
	We want to prove that
	\[
	\limsup_{r\rightarrow 0} \frac{\omega(B((x,0),r))}{r} = 0
	\]
	which is a consequence of
	\[
	\limsup_{r\rightarrow0} -\pi\int_{r}^1  \left ( \frac{1}{\Theta(s)} - \frac 1 \pi \right) \frac{ds}{s} \rightarrow -\infty
	\]
	by inequality \eqref{eq: Ahlfors distortion}.
	
	Since the domain is contained in the upper half plane and $x\in\mathbb R$ we have that $0<\Theta(s) \leq \pi$. This implies that we can change the $\limsup$ by a $\lim$. Also we can rewrite the integral as
	\[
	-\pi \lim_{n\rightarrow \infty} \sum_{j=1}^n \int_{(2k+1)^{-j}}^{(2k+1)^{1-j}} \left ( \frac{1}{\Theta(s)} - \frac 1 \pi \right) \frac{ds}{s}.
	\]
	Now we will use that $x$ is $\mathcal C_\lambda$-normal. This implies that, asymptotically, for a fixed proportion of digits in the decimal expansion of $x$ in basis $2k+1$ the digit is $k-1$ or $k+1$.
	
	If the $j$-th digit in the decimal expansion of $x$ in basis $2k+1$ is either $k-1$ or $k+1$, then for $s \in [2.1(2k+1)^{-j}, (2k+1)^{1-j}]$ we have that $\Theta(s) < \pi - \epsilon$ for some $\epsilon>0$ depending only on the aperture of the cones $a$. 
	This is because the arc of circumference centered at $(x,0)$ with radius $s$ meets with a cone $C_a(p)$ corresponding to a gap in $[0,1]\backslash E_j^\lambda$ (see Remark \ref{rmk:gaps in cantor sets}) where $p$ is the truncation to $j$ digits of $x$ in basis $2k+1$. Thus, $\Theta(s) < \pi-\epsilon$.
	Since $x$ is $\mathcal C_\lambda$-normal, this happens for a positive proportion of digits, the integral tends to $-\infty$, and $\lim_{r\downarrow 0} \frac{\omega(B(x,r))}{r^{d-1}} = 0$.
	\end{proof}
	
	As a consequence of Proposition \ref{propo: upper density is 0 Hs ae} and the previous claim, we can construct a domain where $\lim_{r\downarrow 0} \frac{\omega(B(x,r))}{r^{d-1}} = 0$ in a set of dimension $s$ and $s$ can be made as close to $1$ as desired. This implies that the non-tangential derivative of a harmonic function (the Green function of the domain) can be $0$ in a set of Hausdorff dimension arbitrarily close to $1$. Moreover, we can obtain a domain where this set has Hausdorff dimension $1$ by concatenating domains of the form $\Omega_\lambda$ considered before.

\section{\texorpdfstring{$\mathcal S_{\Sigma}(u)=\mathcal S_{\Sigma}'(u)$ in the $C^{1,\operatorname{Dini}}$ case}{Equality between S(u) and S'(u) in the Dini smooth case}}
	\label{section: equality singular set in Dini dmn}

	Recall that for a harmonic function $u$ in a Lipschitz domain that vanishes in a relatively open subset of the boundary $\Sigma\subset\partial\Omega$, we define
	\[
	\mathcal S_{\Sigma}(u) = \{x \in \Sigma ~|~ |\nabla u(x)| = 0\}
	\]
	and 
	\[
	\mathcal S_{\Sigma}'(u) = \{x \in \Sigma ~|~ u^{-1}(\{0\}) \cap B(x,r)\cap\Omega \neq \varnothing, ~\forall r>0\}.
	\]
	\begin{remark}
		Note that $\mathcal S_{\Sigma}(u)$ is well defined in the $C^{1,\operatorname{Dini}}$ case since $u\in C^1({\Omega\cup\Sigma})$ thanks to the results of \cite{DEK}.
	\end{remark}
	The proof of Proposition \ref{prop: equality singular set} follows from a local expansion of $u$ as the sum of a homogeneous harmonic polynomial and an error term of higher degree (see \cite[Theorem 1.1]{KZ2}).
	\begin{proof}[Proof of Proposition \ref{prop: equality singular set}]
		By  \cite[Theorem 1.1]{KZ2}, for every $x\in\Sigma$ there exists a positive radius $R=R(x)$ and a positive integer $N=N(x)$ such that
		\[
		u(y) = P_N(y-x) + \psi(y-x), \quad\text{ in } B_R(x)\cap\Omega
		\]
		where $P_N$ is a non-trivial homogeneous harmonic polynomial of degree $N$ and the error term $\psi$ satisfies
		\[
		\lim_{y\to 0} |\psi(y)||y|^{-N} = 0 
		\quad\text{ and }\quad 
		\lim_{y\to 0} |\nabla\psi(y)||y|^{-N+1} = 0.
		\]
		
		We will show that $N(x) = 1$ implies $x\notin \mathcal S_{\Sigma}(u)$ and $x\notin \mathcal S_{\Sigma}'(u)$ and that $N(x)>1$ implies $x\in \mathcal S_{\Sigma}(u)$ and $x\in \mathcal S_{\Sigma}'(u)$.
		\vspace{2mm}
		
		\noindent\textbf{Case $N=1$:}
		We have that $\nabla u(x) = \nabla P_1(0) + \nabla \psi(0) = \nabla P_1(0) \neq 0$ since $P_1$ is a non-trivial linear function, thus $x\notin \mathcal S_{\Sigma}(u)$. Also, we know that $u(y) = 0$ for $y\in\partial\Omega\cap B_R(x)$. Suppose there exist $(z_n)_n\in\Omega\cap B_R(x)$ tending to $x$ such that $u(z_n) = 0$. Then, by Rolle's theorem, we get a contradiction since the derivative of $u$ in the direction $\nabla P_1/|\nabla P_1|$ does not vanish in a neighborhood of $x$ (because $\lim_{y \to 0} |\nabla\psi(y)| = 0$) but $u$ vanishes on $(z_n)_n$ and on $\partial\Omega\cap B_R(x)$. Thus $x \notin \mathcal S_{\Sigma}'(u)$.
		
		\vspace{2mm}
		
		\noindent\textbf{Case $N>1$:}
		Clearly, $\nabla P_N(0)=0$ if $N>1$ which implies that $x\in \mathcal S_{\Sigma}(u)$. Assume that $x\notin \mathcal S_{\Sigma}'(u)$ and, without loss of generality, $u$ is positive near $x$. Then, by the generalized Hopf principle of \cite{Saf}, we get that $\partial_\nu u(x) < 0$ contradicting $x\in \mathcal S_{\Sigma}(u)$.
	\end{proof}
	
\appendix

\section{Existence of non-tangential limits for \texorpdfstring{$\nabla u$}{nabla u}}
\label{appendix: nontangential limits}
We will closely follow the ideas of Appendix A of \cite{To} in order to prove the $L^2$ convergence of the non-tangential limits of $\nabla u$ in $\Sigma$.

Let $\Omega \subset \mathbb R^d$ be a Lipschitz domain, $B$ a ball centered on $\partial\Omega$, $\Sigma = B \cap \partial\Omega$, and $\sigma$ denote the surface measure on $\Sigma$. Without loss of generality, we assume $\Omega$ is locally above $\Sigma$ in the direction of $e_d = (0,\hdots,0,1)$.
For $\sigma$-a.e. $x\in \Sigma$, the outer unit normal vector $\nu(x)$ is well defined.
For a parameter $a \in (0,1)$ and $x \in \Sigma$, we consider the inner cone and outer cone
\[
X_a^+(x) = \{y \in \mathbb R^d  |  (x-y,\nu(x)) > a |y-x|  \}, \quad X_a^-(x) = \{y \in \mathbb R^d  |  -(x-y,\nu(x)) > a |y-x|  \},
\]
respectively.
For a function $f$ defined on $\mathbb R^d \backslash \Sigma$, we define the non-tangential limits
\[
f_{+, a}(x)=\lim _{X_{a}^{+}(x) \ni y \rightarrow x} f(y), \quad f_{-, a}(x)=\lim _{X_{a}^{-}(x) \ni y \rightarrow x} f(y),
\]
when they exist.

We prove the following theorem about the convergence of the non-tangential limits of the gradient of the solution of an elliptic PDE (of the type of Section \ref{section: conditions on A(x)}).
\begin{theorem}
	\label{thm: nontangential L2 convergence}
	Let $\Omega \subset \mathbb R^d$ be a Lipschitz domain, $B$ be an open ball centered in $\partial\Omega$, and $\Sigma = B \cap \partial\Omega$ be a Lipschitz graph such that $\Omega \cap B$ is above $\Sigma$. 
	Let $u$ be a solution of $\operatorname{div}(A\nabla u) = 0$ in $\Omega$ with $A$ as in Section \ref{section: conditions on A(x)}. Assume that $u$ is continuous up to the boundary and that it vanishes continuously on $\Sigma$. 
	Then, for all $a \in (0,1)$ large enough, $(\nabla u)_{+,a}$ exists $\sigma$-a.e. and belongs to $L^2_{\operatorname{loc}}(\sigma)$.
	Further, $(\nabla u)_{+,a}$ has {vanishing tangential component}. That is, $(\nabla u)_{+,a} = (\nabla u, \nu)\nu$. Further,
	\[
	\lim _{\varepsilon \rightarrow 0} \nabla u\left(\cdot+\varepsilon e_{n}\right) \rightarrow\left(\nabla u, \nu \right) \nu \quad \text { in } L_{\operatorname{loc}}^{2}(\sigma).
	\]
\end{theorem}

\begin{proof}
	We extend $u$ by $0$ out of $\Omega$ and denote $u^+=\max(u,0)$, $u^-=\max(-u,0)$. Both $u^+$ and $u^-$ are continuous and subsolutions (that is $-\operatorname{div}(A\nabla u^\pm)\leq 0$ in $B$).
	\begin{claim}
		 In the sense of distributions, $\operatorname{div}(A\nabla u) = \operatorname{div}(A\nabla u^+) - \operatorname{div}(A\nabla u^-)$ restricted to $B$ is a signed Radon measure supported on $\Sigma$.
	\end{claim}
	\begin{proof}[Proof of the claim]
	To see this, take arbitrary $\phi \in C_c^\infty (B)$. For $0<\epsilon\ll r(B)$, denote $\Sigma_\epsilon = \Sigma + \epsilon e_d$
	and $\Omega_\epsilon = \Omega + \epsilon e_d$, where $e_d=(0,\ldots,0,1)$. 
	Let  $\psi_\delta$ be a bump function in an $\delta$-neighborhood of $\Sigma$ such that $\Vert\nabla \psi_\delta \Vert_\infty \lesssim \delta^{-1}$ and $\Vert \nabla^2 \psi_\delta \Vert_\infty \lesssim \delta^{-2}$. 
	Then, writing $\phi = \phi \psi_\delta + \phi(1-\psi_\delta)$ for $\phi \in C_c^\infty(B)$, we get
	\begin{align*}
		\left | \int_{B\cap\Omega} A \nabla u \nabla \phi dx  \right | 
		\leq
		\underbrace{\left | \int_{B\cap\Omega} A \nabla u \nabla (\phi\psi_\delta) dx  \right |}_{\boxed A} +	\underbrace{\left | \int_{B\cap\Omega} A \nabla u \nabla (\phi(1-\psi_\delta)) dx  \right |}_{\boxed B}. 
	\end{align*}
	We directly obtain $\boxed B=0$ since $u$ is a weak solution of the PDE in $\Omega$.
	We have, because $u \in W^{1,2}(B)$ and the divergence theorem, that
	\begin{align*}
		\int_{B\cap\Omega} A \nabla u \nabla (\phi\psi_\delta) dx
		&=
		\lim_{\epsilon\to0}\int_{B\cap\Omega_\epsilon} A \nabla u \nabla (\phi\psi_\delta) dx  \\
		&= 
		\lim_{\epsilon\to0}\underbrace{\int_{\partial\Omega_\epsilon \cap B} u (A\nabla (\phi\psi_\delta), \nu) d\sigma}_{\boxed C}
		- 
		\underbrace{\int_{\Omega_\epsilon\cap B} u \operatorname{div}(A\nabla(\phi\psi_\delta))dx}_{\boxed D}
	\end{align*}
	The term \boxed C converges uniformly to $0$ as $\epsilon\to0$ as $u$ is $0$ in $\Sigma$. On the other hand,
	\begin{align*}
		|\operatorname{div}(A\nabla(\phi\psi_\delta))| 
		\leq&
		|\operatorname{div}(A\nabla\psi_\delta)\phi|  + |\operatorname{div}(A\nabla\phi)\psi_\delta| + 
		|2(A\nabla\phi,\nabla\psi_\delta)|  \\
		\lesssim& \frac{1}{\delta^2} \Vert \phi\Vert_{L^\infty(\operatorname{supp} \psi_\delta)} + \Vert \nabla^2\phi\Vert_\infty + \frac{1}{\delta} \Vert \nabla \phi \Vert_\infty 
	\end{align*}
	with constants depending on the Lipschitz and ellipticity constants of $A$.
	
	We can bound \boxed D as
	\begin{equation}\label{eq: divAnablau is a Radon measure}
		\left|\int_{\Omega \cap B} u \operatorname{div}(A\nabla(\phi\psi_\delta))dx\right|
		\lesssim
		\left ( \frac{1}{\delta^2} \Vert \phi\Vert_{L^\infty(\operatorname{supp} \psi_\delta)} + \Vert \nabla^2\phi\Vert_\infty + \frac{1}{\delta} \Vert \nabla \phi \Vert_\infty  \right) 
		\int_{\operatorname{supp} \psi_\delta \cap \Omega} |u| dx
		.
	\end{equation}
	Using the boundary Harnack inequality (see \cite{DS}, for example), we can bound
	\[
	\int_{\operatorname{supp} \psi_\delta \cap \Omega} |u| dx 
	\lesssim
	\int_{\operatorname{supp} \psi_\delta \cap \Omega} g \, dx \lesssim \delta^2
	\]
	where $g$ is the Green function with fixed pole far from $B$. To prove the second inequality, we cover $\operatorname{supp} \psi_\delta$ by a finite family of cubes $Q_i$ with side length $\ell(Q) \approx \delta$. We can do this with approximately $\left( \frac{r(B)}{\delta}\right)^{d-1}$ cubes.
	By {standard estimates for elliptic measure}, we have
	\[
	g(x)\lesssim \frac{\omega(4Q)}{\ell(Q)^{n-2}} \approx \omega(4Q) \delta^{2-d} \quad \mbox{ for all $x\in Q\cap \Omega$.}
	\]
	where $\omega$ is the elliptic measure associated to $\operatorname{div}(A\nabla \cdot)$ for $\Omega$ with respect to a fixed pole $p\in\Omega\backslash B$. 
	
	Finally, we have 
	\[
	\int_{\operatorname{supp} \psi_\delta\cap\Omega} g \,dx \leq
	\sum_{Q} \int_{Q\cap\Omega} g \, dx 
	\lesssim
	\sum_Q  \omega(4Q) \delta^{2-d} \int_Q dx\ \approx \delta^2 \omega(\Sigma)
	\]
	where we have used the doubling properties of elliptic measure.
	\color{black}

	Summing up, we take the limit as $\delta\to0$ in equation \eqref{eq: divAnablau is a Radon measure} to obtain
	\[
	((\operatorname{div}(A\nabla u)),\phi) \leq C \Vert \phi \Vert_{L^\infty(\Sigma)} 
	\]	
	where we have used that the $\operatorname{supp} \psi_\delta \to \Sigma$ as $\delta\to0$,
	Thus $\operatorname{div}(A\nabla u)|_B$ is a Radon measure supported on $\Sigma$.
	\end{proof}
	Moreover, this Radon measure is absolutely continuous with respect to elliptic measure.
	\begin{claim}
		In the sense of distributions, we have
		\begin{equation*}
			\label{eq: claim appendix}
			\operatorname{div}(A \nabla u)|_B = \rho \omega|_\Sigma,
		\end{equation*}
		where $\rho\in L^\infty_{loc}(\Sigma)$ and $\omega$ is the elliptic measure associated to $\operatorname{div}(A\nabla \cdot)$ for $\Omega$ with respect to a fixed pole $p\in\Omega\backslash B$.
	\end{claim}
	We can assume that $B$ is small enough 
	so that $\Omega\backslash 2B\neq\varnothing$ and $p\in\Omega\setminus 2B$.
	\begin{proof}[Proof of the claim]
		To prove the claim, let $B'$ be an open ball concentric with $B$ such that 
		$\overline {B'}\subset B$. We will show that that there exists some constant $C$ depending on $B'$ and $p$ such that for any compact set $K\subset \Sigma \cap B'$, it holds
		\begin{equation}
			\label{eq: divAnabla u bdd by elliptic measure}
			\left (\operatorname{div}(A\nabla u), \chi_K \right)\leq C \omega(K).
		\end{equation}
		By duality, this implies the claim.
		
		Given $\epsilon \in \big(0,\frac12\operatorname{dist}(K,\mathbb R^d\backslash B')\big)$,
		let $\{Q_i\}_{i\in I}$ be a lattice of cubes covering $\mathbb R^d$ such that each $Q_i$ has diameter $\epsilon/2$. Let 
		$\{\phi_i\}_{i\in I}$ be a partition of unity of $\mathbb R^n$, so that each $\phi_i$ is supported
		in $2Q_i$ and satisfies $\|\nabla^j\phi_i\|_\infty\lesssim \ell(Q_i)^{-j}$, for $j=0,1,2$.
		Then, we have
		\begin{equation*}
			( \operatorname{div} (A\nabla u),\chi_K) = ( \operatorname{div}(A\nabla u), \sum_{i\in I'} \phi_i) - 
			( \operatorname{div}(A\nabla u), \sum_{i\in I'} \phi_i -\chi_K),
		\end{equation*}
		where $I'$ is the collection of indices $i\in I$ that satisfy $2Q_i\cap K\neq\varnothing$.
		By the regularity properties of Radon measures, we obtain
		\[
		| ( \operatorname{div}(A\nabla u), \sum_{i\in I'} \phi_i -\chi_K)|
		\leq \left ( |\operatorname{div}(A\nabla u)|,  \chi_{U_\epsilon(K)\setminus K}\right)\to 0\quad
		\mbox{ as $\epsilon\to 0$,}
		\]
		where $U_\epsilon(K)$ is the $\epsilon$-neighborhood of $K$.
		For the other term, we have
		\[
		|( \operatorname{div}(A\nabla u), \sum_{i\in I'} \phi_i)| \leq 
		\sum_{i\in I'}
		\left|\left(  u,  \operatorname{div}(A\nabla\phi_i)\right)\right| \lesssim \sum_{i\in I'}\frac1{\ell(Q_i)^2} \int_{2Q_i}
		|u| dm
		\]
		using that $A$ is elliptic and Lipschitz.
		Since $|u|$ is a continuous subsolution in $B$ that vanishes in $B\backslash\Omega$, by the
		boundary Harnack inequality, we have again $|u(x)| \lesssim  g(x)$ for all $x\in B'\cap\Omega$,
		where $g$ is the Green function of $\Omega$ for $\operatorname{div}(A \nabla \cdot)$ with pole in $\Omega \backslash B$. The constant $C$ depends on $u$, $p$, $\Lambda_A$ (the ellipticity constant of $A$), and $B'$, but not on $K$.
		Thus, proceeding as in the proof of the previous claim, we obtain
		\[
		|( \operatorname{div}(A\nabla u), \sum_{i\in I'} \phi_i)| \lesssim \sum_{i\in I'}
		\frac1{\ell(Q_i)^2} \int_{2Q_i} g(x) dx \approx \sum_{i \in I'} \omega (4Q_i) = \omega(U_{4\epsilon(K)}) .
		\]
		Letting $\epsilon\to 0$, we have $\omega(U_{4\epsilon}(K))\to\omega(K)$ from which equation \eqref{eq: divAnabla u bdd by elliptic measure} follows.
	\end{proof}
	
	By the solvability of the Dirichlet problem with $L^2$ data in Lipschitz domains for divergence form elliptic equations with H\"older coeficients (see Remark 1.4 in \cite{KS}), 
	we have that $\omega$ is a $B_2$ weight with respect to the surface measure $\sigma$. In particular, 
	the density function $\frac{d\omega}{d\sigma}$ belongs to $L^2_{loc}(\sigma)$. Therefore,
	in the sense of distributions,
	\[
	\operatorname{div}(A\nabla u)|_B = h \sigma,\quad \mbox{ for some $h\in L^2_{loc}(\sigma)$.}
	\]
	
	Next, we will show that
	$(\nabla u)_{+,a}$ exists $\sigma$-a.e.\ and moreover $(\nabla u)_{+,a} = (\nabla u,\nu) \nu \in L^2_{loc}(\sigma)$. Consider an arbitrary open ball $\tilde B$ centered in $\Sigma$ such that
	$4\tilde B\subset B$. Let $\phi$ be a $C^\infty$ function which equals $1$ on $2\tilde B$ is supported on $3\tilde B$, and let $v=\phi  u$. Observe that
	\begin{align}
		\label{eq: u given by the fund sol}
		v(x) &= \int \mathcal E(x,y) \operatorname{div}(A\nabla v)(y) \\
		&= 
		\int \mathcal E(x,y) \phi(y)\operatorname{div}(A\nabla u)(y) +
		 \int \mathcal E(x,y) u(y)\operatorname{div}(A\nabla \phi)(y) +
		 2\int \mathcal E(x,y) (A\nabla u(y), \nabla \phi(y)), \nonumber
	\end{align}
	where $\mathcal E$ is the fundamental solution of $\operatorname{div}(A\nabla\cdot)$ (we consider a Lipschitz, elliptic extension of $A$ defined in $\mathbb R^d$). Note that 
	$\nabla u\in L^2_{\operatorname{loc}}(B)$, by Caccioppoli's inequality.
	
	For a finite Borel measure $\eta$, let $T\eta$ be the gradient of the single layer potential of $\eta$ (in the case $A \equiv I$ it coincides with the $(d-1)$-dimensional Riesz transform of $\eta$). That is,
	\[
	T\eta(x) = \int \nabla_1 \mathcal E(x,y) d\eta(y),
	\]
	in the sense of truncations.
	From identity \eqref{eq: u given by the fund sol}, we obtain for $x\not \in\Sigma$,
	\[
	\nabla v(x) = c_d \big(T(\phi h \sigma|_\Sigma)(x) + T(u \operatorname{div}(A\nabla\phi) m)(x) + 2 T((A\nabla u, \nabla \phi)  m\big)(x)\big)
	\]
	(where $m$ is the Lebesgue measure in $\mathbb R^d$). Observe that 
	$T(u\operatorname{div}(A\nabla\phi) m)(x)$ and 
	$T((A\nabla u, \nabla \phi)  m\big)(x)$ are continuous functions in $\tilde B$.
	On the other hand, the non-tangential limit $T(\phi h \sigma|_\Sigma)_{\pm,a}(x)$ exists for $\sigma$-a.e.\ $x\in\partial\Omega$, by the jump formulas for the gradients of single layer potentials transforms
	(see \cite[Theorem 4.4]{KS} or \cite[Theorem 1]{ToJumps} for the constant coefficients case in rectifiable sets, for example). 
	From the fact that $\nabla v=\nabla u$ in $\tilde B$, it follows that
	$(\nabla u)_{\pm,a}(x)$  exists for $\sigma$-a.e.\ $x\in\Sigma\cap \tilde B$. By the $L^2(\sigma)$
	boundedness of $T$ on Lipschitz graphs (see  \cite[Theorem 3.1]{KS}), we deduce that 
	$(\nabla u)_{\pm,a}\in L^2(\sigma|_{\Sigma\cap\tilde B})$. 
	
	Since $u\equiv0$ in
	$\Omega^c$, we have $(\nabla u)_{-,a}=0$ in $\Sigma\cap\tilde B$. As the tangential component of $T(\phi h \sigma|_\Sigma)(x)$ is continuous across $\partial\Omega$ for $\sigma$-a.e.\ $x\in\partial\Omega$, again by the jump formulas for $T$ (\cite[Theorem 4.4]{KS}), we deduce that the tangential component of $(\nabla u)_{+,a}$ coincides with
	the tangential component of $(\nabla u)_{-,a}$ $\sigma$-a.e.\ in $\Sigma\cap\tilde B$, and thus
	$(\nabla u)_{+,a}
	= (\nabla u, \nu) \nu$ in $\Sigma\cap\tilde B$.
	
	We will now prove that $(\operatorname{div}(A\nabla u))|_B = -(A \nabla u, \nu)  \sigma|_\Sigma$ in the sense of distributions. Let $\psi$ be a function in $C^\infty_c(\tilde B)$ .
	Again, for $0<\epsilon\ll r(\tilde B)$, consider $\Sigma_\epsilon = \Sigma + \epsilon e_d$
	and $\Omega_\epsilon = \Omega + \epsilon e_d$,
	where $e_d=(0,\ldots,0,1)$. Then, we have
	\begin{align}\label{eqrt943}
		\langle \operatorname{div}(A\nabla u), \psi\rangle &= \int_{\tilde B \cap \Omega_\epsilon} u \operatorname{div}(A\nabla \psi) dm  = \lim_{\epsilon\to 0} \int_{\tilde B \cap\Omega_\epsilon} u \operatorname{div}(A\nabla \psi) dm \\
		& = \lim_{\epsilon\to 0} \int_{\tilde B \cap\partial \Omega_\epsilon} u (A\nabla\psi, \nu) d\sigma
		- \lim_{\epsilon\to 0} \int_{\tilde B \cap\partial \Omega_\epsilon} \psi (A\nabla u, \nu) d\sigma\nonumber\\
		& = 0 -  \int_{\Sigma} \psi (A\nabla u, \nu) d\sigma.\nonumber
	\end{align}
	The last identity follows from the uniform convergence of $u$ to $0$ as $\Sigma_\epsilon \to \Sigma$ and that $\nabla u(\cdot  + \epsilon e_n)$ converges  to $(\nabla u)_{+,a}$
	in $L^2(\sigma|_{\Sigma\cap \tilde B})$ (this is proven by arguments analogous to the ones above for the $\sigma$-a.e.\ existence of the limit $(\nabla u)_{+,a}(x)$ in $\Sigma$). From \eqref{eqrt943}, we deduce that $\operatorname{div}(A\nabla u) = -(A\nabla u,\nu)  \sigma$ in $\tilde B$, and thus also in $B$.
\end{proof}

\end{document}